\documentclass[leqno,english]{amsproc} 
\usepackage{hyperref}
\usepackage{amssymb}
\usepackage[square,numbers]{natbib}
\usepackage[all]{xy}
\usepackage{ifthen}

\renewcommand{\thesection}{\arabic{section}}
\renewcommand{\thesubsection}{\thesection.\arabic{subsection}}
\renewcommand{\thesubsubsection}{\thesection.\arabic{subsubsection}}
\numberwithin{subsubsection}{section}

\renewcommand{\theequation}{*}

\theoremstyle{plain}

\swapnumbers
\newtheorem{thm}[subsubsection]{Theorem}
\newtheorem{lemm}[subsubsection]{Lemma}
\newtheorem{prop}[subsubsection]{Proposition}
\newtheorem*{claim}{Claim}
\newtheorem*{fact}{Fact}
\newtheorem{obsv}[subsubsection]{Observation}

\DeclareMathOperator{\Mor}{Mor}
\DeclareMathOperator{\End}{\mathit{End}}
\DeclareMathOperator{\Hom}{\mathit{Hom}}
\DeclareMathOperator{\Tw}{Tw}

\DeclareMathOperator{\NN}{\mathbb{N}}
\DeclareMathOperator{\ZZ}{\mathbb{Z}}
\DeclareMathOperator{\QQ}{\mathbb{Q}}
\DeclareMathOperator{\kk}{\Bbbk}

\DeclareMathOperator{\AOp}{\mathsf{A}}

\DeclareMathOperator{\COp}{\mathsf{C}}
\DeclareMathOperator{\DOp}{\mathsf{D}}

\DeclareMathOperator{\FOp}{\mathsf{F}}

\DeclareMathOperator{\IOp}{\mathsf{I}}
\DeclareMathOperator{\LOp}{\mathsf{L}}
\DeclareMathOperator{\POp}{\mathsf{P}}
\DeclareMathOperator{\QOp}{\mathsf{Q}}
\DeclareMathOperator{\ROp}{\mathsf{R}}

\DeclareMathOperator{\A}{\mathcal{A}}
\DeclareMathOperator{\C}{\mathcal{C}}
\DeclareMathOperator{\F}{\mathcal{F}}
\DeclareMathOperator{\I}{\mathcal{I}}
\DeclareMathOperator{\J}{\mathcal{J}}
\DeclareMathOperator{\K}{\mathcal{K}}
\DeclareMathOperator{\M}{\mathcal{M}}
\DeclareMathOperator{\Op}{\mathcal{O}}
\DeclareMathOperator{\X}{\mathcal{X}}

\DeclareMathOperator{\Bij}{\mathcal{B}\mathit{ij}}
\DeclareMathOperator{\Lev}{\Xi}

\DeclareMathOperator{\Ho}{Ho}

\DeclareMathOperator{\Cyl}{\mathit{Cyl}}

\DeclareMathOperator{\Id}{Id}
\DeclareMathOperator{\id}{id}
\DeclareMathOperator*{\colim}{colim}

\title[Operadic Cobar Construction and Homotopy Morphisms]{Operadic Cobar Constructions,\\
Cylinder Objects\\
and\\Homotopy Morphisms of Algebras over Operads}
\author{Benoit Fresse}
\date{1 February, 2009 (revised 17 June, 2009)}
\address{UMR 8524 de l'Universit\'e des Sciences et Technologies de Lille et du CNRS\\
Cit\'e Scientifique -- B\^atiment M2\\
F-59655 Villeneuve d'Ascq C\'edex (France)}
\email{Benoit.Fresse@math.univ-lille1.fr}
\urladdr{http://math.univ-lille1.fr/\~{ }fresse}
\thanks{Research supported in part by ANR grant JCJC06 OBTH}

\begin{document}

\begin{abstract}
The purpose of this paper is twofold.
First,
we review applications of the bar duality of operads to the construction of explicit cofibrant replacements
in categories of algebras over an operad.
In view toward applications,
we check that the constructions of the bar duality work properly for algebras over operads
in unbounded differential graded modules over a ring.

In a second part,
we use the operadic cobar construction to define explicit cyclinder objects in the category of operads.
Then
we apply this construction to prove that certain homotopy morphisms of algebras over operads
are equivalent to left homotopies
in the model category of operads.
\end{abstract}

\maketitle


\section*{Contents}{\parindent=0cm

\medskip\textbf{Introduction}

\medskip\textbf{The homotopy of operads and of algebras over operads}

{\leftskip=1cm

\textit{\S\ref{Background}. The background of dg-modules}

\textit{\S\ref{Operads}. Operads in dg-modules}

\textit{\S\ref{OperadAlgebras}. The category of algebras associated to an operad}

}

\medskip\textbf{Applications of the operadic cobar construction}

{\leftskip=1cm

\textit{\S\ref{CobarConstruction}. The operadic cobar construction}

\textit{\S\ref{CoalgebraModels}. The cobar construction and cofibrant models of algebras over operads}

\textit{\S\ref{CylinderHomotopyMorphisms}. Cylinder objects and homotopy morphisms}

}

\medskip\textbf{Bibliography}

}

\medskip

\part*{Introduction}

The notion of an operad makes sense in the general setting of symmetric monoidal categories.
To do homotopy theory,
we consider operads in a symmetric monoidal category
equipped with a model structure.
In good cases,
the category of algebras associated to an operad (and the category of operads itself) inherits a model structure
with as weak-equivalences the morphisms
which form weak-equivalences
in the underlying symmetric monoidal model category\footnote{If we take mild assumptions,
then the axioms of model categories are not fully satisfied by the category of algebras associated to an operad,
but we still have a semi-model structure in the sense of~\cite{HoveySemiModel}
which is sufficient for most applications of homotopical algebra.
Therefore,
we use abusively the word ``model structure'' to refer to the structure of a semi-model category,
which is taken as a good background for the applications of homotopical algebra
to categories of algebras over an operad.}.

The homotopy category of a model category can be defined naively as a fraction category
in which the weak-equivalences
are formally inverted to yield actual isomorphisms.
The existence of a model structure gives a manageable representation
of this category:
each object can be replaced by an equivalent cofibrant-fibrant model
and the morphisms of the homotopy category
are represented by equivalence classes
of actual morphisms between these cofibrant-fibrant replacements.

\medskip
The first purpose of this paper is to review a definition of explicit cofibrant replacements
for algebras over operads in differential graded modules (for short, we say dg-modules).
These cofibrant replacements are determined from the structure of operadic cobar constructions,
which themselves define cofibrant objects of a particular form in the category of operads.

In many references of the literature on the operadic cobar construction,
authors deal with operads in $\NN$-graded dg-modules
defined over a field of characteristic zero.
The bar duality of operads has been introduced in~\cite{GetzlerJones,GinzburgKapranov}
in that setting.

But certain applications of operads hold in the context of modules over a ring
and require the use of unbounded chain complexes (see for instance~\cite{HinichSchechtman,KrizMay,Mandell}).
For that reason,
we check carefully that usual applications of the operadic cobar construction
can be generalized to the setting of $\ZZ$-graded dg-modules over a ring.
To summarize,
we explain how to address two sources of technical difficulties:
\begin{itemize}
\item
the existence of $\ZZ$-torsion in the ground ring or in modules over the ground ring,
\item
the structure of cofibrant dg-modules and the convergence of spectral sequences in the $\ZZ$-graded context.
\end{itemize}

In a previous paper~\cite{FresseKoszulDuality},
we study the homotopy of cobar constructions within the category of operads.
In this paper,
we focus on applications of operadic cobar constructions to the homotopy categories of algebras over operads.

Recall that the operadic cobar construction is an operad $B^c(\DOp)$
naturally associated to a cooperad $\DOp$.
Any operad $\POp$
has a cofibrant replacement of the form $\QOp = B^c(\DOp)$
for some cooperad $\DOp$.
The usual bar duality associates a $\DOp$-coalgebra $\Gamma = \Gamma_{\POp}(A)$ to any $\POp$-algebra $A$
and a $\POp$-algebra $R_{\POp}(\Gamma)$
to any $\DOp$-coalgebra $\Gamma$.
The classical duality of rational homotopy between cocommutative coalgebras and Lie coalgebras
can be viewed as an instance of this construction
where we take a desuspension of the cooperad of cocommutative coalgebras $\DOp = \Lambda^{-1}\COp^{\vee}$
and the operad of Lie algebras $\POp = \LOp$.
The bar duality of associative algebras
can be identified with another application of this construction
where we take the cooperad of coassociative coalgebras $\DOp = \AOp^{\vee}$
and the operad of associative algebras $\POp = \AOp$.

We review the construction of the coalgebra $\Gamma = \Gamma_{\POp}(A)$, of the $\POp$-algebra $R_{\POp}(\Gamma)$,
and we check that the composite construction $R_A = R_{\POp}(\Gamma_{\POp}(A))$ defines a cofibrant replacement of~$A$
in the category of $\POp$-algebras
under mild assumptions on the unbounded dg-module underlying $A$
when the ground ring is not a field.
We also study morphisms between cofibrant $\POp$-algebras of this form $R_A = R_{\POp}(\Gamma_{\POp}(A))$.

Recall that the structure of a $\POp$-algebra $A$
is determined by an operad morphism~$\phi: \POp\rightarrow\End_A$,
where $\End_A$ is a certain universal operad acting on~$A$,
the endomorphism operad of~$A$.
In a second part of the paper,
we study homotopy morphisms between pairs of $\POp$-algebras
with the same underlying dg-module.
We prove that these homotopy morphisms
are equivalent to left homotopies between the morphisms $\phi^0,\phi^1: \POp\rightarrow\End_A$
which determine the $\POp$-algebra structures.
For this aim,
we use an extension of the cobar construction
to define explicit cylinder objects in the category of operads.
In the example of the Lie operad $\POp = \LOp$,
we obtain that the $L_{\infty}$-morphisms which reduce to the identity on objects
are equivalent to left-homotopies in the category of operads.
In the example of the associative operad $\POp = \AOp$,
we obtain the same result for the standard notion of an $A_{\infty}$-morphism.

In~\cite{FressePropHomotopy},
we prove that the existence of a left homotopy in the category of operads implies the existence of a morphism
in the homotopy category of algebras.
The construction of this paper gives a more precise result
in the sense that we have a characterization of the homotopy morphisms
associated to left homotopies in the category of operads.

\medskip
In this paper,
we use the standard homotopy category of the theory of model categories.
But another notion of homotopy morphism for algebras over operads
has been introduced by Boardman-Vogt in~\cite{BoardmanVogt}.
These homotopy morphisms are associated to cofibrant replacements of colored operads
which model morphisms of algebras over operads.
We plan to prove in a follow up that the standard homotopy category of algebras over an operad
is equivalent to the homotopy category of Boardman-Vogt' homotopy morphisms.

This paper should serve as a preparation
for this future work,
because the correspondence between operad homotopies and homotopy morphisms of algebras over operads
is a part of this equivalence.

\medskip
In the first part of the paper, 
we review the overall definitions of the homotopy theory of operads in dg-modules:
we devote a preliminary section to setting up the structure of the base category of dg-modules~(\S\ref{Background});
we recall the definition of an operad
and we review the definition of the model structure of the category of operads next~(\S\ref{Operads});
lastly, we recall the definition of the categories of algebras associated to operads
and their model structures~(\S\ref{OperadAlgebras}).

In the second part of the paper, 
we address applications of the operadic cobar construction
to the homotopy of algebras over operads.
First of all, in~\S\ref{CobarConstruction},
we review the definition of the cobar construction of a cooperad.
The applications of the cobar construction to the definition of cofibrant replacements of algebras over operads
are addressed in~\S\ref{CoalgebraModels}.
The correspondence between homotopy morphisms of algebras over operads
and left homotopies of operad morphisms
is established in~\S\ref{CylinderHomotopyMorphisms}.

\part*{The homotopy of operads and of algebras over operads}

\section{The background of dg-modules}\label{Background}
\renewcommand{\thesubsubsection}{\thesection.\arabic{subsubsection}}

The first purpose of this section is to fix conventions on the category of dg-modules, which give the background of our constructions.
Then we review the definition of the symmetric monoidal model structure of that category.

\subsubsection{The category of dg-modules}\label{Background:DGModules}
For us,
a dg-module $C$ refers to a module over a fixed ground ring $\kk$
equipped with a $\ZZ$-grading $C = \bigoplus_{*\in\ZZ} C_*$
and a differential $\delta: C\rightarrow C$
which decreases degree by $1$.
The degree of a homogeneous element $x\in C$
is denoted by $\deg(x)$.
Usually,
we refer to a dg-module by the notation of the underlying module $C$
which is supposed to be equipped with a natural differential
for which we use the notation~$\delta$.

The category of dg-modules is denoted by the letter $\C$.
Naturally,
a morphism of dg-modules is a morphism of $\kk$-modules $f: C\rightarrow D$
which preserves gradings and commutes with differentials.
The notation $\Mor_{\C}(C,D)$
refers to the set of dg-module morphisms $f: C\rightarrow D$.

In the paper,
we also use dg-module homomorphisms
which are morphisms of $\kk$-modules $f: C\rightarrow D$
such that $f(C_*)\subset D_{*+d}$
for a fixed homogeneity degree $d = \deg(f)$.
The notation $\Hom_{\C}(C,D)$ refers to the $\kk$-module
spanned by homogeneous dg-module homomorphisms.
The object $\Hom_{\C}(C,D)$ forms itself a dg-module
because we have a natural differential $\delta: \Hom_{\C}(C,D)\rightarrow\Hom_{\C}(C,D)$
defined for each $f\in\Hom_{\C}(C,D)$
by the graded commutator $\delta(f) = \delta f - (-1)^{\deg(f)} f\delta$,
where $\delta$ refer to the internal differential of~$C$ and~$D$.

\subsubsection{Twisted dg-modules}\label{Background:TwistedObjects}
In certain constructions,
a twisting homomorphism of degree $-1$
\begin{equation*}
\partial\in\Hom_{\C}(C,C)
\end{equation*}
is added to the internal differential of a dg-module $\delta: C\rightarrow C$
to produce a new dg-module structure.
This new dg-module, to which we refer by the pair $(C,\partial)$,
has the same underlying graded module as $C$
but a differential defined by the sum $\delta+\partial: C\rightarrow C$.
The identity of differentials $(\delta+\partial)^2 = 0$
is equivalent to the equation
\begin{equation}
\delta(\partial) + \partial^2 = 0
\end{equation}
in the dg-hom $\Hom_{\C}(C,C)$.
In the sequel,
we refer to (*) as the equation of twisting homomorphisms.

\subsubsection{The symmetric monoidal category of dg-modules}\label{Background:DGModuleTensorProduct}
The tensor product of dg-modules $C,D\in\C$
is the dg-module $C\otimes D\in\C$
such that
\begin{equation*}
(C\otimes D)_n = \bigoplus_{p+q = n} C_p\otimes D_q
\end{equation*}
together with the differential $\delta: C\otimes D\rightarrow C\otimes D$
defined by the formula
\begin{equation*}
\delta(x\otimes y) = \delta(x)\otimes y + (-1)^{\deg(x)} x\otimes\delta(y),
\end{equation*}
for any homogeneous tensor $x\otimes y\in C\otimes D$.

This tensor product $\otimes: \C\times\C\rightarrow\C$
provides the category of dg-modules $\C$
with the structure of a symmetric monoidal category
since:
\begin{itemize}
\item
the ground ring~$\kk$, identified with a dg-module concentrated in degree~$0$,
defines a unit for the tensor product of dg-modules;
\item
the tensor product of dg-modules is obviously associative;
\item
we have a symmetry isomorphism $\tau: C\otimes D\rightarrow D\otimes C$
defined by
\begin{equation*}
\tau(x\otimes y) = (-1)^{\deg(x)\cdot\deg(y)} y\otimes x,
\end{equation*}
for any homogeneous tensor $x\otimes y\in C\otimes D$.
\end{itemize}
The isomorphisms that give the unit, associativity and symmetry relations of a symmetric monoidal structure
are required to satisfy natural coherence axioms.
The verification of these axioms is easy in the context of dg-modules.

The dg-modules of homomorphisms $\Hom_{\C}(C,D)$
fit the adjunction relation
\begin{equation*}
\Mor_{\C}(K\otimes C,D)\simeq\Mor_{\C}(K,\Hom_{\C}(C,D))
\end{equation*}
with respect to the tensor product $\otimes: \C\times\C\rightarrow\C$.
Accordingly,
the category of dg-modules forms a closed symmetric monoidal category
with the dg-modules $\Hom_{\C}(C,D)$ as internal hom-objects.

\subsubsection{Signs}\label{Background:Signs}
In our work
we adopt the notation $\pm$ to refer to a sign which arises from a tensor permutation.
In principle
we do not make these signs explicit
since they are fully determined by the definition of the symmetry isomorphism of dg-modules.
Globally,
any permutation of dg-symbols $x y\mapsto y x$ produces a sign $\pm = (-1)^{\deg(x)\cdot\deg(y)}$.
The notion of a dg-symbol includes not only the elements of a dg-module $C$,
but also the dg-module homomorphisms $f: C\rightarrow D$,
which are nothing but homogeneous elements of the dg-hom $\Hom_{\C}(C,D)$,
and the differentials $\delta: C\rightarrow C$,
which are nothing but homomorphisms of degree $-1$.

\subsubsection{Model categories}\label{Background:ModelCategories}
Recall briefly that a model structure on a category $\A$
is defined by three classes of morphisms,
called weak-equivalences, cofibrations, and fibrations,
satisfying axioms modeled on properties of the usual weak-equivalences, cofibrations, and fibrations of topology.
The purpose of this structure is to handle the morphisms of the homotopy category of~$\A$,
usually denoted by $\Ho\A$.
In the general setting of a category with weak-equivalences,
the homotopy category can only be defined naively as a localization $\Ho\A = \A[W^{-1}]$
with respect to the class of weak-equivalences $W\subset\Mor\A$.
In the context of a model category,
we have a natural homotopy relation on morphisms
and the morphisms of the homotopy category are represented by homotopy classes of morphisms of~$\A$.
The notion of a cofibration and of a fibration is used to characterize the morphisms and objects
which behave properly with respect to this homotopy relation.

The reader is supposed to be familiar with this background,
for which we can refer to~\cite{Hirschhorn,Hovey}.
We only recall the definition of the model structures used in the paper:
the model structure of the category of dg-modules below,
the model structure of the category of operads in dg-modules in~\S\ref{OperadHomotopy}
and the model structure of the categories of algebras over an operad in~\S\ref{OperadAlgebras}.

\begin{fact}[{see for instance~\cite[Theorem 2.3.11]{Hovey}}]\label{Background:DGModuleModelStructure}
The category of dg-modules is equipped with a model structure
such that:
\begin{itemize}
\item
the weak-equivalences are the morphisms which induce an isomorphism in homology;
\item
the fibrations are the morphisms of dg-modules which are degreewise surjective;
\item
the cofibrations are the morphisms which have the right lifting property with respect to acyclic fibrations.
\end{itemize}
\end{fact}

This model structure is symmetric monoidal in a strong sense:

\begin{fact}[{see~\cite[Proposition 4.2.13]{Hovey}}]\hspace*{2mm}
\begin{itemize}
\item
The ground ring $\kk$, which defines the unit of the tensor product of~$\C$,
forms a cofibrant object in~$\C$.
\item
The pushout-product axiom (see~\cite[\S 4.2.1]{Hovey})
holds for the tensor product of dg-modules
and we have a dual pullback-hom axiom for the internal hom of the category of dg-modules.
\end{itemize}
\end{fact}

\subsubsection{Relative cell complexes in cofibrantly generated model categories}\label{Background:CellComplexes}
Let $\K$ be a set of maps in a category~$\A$.
Recall briefly that a relative $\K$-cell complex in $\A$
is a (possibly transfinite) composite
\begin{equation*}
A = A_0\rightarrow A_1\rightarrow\dots
\rightarrow A_{\lambda-1}\xrightarrow{j_\lambda} A_{\lambda}\rightarrow\dots
\rightarrow\colim_{\lambda} A_\lambda = B
\end{equation*}
such that each $j_\lambda$ is defined by a pushout
\begin{equation*}
\xymatrix{ \displaystyle\bigvee_{\alpha} C_\alpha\ar[d]_{\sum_\alpha i_\alpha}\ar[r]^f & A_{\lambda-1}\ar@{.>}[d]^{j_\lambda} \\
\displaystyle\bigvee_{\alpha} D_\alpha\ar@{.>}[r]_g & A_{\lambda} }
\end{equation*}
with $i_\alpha\in\K$, $\forall\alpha$.

The usual model categories are equipped with:
\begin{itemize}
\item
a set of generating cofibrations $\I\subset\Mor\A$
which serve to produce factorizations $f = p i$
such that $p$ is an acyclic fibration and $i$ is a relative $\I$-cell complex,
\item
and a set of generating acyclic cofibrations $\J\subset\Mor\A$
which serve to produce factorizations $f = q j$
such that $q$ is a fibration and $j$ is a relative $\J$-cell complex.
\end{itemize}
The relative $\I$-cell complexes are automatically cofibrations.
The relative $\J$-cell complexes are automatically acyclic cofibrations.
The construction of the factorization $f = p i$ (respectively, $f = q j$)
is performed by the small object argument.
For an account of this construction,
we refer to~\cite[\S 10.5]{Hirschhorn} or to~\cite[\S 2.1]{Hovey}.

For our needs,
we recall the definition of the generating (acyclic) cofibrations
of the category of dg-modules
and we review the structure of the cofibrant cell complexes of~$\C$,
the dg-modules $C\in\C$
such that the initial morphism $i: 0\rightarrow C$
is a relative $\I$-cell complex.

\subsubsection{Cofibrant cell complexes in dg-modules}\label{Background:DGModuleCellComplexes}
Let $E[d]$ be the dg-module spanned by a homogeneous element $e_d$ of degree $d$
and a homogeneous element $b_{d-1}$ of degree $d-1$,
together with the differential such that $\delta(e_d) = b_{d-1}$.
Let $B[d]$ be the submodule of $E[d]$ spanned by~$b_{d-1}$.
The embeddings $i_d: B[d]\rightarrow E[d]$, $d\in\ZZ$,
define the generating cofibrations of the category of dg-modules.
The morphisms $j_d: 0\rightarrow E[d]$, $d\in\ZZ$,
define the generating acyclic cofibrations.

The cofibrant cell complexes of the category of dg-modules are identified with twisted dg-modules $C = (E,\partial)$
formed from a free graded $\kk$-module $E = \bigoplus_{\alpha\in\Lambda}\kk e_\alpha$,
equipped with a trivial internal differential,
together with a basis filtration
\begin{equation*}
\emptyset = \{e_\alpha\}_{\alpha\in\Lambda_0}\subset\{e_\alpha\}_{\alpha\in\Lambda_1}\subset\dots
\subset\{e_\alpha\}_{\alpha\in\Lambda_{\lambda}}\subset\dots
\subset\{e_\alpha\}_{\alpha\in\Lambda}
\end{equation*}
such that $\partial(E_{\lambda})\subset E_{\lambda-1}$,
where we set $E_{\lambda} = \bigoplus_{\alpha\in\Lambda_{\lambda}}\kk e_\alpha$.

This identity of structures follows from a straightforward inspection of the structure of pushouts
along generating cofibrations in the category of dg-modules.

Indeed,
one observes easily that a pushout
\begin{equation*}
\xymatrix{ \displaystyle\bigoplus_{\alpha\in S_{\lambda}} B[d_{\alpha}]\ar[r]^{f}\ar[d]_{\sum_{\alpha} i_{d_{\alpha}}} &
C_{\lambda-1}\ar@{.>}[d]^{j_{\lambda}} \\
\displaystyle\bigoplus_{\alpha\in S_{\lambda}} E[d_{\alpha}]\ar@{.>}[r]_{g} & C_{\lambda} },
\end{equation*}
amounts to the definition of a twisted dg-module
\begin{equation*}
C_{\lambda} = (C_{\lambda-1}\oplus\{\bigoplus_{\alpha\in S_{\lambda}}\kk e_{\alpha}\},\partial)
\end{equation*}
such that $\partial$ vanishes over $C_{\lambda-1}$ and maps any basis element $e_{\alpha}$, $\alpha\in S_{\lambda}$, to an element of $C_{\lambda-1}$.
The basis elements $e_{\alpha}$ represent the image of the generators $e_{d_{\alpha}}\in E[d_{\alpha}]$
under the morphism $g: \bigoplus_{\alpha\in S_{\lambda}} E[d_{\alpha}]\rightarrow C_{\lambda}$.
The twisting homomorphism is determined on $e_{\alpha}$
by the identity $\partial(e_{\alpha}) = f(b_{d_{\alpha}})$
which follows from the relation $\delta(e_{d_{\alpha}}) = b_{d_{\alpha}}$ in $E[d_{\alpha}]$
and the commutation of $g$ with differentials.

One sees immediately that a sequence of pushouts of this form produces a twisted dg-module
together with a basis filtration of the form stated,
where $\Lambda_{\mu} = \coprod_{\lambda\leq\mu} S_{\lambda}$.
To see that this correspondence of structures gives an equivalence,
note simply that $C_{\lambda}$
represents the dg-submodule of $C = \colim_{\lambda} C_{\lambda}$
spanned by the basis elements $e_{\alpha}$ such that $\alpha\in\Lambda_{\lambda}$.

\section{Operads in dg-modules}\label{Operads}
\renewcommand{\thesubsubsection}{\thesubsection.\arabic{subsubsection}}

The purpose of this section is to review standard definitions on the category of operads:
in~\S\ref{OperadDefinition},
we recall the basic definitions of the structure of an operad;
in~\S\ref{TreeOperadStructures},
we review the definition of the monad of tree tensors
which gives an intuitive and global representation of the composition structure of an operad;
in~\S\ref{OperadHomotopy},
we recall the definition the model structure of the category of operads in dg-modules;
in~\S\ref{QuasiFreeOperads},
we study the structure of cofibrant cell objects of the category of operads.

\subsection{Recollections of basic definitions}\label{OperadDefinition}
The structure formed by a collection $M = \{M(n)\}_{n\in\NN}$,
where $M(n)$ is a dg-module equipped with an action of the group $\Sigma_n$ of permutations of~$\{1,\dots,n\}$,
is called a $\Sigma_*$-object.

According to the original definition of~\cite{May},
an operad consists of a $\Sigma_*$-object $\POp$
equipped with a unit element $1\in\POp(1)$
and composition products
\begin{equation}
\POp(r)\otimes\POp(n_1)\otimes\dots\otimes\POp(n_r)\xrightarrow{\mu}\POp(n_1+\dots+n_r)
\end{equation}
that satisfy natural equivariance, associativity and unit relations.
The composition product of~$p\in\POp(r)$
with $q_1\in\POp(n_1),\dots,q_r\in\POp(n_r)$
is denoted by $p(q_1,\dots,q_r)\in\POp(n_1+\dots+n_r)$.
An operad is also equipped with partial composites
\begin{equation}\renewcommand{\theequation}{**}
p\circ_e q = p(1,\dots,q,\dots,1),\quad\text{for $e = 1,\dots,r$},
\end{equation}
defined by the composition of~$p\in\POp(r)$
with an operation $q\in\POp(s)$
at the $e$th position.
The associativity of composition products implies that the structure of an operad
can equivalently be defined in terms partial composites,
for which we also have natural equivariance, associativity and unit relations~(see~\cite{Markl} or \cite[Proposition 13]{MarklSurvey}
and \cite[\S 1.3]{MarklShniderStasheff}).
The proof of this equivalence of definitions is revisited next.
For the moment,
we only assume that the existence of a composition structure~(*)
implies the existence of partial composites~(**).

The previous definitions are the most natural if we regard the elements of an operad
as operations with $r$ inputs.
But the composition structure of an operad can also be embodied into a conceptual and compact categorical definition.

The unit element $1\in\POp(1)$
is equivalent to a morphism $\eta: \IOp\rightarrow\POp$,
where $\IOp$ is the $\Sigma_*$-object such that
\begin{equation*}
\IOp(n) = \begin{cases} \kk, & \text{if $n=1$}, \\ 0, & \text{otherwise}. \end{cases}
\end{equation*}
The composition products are equivalent to a single morphism of $\Sigma_*$-objects
$\mu: \POp\circ\POp\rightarrow\POp$,
where $\circ$ is a certain operation on the category of $\Sigma_*$-objects,
the composition product of~$\Sigma_*$-objects.
The equivariance axioms of the original definition is encoded in the definition of~$\circ$.
The composition product~$\circ$ is a unitary and associative operation, with the $\Sigma_*$-object~$\IOp$ as a unit,
and the unit and associativity axioms of the definition of an operad
are equivalent to the commutativity of usual diagrams:
\begin{equation*}
\vcenter{\xymatrix{ \IOp\circ\POp\ar[r]^{\eta\circ\POp}\ar[dr]_{\simeq} & \POp\circ\POp\ar[d]^{\mu} &
\POp\circ \IOp\ar[l]_{\eta\circ\POp}\ar[dl]^{\simeq} \\
& \POp & }}
\qquad\text{and}
\qquad\vcenter{\xymatrix{ \POp\circ\POp\circ\POp\ar[r]^{\POp\circ\mu}\ar[d]_{\mu\circ\POp} & \POp\circ\POp\ar[d]^{\mu} \\
\POp\circ\POp\ar[r]_{\mu} & \POp }}.
\end{equation*}

For the purpose of operadic cobar constructions,
we give a graphical definition of the composition product~$M\circ N$
from which we derive an intuitive representation
of the composition product of an operad $\mu: \POp\circ\POp\rightarrow\POp$.

For another detailed account on the composition structure of $\Sigma_*$-objects,
we refer to the surveys of~\cite[\S 1]{FresseKoszulDuality} and~\cite[\S 1.8]{MarklShniderStasheff}.
The abstract definition of an operad in terms of a triple $(\POp,\mu,\eta)$
goes back to~\cite{Smirnov}.
Next (\S\ref{OperadAlgebras:OperadMonad}),
we recall that this abstract definition reflects the correspondence of~\cite{May}
between operads and monads.

\subsubsection{The numbering free representation of $\Sigma_*$-objects}\label{OperadDefinition:SigmaObjectInputs}
In our graphical constructions,
we use that a $\Sigma_*$-object is equivalent to a functor $M: \Bij\rightarrow\C$,
where $\Bij$ refers to the category of finite sets and bijections between them.
We adopt the notation $|I|$ to refer to the cardinal
of any finite set $I$.

In one direction,
for a functor $M: \Bij\rightarrow\C$,
the dg-module
\begin{equation*}
M(n) =\linebreak M(\{1,\dots,n\})
\end{equation*}
inherits a natural $\Sigma_n$-action since a permutation $w\in\Sigma_n$
is equivalent to a bijection $w: \{1,\dots,n\}\rightarrow\{1,\dots,n\}$.
Hence,
the collection of dg-modules $M(n) = M(\{1,\dots,n\})$, $n\in\NN$,
forms a $\Sigma_*$-object naturally associated to~$M: \Bij\rightarrow\C$.

In the converse direction,
to a $\Sigma_*$-object $M$
we associate the functor $M: \Bij\rightarrow\C$
such that
\begin{equation*}
M(I) = \Bij(\{1,\dots,n\},I)\otimes_{\Sigma_n} M(n),
\end{equation*}
for any set~$I$ of cardinal~$n = |I|$.
The tensor product of a dg-module $C$
with a set $S$ is the coproduct over $S$
of copies of~$C$.
The tensor product over~$\Sigma_n$
coequalizes the natural $\Sigma_n$-action on~$M(n)$
with the action of permutations by right translations on~$\Bij(\{1,\dots,n\},I)$.
An element $x\in M(I)$
is represented by a box labeled by $x$
with $1$ output and $n$ inputs indexed by the set~$I = \{i_1,\dots,i_n\}$:
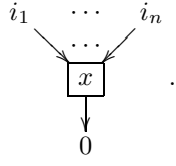
\begin{equation*}
\vcenter{\xymatrix@H=6pt@W=3pt@M=2pt@!R=1pt@!C=1pt{ i_1\ar[dr] & \cdots\ar@{}[d]|{\displaystyle{\cdots}} & i_n\ar[dl] \\
& *+<8pt>[F]{x}\ar[d] & \\ & 0 & }}.
\end{equation*}

\subsubsection{The representation of composites by trees with levels}\label{OperadDefinition:LevelTrees}
The terms $(M\circ N)(I)$
of the composite $\Sigma_*$-object $M\circ N$
are defined intuitively as dg-modules
spanned by tensors arranged on oriented trees with one level of inputs indexed by the set $I$,
two level of vertices,
and one output (for short, we call this structure an $I$-tree with two levels):
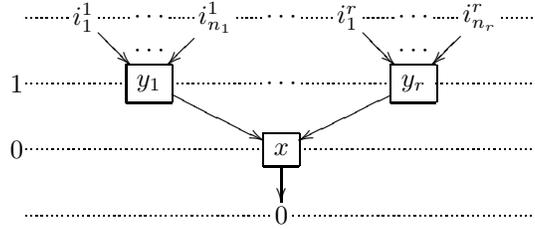
\begin{equation*}
\vcenter{\xymatrix@H=6pt@W=3pt@M=2pt@!R=1pt@!C=1pt{ \ar@{.}[r] &
i^1_1\ar[dr]\ar@{.}[r] & \cdots\ar@{.}[r]\ar@{}[d]|{\displaystyle{\cdots}} & i^1_{n_1}\ar[dl]\ar@{.}[r] &
\cdots\ar@{.}[r] & i^r_1\ar[dr]\ar@{.}[r] & \cdots\ar@{.}[r]\ar@{}[d]|{\displaystyle{\cdots}} & i^r_{n_r}\ar[dl]\ar@{.}[r] & \\
*+<2pt>{1}\ar@{.}[rr] && *+<8pt>[F]{y_1}\ar[drr]\ar@{.}[rr] && \cdots\ar@{.}[rr] && *+<8pt>[F]{y_r}\ar[dll]\ar@{.}[rr] && \\
*+<2pt>{0}\ar@{.}[rrrr] &&&& *+<8pt>[F]{x}\ar[d]\ar@{.}[rrrr] &&&& \\ \ar@{.}[rrrr] &&&& 0\ar@{.}[rrrr] &&&& }}.
\end{equation*}
The vertices at level $1$ of the tree are labeled by elements $y_v\in N$.
The inputs of $y_v$ are in bijection with the inputs of the associated vertex~$v$.
The vertex at level $0$ is labeled by an element $x\in M$
whose inputs are in bijection with vertices at level $1$.

The structure of an $I$-tree with two levels $\tau$
is fully determined by a set $I_0$, whose elements represent the vertices at level $1$ of~$\tau$,
together with a partition of the inputs $I$ indexed by the vertices at level $1$:
\begin{equation*}
I = \coprod_{v\in I_0} I_v.
\end{equation*}
The edges of the tree $\tau$ connect each input $i\in I$ to the vertex $v\in I_0$ such that $i\in I_v$,
each vertex at level $1$ to the vertex at level $0$,
and the vertex at level $0$ to the output $0$.
To an $I$-tree with two levels $\tau$,
we associate the tensor product
\begin{equation*}
\tau(M,N) = M(I_0)\otimes\{\bigotimes_{v\in I_0} N(I_v)\}.
\end{equation*}

Define an isomorphism of $I$-trees with two levels as a bijection between the sets of vertices at level $1$
such that the partition $I = \coprod_{v\in I_0} I_v$
of the tree inputs is preserved.
Naturally,
an isomorphism of $I$-trees with two levels $\theta: \sigma\rightarrow\tau$
induces an isomorphism of dg-modules $\theta_*: \sigma(M,N)\rightarrow\tau(M,N)$.

Let $\Lev_2(I)$ denote the groupoid formed by~$I$-trees with two levels.
The dg-module $(M\circ N)(I)$
is defined by a sum
\begin{equation*}
(M\circ N)(I) = \bigoplus_{\tau\in\Lev_2(I)} \tau(M,N)/\equiv
\end{equation*}
over~$\Lev_2(I)$
divided out by the action of isomorphisms.

Any bijection $u: I\rightarrow J$
induces a groupoid isomorphism $u_*: \Lev_2(I)\rightarrow\Lev_2(J)$
whose action consists in reindexing the inputs of~$I$-trees.
This isomorphism gives rise to a dg-module
isomorphism
\begin{equation*}
u_*: (M\circ N)(I)\rightarrow(M\circ N)(J).
\end{equation*}
From this observation,
we conclude that $(M\circ N)(I)$ defines a functor on the category of finite sets and bijections
and hence forms a $\Sigma_*$-object naturally associated to $M$ and $N$.

In the literature,
the composition product~$M\circ N$ is usually defined without any reference to tree structures.
Nevertheless,
the reader can check easily that the definition of this paragraph
is strictly equivalent to other standard definitions.
The existence of coherent associativity and unit isomorphisms
\begin{equation*}
(M\circ N)\circ P\simeq M\circ(N\circ P)
\quad\text{and}\quad M\circ \IOp\simeq M\simeq \IOp\circ M
\end{equation*}
is well known
and can be checked by an easy inspection (see for instance~\cite[\S 1.8]{MarklShniderStasheff}).

\subsubsection{The composition product of an operad}\label{OperadDefinition:CompositionProduct}
The elements of~$(M\circ N)(n) = (M\circ N)(\{1,\dots,n\})$
represent formal composites $w\cdot x(y_1,\dots,y_r)$,
where $w$ is a permutation of~$\{1,\dots,n\}$
which shares out the inputs of the composite between the factors~$y_1,\dots,y_r\in N$.
These formal composites $w\cdot x(y_1,\dots,y_r)$
satisfy $\Sigma_*$-invariance relations
which are encoded in the definition of~$(M\circ N)$.
Therefore
the equivariance axioms of operads imply readily that the composition products
of an operad
\begin{equation*}
\POp(r)\otimes\POp(n_1)\otimes\dots\otimes\POp(n_r)\xrightarrow{\mu}\POp(n_1+\dots+n_r)
\end{equation*}
assemble to a morphism of $\Sigma_*$-objects
$\mu: \POp\circ\POp\rightarrow\POp$
which maps the formal composites $w\cdot p(q_1,\dots,q_r)\in\POp\circ\POp$
to their evaluation in~$\POp$.

In the paper,
we essentially use the intuitive representation of the composition product $\mu: \POp\circ\POp\rightarrow\POp$.

The partial composites $w\cdot p\circ_e q = w\cdot p(1,\dots,q,\dots,1)$
are associated to tensors
of the form
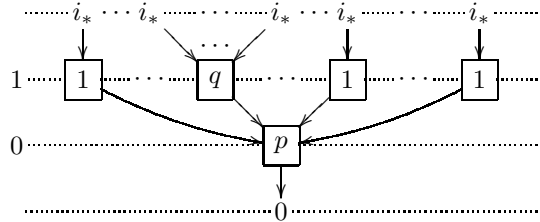
\begin{equation*}
\xymatrix@H=6pt@W=3pt@M=2pt@!R=1pt@!C=1pt{ \ar@{.}[r] & i_*\ar[d]\ar@{.}[r]|{\displaystyle\cdots} &
i_*\ar[dr]\ar@{.}[r] & \cdots\ar@{.}[r]\ar@{}[d]|{\displaystyle{\cdots}} & i_*\ar[dl]\ar@{.}[r]|{\displaystyle\cdots} &
i_*\ar[d]\ar@{.}[r] & \cdots\ar@{.}[r] &
i_*\ar[d]\ar@{.}[r] & \\
1\ar@{.}[r] & *+<3mm>[F]{1}\ar@/_4pt/[drrr]\ar@{.}[r] & \cdots\ar@{.}[r]
& *+<3mm>[F]{q}\ar[dr]\ar@{.}[r] & \cdots\ar@{.}[r] &
*+<3mm>[F]{1}\ar[dl]\ar@{.}[r] & \cdots\ar@{.}[r] & *+<3mm>[F]{1}\ar@/^4pt/[dlll]\ar@{.}[r] & \\
0\ar@{.}[rrrr] &&&& *+<3mm>[F]{p}\ar[d]\ar@{.}[rrrr] &&&& \\
\ar@{.}[rrrr] &&&& 0\ar@{.}[rrrr] &&&& }
\end{equation*}
in our representation of~$\POp\circ\POp$.

\subsubsection{Free operads}\label{OperadDefinition:FreeOperads}
Let $\Op$ denote the category of operads. Let $\M$ denote the category of $\Sigma_*$-objects.
We have an obvious forgetful functor $U: \Op\rightarrow\M$.
This functor has a left adjoint $\FOp: \M\rightarrow\Op$
which maps any $\Sigma_*$-object $M$
to a free operad $\FOp(M)$
generated by $M$.
This free operad is endowed with a natural morphism $\eta: M\rightarrow \FOp(M)$
and is characterized by the usual adjunction relation, namely:
any morphism $f: M\rightarrow\POp$ toward an operad $\POp$
has a unique factorization
\begin{equation*}
\xymatrix{ M\ar[rr]^{f}\ar[dr]_{\eta} && \POp \\ & \FOp(M)\ar@{.>}[ur]_{\exists!\phi_f} & }
\end{equation*}
such that $\phi_f$ is a morphism of operads.

A monad on $\Sigma_*$-objects is defined by the functor underlying the free operad $\FOp: \M\rightarrow\M$
together with the natural transformation $\eta: M\rightarrow \FOp(M)$,
which represents the universal morphism of the free operad,
and the natural transformation $\mu: \FOp(\FOp(M))\rightarrow \FOp(M)$
associated to the identity morphism $\id: \FOp(M)\rightarrow \FOp(M)$
by the universal property of the free operad.
According to~\cite[Proposition 1.12]{GetzlerJones},
the category of operads is isomorphic to the category of algebras over~$\FOp$,
the structures formed by a $\Sigma_*$-object $\POp$
together with a morphism $\lambda: \FOp(\POp)\rightarrow\POp$
which satisfy natural unit and associativity relations
with respect to the monad structure of~$\FOp$.
To produce the morphism $\lambda: \FOp(\POp)\rightarrow\POp$
which provides any operad with the structure of an algebra over~$\FOp$,
we simply apply the universal property of the free operad to the identity morphism $\id: \POp\rightarrow\POp$

In the next subsection,
we review an explicit construction of the free operad $\FOp(M)$
but we do not address the monadic structure alluded to in this paragraph:
in view toward applications of~\S\S\ref{CobarConstruction}-\ref{CylinderHomotopyMorphisms},
we rather make explicit a reduced version of this structure.

\subsection{Augmented operads, connected operads and tree composites}\label{TreeOperadStructures}
The unit $\Sigma_*$-object $\IOp$
inherits an obvious operad structure
and represents the initial object of the category of operads.
By a natural extension of usual terminologies of algebra,
an operad $\POp$ equipped with an augmentation morphism $\epsilon: \POp\rightarrow \IOp$,
which makes $\POp$ an object of the category of operads over $\IOp$,
is called an augmented operad (see~\cite[Definition 20]{MarklSurvey}).
The kernel of the augmentation morphism of an augmented operad,
for which we adopt the notation $\tilde{\POp} = \ker(\epsilon: \POp\rightarrow \IOp)$,
is called the augmentation ideal of~$\POp$.
The unit morphism and the augmentation
of an augmented operad yields a natural splitting $\POp = \IOp\oplus\tilde{\POp}$.

In this subsection,
we study the composition structure of augmented operads.
First,
we review an explicit definition of the free operad $\FOp(M)$ generated by a $\Sigma_*$-object~$M$
and we observe that $\FOp(M)$ has a natural splitting $\FOp(M) = \IOp\oplus\tilde{\FOp}(M)$.
Then
we check that the functor~$\tilde{\FOp}(M)$ inherits a natural monad structure
and we prove that the composition structure of an augmented operad~$\POp$
is part of a total composition product $\lambda: \tilde{\FOp}(\tilde{\POp})\rightarrow\tilde{\POp}$
which makes the augmentation ideal of~$\POp$
an algebra over the monad $\tilde{\FOp}$.

From this observation,
we obtain that the structure of an algebra over $\tilde{\FOp}$
includes an augmented operad structure.
In the converse direction,
one observes that the total composition product $\lambda: \tilde{\FOp}(\tilde{\POp})\rightarrow\tilde{\POp}$ of an algebra over $\tilde{\FOp}$
is fully determined by the composition structure of an operad.
Hence,
we have an equivalence of categories between algebras over $\tilde{\FOp}$
and augmented operads.

For our needs,
we revisit the proof of this equivalence of categories,
already defined in~\cite[\S 5]{MarklSurvey}
(we also refer to~\cite[Proposition 1.12]{GetzlerJones} and~\cite[\S 1.2]{GinzburgKapranov} for the analogous case of unaugmented operads).
The main novelty of our approach lies in the observation that the composite $\Sigma_*$-object $\POp\circ\POp$
embeds into the free operad $\FOp(\tilde{\POp})$.
This argument is used crucially in~\S\ref{CylinderHomotopyMorphisms}, where we define and study cylinder objects in the category of operads,
and motivates the review of this subsection.

The free operad $\FOp(M)$ is defined intuitively as a dg-module spanned by formal composites of generating operations $x\in M$.
These formal composites are represented by tensors arranged on vertices of a tree.
First of all,
we define the tree structures
which occur in this representation of~$\FOp(M)$.
For that purpose, we adapt definitions of~\cite[\S 2]{Serre}.
The same formalism is used in~\cite[\S 3]{FresseKoszulDuality}.

\subsubsection{Tree structures}\label{TreeOperadStructures:TreeStructures}
Let $I$ be any finite set.
An $I$-tree refers to an abstract oriented tree with one outgoing edge, whose target is usually denoted by $0$,
and ingoing edges, whose sources are indexed by~$I$.
Formally,
the structure of an $I$-tree $\tau$
is defined by a set of vertices $V(\tau)$,
a set of edges $E(\tau)$,
together with a source map $s: E(\tau)\rightarrow V(\tau)\amalg I$
and a target map $t: E(\tau)\rightarrow V(\tau)\amalg\{0\}$
such that the following properties hold:
\begin{enumerate}\setcounter{enumi}{-1}
\item\label{TreeConnectedness}
each vertex $v\in V(\tau)$ is connected to the output $0$
by a chain of edges
\begin{equation*}
v\xrightarrow{e_1} v_1\xrightarrow{e_2}\cdots\xrightarrow{e_{l-1}} v_{l-1}\xrightarrow{e_l} 0
\end{equation*}
so that $v = s(e_1)$, $t(e_1) = s(e_2)$, \dots, $t(e_{l-1}) = s(e_l)$ and $t(e_l) = 0$;
\item\label{TreeInputs}
for each $i\in I$,
there is one and only one edge $e\in E(\tau)$
such that $s(e) = i$;
\item\label{TreeVertexOutput}
for each vertex $v\in V(\tau)$,
there is one and only one edge $e\in E(\tau)$
such that $s(e) = v$;
\item\label{TreeRootOutput}
there is one and only one edge $e\in E(\tau)$ such that $t(e) = 0$.
\end{enumerate}
As an example,
the tree of figure~\ref{Fig:TreeExample} has $V(\tau) = \{v_1,\dots,v_5\}$ as vertex set,
$E(\tau) = \{e_1,\dots,e_{11}\}$ as edge set
and $I = \{i_1,\dots,i_6\}$ as input set.
In this graphical representation of a tree,
the edges $e$ are represented by an arrow oriented from their source $s(e) = u$
to their target $t(e) = v$.
\begin{figure}
\begin{equation*}
\xymatrix@M=0pt@!R=2em@!C=2em{ &&& *+<3mm>{i_2}\ar[dr]|-*+<1mm>{e_9} & *+<3mm>{i_3}\ar[d]|-*+<1mm>{e_{10}} & *+<3mm>{i_4}\ar[dl]|-*+<1mm>{e_{11}} && \\
&& *+<3mm>{i_1}\ar[ddrr]|-*+<1mm>{e_3} && *+<3mm>[o][F]{v_5}\ar[d]|-*+<1mm>{e_6} & *+<3mm>{i_5}\ar[dr]|-*+<1mm>{e_7} && *+<3mm>{i_6}\ar[dl]|-*+<1mm>{e_8} \\
& *+<3mm>[o][F]{v_2}\ar@/_/[drrr]|-*+<1mm>{e_2} &&& *+<3mm>[o][F]{v_3}\ar[d]|-*+<1mm>{e_4} && *+<3mm>[o][F]{v_4}\ar@/^/[dll]|-*+<1mm>{e_5} & \\
&&&& *+<3mm>[o][F]{v_1}\ar[d]|-*+<1mm>{e_1} &&& \\
&&&& *+<3mm>{0} &&& }
\end{equation*}
\caption{}\label{Fig:TreeExample}\end{figure}

The set of $I$-trees, denoted by $\Theta(I)$,
is equipped with a natural groupoid structure.
Formally,
an isomorphism of $I$-trees $\theta: \sigma\rightarrow\tau$
is defined by bijections $\theta_V: V(\sigma)\rightarrow V(\tau)$ and $\theta_E: E(\sigma)\rightarrow E(\tau)$
preserving the source and target of edges.

\subsubsection{The dg-module of tree tensors}\label{TreeOperadStructures:TreeTensors}
The inputs of a vertex $v\in V(\tau)$ in an $I$-tree $\tau$
is the set $I_v\subset V(\tau)\amalg I$
formed by the source $s(e)$ of the edges such that $t(e) = v$.
In the example of figure~\ref{Fig:TreeExample},
we have $I_{v_1} = \{i_1,v_2,v_3,v_4\}$, $I_{v_2} = \emptyset$, $I_{v_3} = \{v_5\}$, $I_{v_4} = \{i_5,i_6\}$
and $I_{v_5} = \{i_2,i_3,i_4\}$.

To a $\Sigma_*$-object $M$,
we associate the dg-module
\begin{equation*}
\tau(M) = \bigotimes_{v\in V(\tau)} M(I_v),
\end{equation*}
spanned by tensor products $\bigotimes_v x_v$
whose elements $x_v$
are represented by a labeling of the vertices of~$\tau$.
The inputs of a label $x_v\in M(I_v)$ are in bijection with the inputs
of the associated vertex $v$
by definition of the dg-module $M(I_v)$.

Naturally,
an isomorphism of $I$-trees $\theta: \sigma\rightarrow\tau$
induces an isomorphism of dg-modules $\theta_*: \sigma(M)\rightarrow\tau(M)$.

\subsubsection{The construction of free operads}\label{TreeOperadStructures:FreeOperadConstruction}
The term $\FOp(M)(I)$ of the free operad $\FOp(M)$
is defined explicitly by the sum
\begin{equation*}
\FOp(M)(I) = \bigoplus_{\tau\in\Theta(I)} \tau(M)/\equiv
\end{equation*}
over the set of $I$-trees~$\Theta(I)$
divided out by the action of isomorphisms.
Naturally,
any bijection $u: I\rightarrow J$
induces a groupoid isomorphism $u_*: \Theta(I)\rightarrow\Theta(J)$,
whose action consists in reindexing the inputs of~$I$-trees,
and this isomorphism gives rise to a dg-module
isomorphism
\begin{equation*}
u_*: \FOp(M)(I)\rightarrow \FOp(M)(J),
\end{equation*}
so that the collection of dg-modules $\FOp(M)(I)$
defines a functor on the category of finite sets and bijections between them.

The unit morphism $\eta: \IOp\rightarrow \FOp(M)$,
identifies $\IOp(1) = \kk$
with the summand of~$\FOp(M)$
associated to the $1$-tree with no vertices
\begin{equation*}
\downarrow = \vcenter{\xymatrix@H=6pt@W=3pt@M=2pt@!R=1pt@!C=1pt{ 1\ar[d] \\ 0 }},
\end{equation*}
for which we have an identity $\downarrow(M) = \kk$.
The elements of the composite $\FOp(M)\circ \FOp(M)$
consist of tree tensors arranged on the vertices of a big tree.
From this representation,
we see readily that an element of~$\FOp(M)\circ \FOp(M)$ is identified with a tree tensor formed on a big tree
equipped with a partition into (possibly empty) small trees
arranged on two levels.
Hence,
we have a natural composition product $\mu: \FOp(M)\circ \FOp(M)\rightarrow \FOp(M)$
which simply forgets the extra partition and level structure.
The morphisms $\eta: \IOp\rightarrow \FOp(M)$ and $\mu: \FOp(M)\circ \FOp(M)\rightarrow \FOp(M)$
satisfy visibly the unit and associativity axioms of operads.
Hence,
we have a well defined operad structure on $\FOp(M)$.

For the $I$-tree with one vertex
\begin{equation*}
\psi = \vcenter{\xymatrix@H=6pt@W=3pt@M=2pt@!R=1pt@!C=1pt{ i_1\ar[dr] & \cdots\ar@{}[d]|{\displaystyle{\cdots}} & i_n\ar[dl] \\
& *+<3mm>[o][F]{v}\ar[d] & \\ & 0 & }},
\end{equation*}
we have a natural isomorphism $\psi(M)\simeq M(I)$.
This isomorphism
determines the canonical morphism of $\Sigma_*$-objects
\begin{equation*}
\eta: M\rightarrow\FOp(M)
\end{equation*}
associated to the structure of the free operad.
In the sequel,
we use this morphism to identify $M$
with a summand of~$\FOp(M)$.

We still have to check:

\begin{prop}[{see~\cite[Proposition 1.10]{GetzlerJones} and~\cite[\S 1.2]{GinzburgKapranov}}]\label{TreeOperadStructures:FreeOperadStatement}
The operad $\FOp(M)$
together with the canonical morphism $\eta: M\rightarrow\FOp(M)$
satisfies the universal property of free objects (\S\ref{OperadDefinition:FreeOperads}).\qed
\end{prop}

This proposition is established in the references.
On the other hand,
we prove by independent arguments that (augmented) operads are equivalent to algebras over the monad $\tilde{\FOp}$ (defined next)
formed from the augmentation ideal of~$\FOp$.
Proposition~\ref{TreeOperadStructures:FreeOperadStatement}
can be obtained as a byproduct of the construction of this equivalence.
Therefore
the reader can extract a complete proof of proposition~\ref{TreeOperadStructures:FreeOperadStatement}
from explanations of the next paragraphs.

\subsubsection{The augmented structure of the free operad}\label{TreeOperadStructures:AugmentedFreeOperad}
The groupoids of $I$-trees admit a natural splitting $\Theta(I) = \coprod_r\Theta_r(I)$,
where $\Theta_r(I)$ consists of $I$-trees $\tau$ whose vertex set $V(\tau)$
has $r$ elements.
From this observation,
we deduce that the underlying $\Sigma_*$-object of the free operad $\FOp(M)$
inherits a splitting $\FOp(M) = \bigoplus_{r=0}^{\infty} \FOp_r(M)$,
where $\FOp_r(M)$ consists of summands $\tau(M)$
such that $\tau\in\Theta_d(I)$.

Note that $\Theta_0(I)$ is reduced to the unit $1$-tree with no vertices
if the input set~$I$ is reduced to $1$ element
and is empty otherwise.
Hence we have $\FOp_0(M) = \IOp$.
The projection onto this summand $\FOp_0(M) = \IOp$
defines clearly an operad morphism $\epsilon: \FOp(M)\rightarrow \IOp$
so that $\FOp(M)$ forms an augmented operad (see also~\cite[\S 4]{MarklSurvey} for this observation).
The augmentation ideal of~$\FOp(M)$
satisfies
\begin{equation*}
\tilde{\FOp}(M) = \bigoplus_{\tau\in\tilde{\Theta}(I)} \tau(M)/\equiv,
\end{equation*}
where $\tilde{\Theta}(I) = \coprod_{r>0}\Theta_d(I)$ is the set of $I$-trees
with a non-empty set of vertices.

Note that $\Theta_1(I)$ is reduced to the isomorphism class of the one-vertex tree~$\psi$
of~\S\ref{TreeOperadStructures:FreeOperadConstruction}.
Hence, we also have an isomorphism $M\simeq\FOp_1(M)$
and the universal morphism of the free operad $\eta: M\rightarrow\FOp(M)$
identifies the $\Sigma_*$-object $M$
with that summand $\FOp_1(M)\subset\FOp(M)$.

\subsubsection{The reduced monadic structure of the free operad}\label{TreeOperadStructures:ReducedFreeOperad}
In this paragraph,
we give, after~\cite[\S 5]{MarklSurvey}, a direct definition (without referring to the universal property of free objects)
of a monad structure on the augmentation ideal of the free operad $\FOp(M)$.

The isomorphism $M(I)\simeq\psi(M)$ of~\S\ref{TreeOperadStructures:FreeOperadConstruction},
and the universal morphism of the free operad $\eta: M\rightarrow \FOp(M)$,
identifies the $\Sigma_*$-object $M$
with a summand of~$\tilde{\FOp}(M)\subset\FOp(M)$.
The inclusion of this summand $M = \FOp_1(M)$
into $\tilde{\FOp}(M)$
defines the monadic unit
\begin{equation*}
\eta: M\rightarrow\tilde{\FOp}(M).
\end{equation*}

An element of the composite construction~$\tilde{\FOp}(\tilde{\FOp}(M))$
consists of non-empty tree tensors arranged on the vertices of a tree.
From this representation,
we see that an element of~$\tilde{\FOp}(\tilde{\FOp}(M))$ amounts to a tensor arranged on a big tree
equipped with a partition into non-empty small trees.
We have a natural monadic product
\begin{equation*}
\mu: \tilde{\FOp}(\tilde{\FOp}(M))\rightarrow\tilde{\FOp}(M)
\end{equation*}
which simply forgets the partition structure
to retain the structure of the big tree only.
This process is displayed in figure~\ref{Fig:TreeComposite}.
\begin{figure}[t]
\begin{equation*}
\vcenter{\xymatrix@H=8pt@W=8pt@R=12pt@C=12pt@M=2pt{ i\ar@/_/[ddrr] & j\ar[dr] && k\ar[dl] & l\ar@/^/[ddl] \\
&& *+<6pt>[F]{x}\ar[d]\save[]!C.[d]!C *+<11pt>[F:<2pt>]\frm{}\restore && \\
&& *+<6pt>[F]{y}\ar[dr] &*+<6pt>[F]{z}\ar[d]\save[]!C.[d]!C *+<11pt>[F:<2pt>]\frm{}\restore & \\
&&& *+<6pt>[F]{t}\ar[d] & \\ &&& 0 & }}
\quad\mapsto
\vcenter{\xymatrix@H=8pt@W=8pt@R=12pt@C=12pt@M=2pt{ i\ar@/_/[ddrr] & j\ar[dr] && k\ar[dl] & l\ar@/^/[ddl] \\
&& *+<6pt>[F]{x}\ar[d] && \\
&& *+<6pt>[F]{y}\ar[dr] & *+<6pt>[F]{z}\ar[d] & \\
&&& *+<6pt>[F]{t}\ar[d] & \\ &&& 0 & }}
\end{equation*}
\label{Fig:TreeComposite}
\caption{}\end{figure}

The morphisms $\eta: M\rightarrow\tilde{\FOp}(M)$ and $\mu: \tilde{\FOp}(\tilde{\FOp}(M))\rightarrow\tilde{\FOp}(M)$
satisfy visibly the unit and associativity axioms of monads.
Hence,
we have a well defined monad structure on the functor $\tilde{\FOp}: \M\rightarrow\M$.

From the definition,
it appears that the structure of an algebra over~$\tilde{\FOp}$
is fully determined by a collection of dg-module morphisms
$\lambda_{\tau}: \tau(\tilde{\POp})\rightarrow\tilde{\POp}(I)$, $\tau\in\tilde{\Theta}(I)$,
commuting with the action of $I$-tree isomorphisms, with reindexing bijections,
and so that:
\begin{enumerate}
\item\label{UnitTreeComposites}
for the one-vertex tree $\psi$,
the morphism
\begin{equation*}
\lambda_{\psi}: \psi(\tilde{\POp})\rightarrow\tilde{\POp}(I)
\end{equation*}
reduces to the canonical isomorphism $\tilde{\POp}(I)\simeq\psi(\tilde{\POp})$ of~\S\ref{TreeOperadStructures:FreeOperadConstruction};
\item\label{AssociativityTreeComposites}
for any partition of a tree~$\tau$ into small nonempty subtrees $\tau_v$
arranged on a big tree $\sigma$,
the composite
\begin{equation*}
\xymatrix{ \tau(\tilde{\POp})\ar[r]^{\sigma(\lambda_*)} & \sigma(\tilde{\POp})\ar[r]^{\lambda_{\sigma}} & \tilde{\POp}(I), }
\end{equation*}
where $\sigma(\lambda_*)$ refers to the evaluation of the morphisms $\lambda_{\tau_v}: \tau_v(\tilde{\POp})\rightarrow\tilde{\POp}(I_v)$
on the components of~$\tau$,
agrees with $\lambda_{\tau}: \tau(\tilde{\POp})\rightarrow\tilde{\POp}(I)$.
\end{enumerate}

\subsubsection{Tree associativity and partial composition products}\label{TreeOperadStructures:TreeAssociativity}
The associativity relations~(\ref{AssociativityTreeComposites})
yielded by trees with three vertices
are represented in figures~(\ref{Fig:TreeLinearAssociativity}-\ref{Fig:TreeRamifiedAssociativity}).
In these representations,
the tree composition product $\lambda_*: \tau(\POp)\rightarrow\POp(I)$
is applied to each framed subtree.
\begin{figure}[t]
\begin{equation*}
\xymatrix@W=2pt@H=4pt@R=8pt@C=2pt@M=2pt{ &&&&
\save [].[rrrrdddd]!C="g2"\restore
& k_1\ar[dr] & \cdots & k_t\ar[dl] & 
&&&
\save [].[rrrrdddd]!C="g3"\restore
& k_1\ar[ddr] & \cdots\ar@{}[dd]|{\displaystyle{\cdots}} & k_t\ar[ddl] & 
&&&
&& 
\\ &&&&
j_1\ar[drr] & \cdots & *+<6pt>[F]{p_{3}}\save[]!C.[d]!C *+<8pt>[F-,]\frm{}\restore \ar[d] & \cdots & j_s\ar[dll] 
&&&
j_1\ar@{}[r]|(0.8){\displaystyle{\cdots}}\ar[drr] &&&& j_s\ar@{}[l]|(0.8){\displaystyle{\cdots}}\ar[dll] 
&&&
&& 
\\ &&&&
i_1\ar[drr] & \cdots & *+<6pt>[F]{p_{2}}\ar[d] & \cdots & i_r\ar[dll] 
&&&
i_1\ar[drr] & \cdots & *+<6pt>[F]{p_{23}}\ar[d] \save[]!C.[d]!C *+<8pt>[F-,]\frm{}\restore
& \cdots & i_r\ar[dll] 
&&&
&& 
\\ &&&&
&& *+<6pt>[F]{p_{1}}\ar[d]\save[]!C.[]!C *+<8pt>[F-,]\frm{}\restore && 
&&&
&& *+<6pt>[F]{p_{1}}\ar[d] && 
&&&
&& 
\\ &&&&
&& 0 && 
&&&
&& 0 && 
&&&
&& 
\\
&&&& &&& &&&& &&& &&&&&& \\
\save [].[rrrrdddd]!C="g1"\restore
& k_1\ar[dr] & \cdots & k_t\ar[dl] & 
&&&
&&&& 
&&&
\save [].[rrrrrrdddd]!C="g4"\restore
&& k_1\ar[dddr] & \cdots\ar@{}[ddd]|(0.33){\displaystyle{\cdots}} & k_t\ar[dddl] && 
\\
j_1\ar[drr] & \cdots & *+<6pt>[F]{p_{3}}\ar[d] \save[]!C.[dd]!C *+<8pt>[F-,]\frm{}\restore
& \cdots & j_s\ar[dll] 
&&&
&&&& 
&&&
& j_1\ar[ddrr]\ar@{}[r]|(0.65){\displaystyle{\cdots}} &&&& j_s\ar@{}[l]|(0.65){\displaystyle{\cdots}}\ar[ddll] & 
\\
i_1\ar[drr] & \cdots & *+<6pt>[F]{p_{2}}\ar[d] & \cdots & i_r\ar[dll] 
&&&
&&&& 
&&&
i_1\ar[drrr] & \cdots &&&& \cdots & i_r\ar[dlll] 
\\
&& *+<6pt>[F]{p_1}\ar[d] && 
&&&
&&&& 
&&&
&&& *+<6pt>[F]{p_{123}}\ar[d] &&& 
\\
&& 0 && 
&&&
&&&& 
&&&
&&& 0 &&& 
\\
&&&& &&& &&&& &&& &&&&&& \\ &&&&
\save [].[rrrrdddd]!C="g5"\restore
& k_1\ar[dr] & \cdots & k_t\ar[dl] & 
&&&
\save [].[rrrrdddd]!C="g6"\restore
& k_1\ar[dr] & \cdots & k_t\ar[dl] & 
&&&
&& 
\\ &&&&
j_1\ar[drr] & \cdots & *+<6pt>[F]{p_{3}}\ar[d]\save[]!C.[]!C *+<8pt>[F-,]\frm{}\restore & \cdots & j_s\ar[dll] 
&&&
j_1\ar[ddrr] & \cdots & *+<6pt>[F]{p_{3}}\ar[dd] \save[]!C.[dd]!C *+<8pt>[F-,]\frm{}\restore
& \cdots & j_s\ar[ddll] 
&&&
&& 
\\ &&&&
i_1\ar[drr] & \cdots & *+<6pt>[F]{p_{2}}\ar[d]\save[]!C.[d]!C *+<8pt>[F-,]\frm{}\restore & \cdots & i_r\ar[dll] 
&&&
i_1\ar[drr]\ar@{}[r]|(0.6){\displaystyle{\cdots}} &&&& i_r\ar@{}[l]|(0.6){\displaystyle{\cdots}}\ar[dll] 
&&&
&&& 
\\ &&&&
&& *+<6pt>[F]{p_{1}}\ar[d] && 
&&&
&& *+<6pt>[F]{p_{12}}\ar[d] && 
&&&
&& 
\\ &&&&
&& 0 && 
&&&
&& 0 && 
&&&
&& 
\ar@/^8pt/ "g1"!U+<0pt,6pt>;"g2"^(0.2){\simeq}
\ar "g2"!R+<6pt,0pt>;"g3"!L-<6pt,0pt>^{\lambda_*}
\ar@/^8pt/ "g3";"g4"!U+<0pt,6pt>^(0.75){\lambda_*}
\ar@/_8pt/ "g1"!D-<0pt,6pt>;"g5"_(0.2){\simeq}
\ar "g5"!R+<6pt,0pt>;"g6"!L-<6pt,0pt>_{\lambda_*}
\ar@/_8pt/ "g6";"g4"!D-<0pt,6pt>_(0.75){\lambda_*}
\ar "g1"!R+<6pt,0pt>;"g4"!L-<6pt,0pt>^{\lambda_*}
}
\end{equation*}
\caption{}\label{Fig:TreeLinearAssociativity}\end{figure}
\begin{figure}[t]
\begin{equation*}
\xymatrix@W=2pt@H=4pt@R=12pt@C=1pt@M=1pt{
&& 
\save [].[rrrrrrrrddd]!C="g2"\restore 
& j_1\ar[dr] & \cdots & j_s\ar[dl] && k_1\ar[dr] & \cdots & k_t\ar[dl] & 
&&&&
\save [].[rrrrrrrrddd]!C="g3"\restore 
& j_1\ar[ddrrr] & \cdots & j_s\ar[ddr] && k_1\ar[dr] & \cdots & k_t\ar[dl] & 
&& 
\\
&& 
i_1\ar@/_2pt/[drrrr] & \cdots & *+<6pt>[F]{p_{2}}\ar[drr]\save[]!C.[rrd]!C *+<8pt>[F-,]\frm{}\restore
&& \cdots && *+<6pt>[F]{p_{3}}\ar[dll]\save[]!C *+<8pt>[F-,]\frm{}\restore & \cdots & i_r\ar@/^6pt/[dllll] 
&&&&
i_1\ar@/_2pt/[drrrr] & \cdots & && \cdots && *+<6pt>[F]{p_{3}}\ar[dll]\save[]!C.[dll]!C *+<8pt>[F-,]\frm{}\restore & \cdots & i_r\ar@/^2pt/[dllll] 
&& 
\\
&& 
&&&& *+<6pt>[F]{p_{1}}\ar[d] &&&& 
&&&&
&&&& *+<6pt>[F]{p_{12}}\ar[d] &&&& 
&& 
\\
&& 
&&&& 0 &&&& 
&&&&
&&&& 0 &&&& 
&& 
\\ &&&& &&&& && &&&& && &&&& &&&& \\
\save [].[rrrrrrrrddd]!C="g1"\restore 
& j_1\ar[dr] & \cdots & j_s\ar[dl] && k_1\ar[dr] & \cdots & k_t\ar[dl] & 
&& 
&&&&
&& 
\save [].[rrrrrrrrddd]!C="g4"\restore 
& j_1\ar[ddrrr] & \cdots & j_s\ar[ddr] && k_1\ar[ddl] & \cdots & k_t\ar[ddlll] & 
\\
i_1\ar@/_2pt/[drrrr] & \cdots & *+<6pt>[F]{p_{2}}\ar[drr] \save[]!C.[rrrr]!C.[rrd]!C *+<8pt>[F-,]\frm{}\restore
&& \cdots && *+<6pt>[F]{p_{3}}\ar[dll] & \cdots & i_r\ar@/^2pt/[dllll] 
&& 
&&&&
&& 
i_1\ar@/_2pt/[drrrr] & \cdots &&& \cdots &&& \cdots & i_r\ar@/^2pt/[dllll] 
\\
&&&& *+<6pt>[F]{p_{1}}\ar[d] &&&& 
&& 
&&&&
&& 
&&&& *+<6pt>[F]{p_{1 2 3}}\ar[d] &&&& 
\\
&&&& 0 &&&& 
&& 
&&&&
&& 
&&&& 0 &&&& 
\\ &&&& &&&& && &&&& && &&&& &&&& \\
&& 
\save [].[rrrrrrrrddd]!C="g5"\restore 
& j_1\ar[dr] & \cdots & j_s\ar[dl] && k_1\ar[dr] & \cdots & k_t\ar[dl] & 
&&&&
\save [].[rrrrrrrrddd]!C="g6"\restore 
& j_1\ar[dr] & \cdots & j_s\ar[dl] && k_1\ar[ddl] & \cdots & k_t\ar[ddlll] & 
&& 
\\
&& 
i_1\ar@/_6pt/[drrrr] & \cdots & *+<6pt>[F]{p_{2}}\ar[drr]\save[]!C *+<8pt>[F-,]\frm{}\restore
&& \cdots && *+<6pt>[F]{p_{3}}\ar[dll]\save[]!C.[lld]!C *+<8pt>[F-,]\frm{}\restore
& \cdots & i_r\ar@/^2pt/[dllll] 
&&&&
i_1\ar@/_2pt/[drrrr] & \cdots & *+<6pt>[F]{p_{2}}\ar[drr]\save[]!C.[drr]!C *+<8pt>[F-,]\frm{}\restore
&& \cdots && & \cdots & i_r\ar@/^2pt/[dllll] 
&& 
\\
&& 
&&&& *+<6pt>[F]{p_{1}}\ar[d] &&&& 
&&&&
&&&& *+<6pt>[F]{p_{1 3}}\ar[d] &&&& 
&& 
\\
&& 
&&&& 0 &&&& 
&&&&
&&&& 0 &&&& 
&& 
\ar@/^8pt/ "g1"!U+<0pt,6pt>;"g2"^(0.2){\simeq}
\ar "g2"!R+<6pt,0pt>;"g3"!L-<6pt,0pt>^{\lambda_*}
\ar@/^8pt/ "g3";"g4"!U+<0pt,6pt>^(0.75){\lambda_*}
\ar@/_8pt/ "g1"!D-<0pt,6pt>;"g5"_(0.2){\simeq}
\ar "g5"!R+<6pt,0pt>;"g6"!L-<6pt,0pt>_{\lambda_*}
\ar@/_8pt/ "g6";"g4"!D-<0pt,6pt>_(0.75){\lambda_*}
\ar "g1"!R+<6pt,0pt>;"g4"!L-<6pt,0pt>^{\lambda_*}
}
\end{equation*}
\caption{}\label{Fig:TreeRamifiedAssociativity}\end{figure}

The composites over trees with two vertices
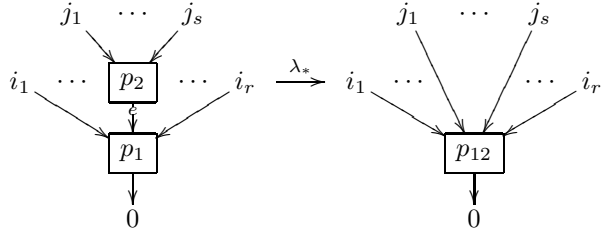
\begin{equation*}
\vcenter{\xymatrix@W=2pt@H=4pt@R=12pt@C=6pt@M=2pt{ & j_1\ar[dr] & \cdots & j_s\ar[dl] & 
&&&
& j_1\ar[ddr] & \cdots & j_s\ar[ddl] & 
\\
i_1\ar[drr] & \cdots & *+<3mm>[F]{p_2}\ar[d]|(0.4){e} & \cdots & i_r\ar[dll] 
\ar@[]!R+<6pt,0pt>;[rrr]!L-<6pt,0pt>^{\lambda_*} &&&
i_1\ar[drr] & \cdots && \cdots & i_r\ar[dll] 
\\
&& *+<3mm>[F]{p_1}\ar[d] &&
&&&
&& *+<3mm>[F]{p_{1 2}}\ar[d] && \\
&& 0 &&
&&&
&& 0 && }}
\end{equation*}
provide an algebra over the monad $\tilde{\FOp}$
with composition operations
\begin{equation*}
\circ_e: \tilde{\POp}(I)\otimes\tilde{\POp}(J)\rightarrow\tilde{\POp}(I\setminus\{e\}\cup J).
\end{equation*}
In the formalism of~\S\ref{OperadDefinition},
these composition operations, which arise from the structure of an algebra over $\tilde{\FOp}$,
represent partial composites $w\cdot p\circ_e q$,
where $w$ is a permutation which corresponds to the sharing of indices.
Note that:

\begin{obsv}\label{TreeOperadStructures:AssociativityInterpretation}
The associativity relations of figures (\ref{Fig:TreeLinearAssociativity}-\ref{Fig:TreeRamifiedAssociativity})
are equivalent to the associativity relations
of the partial composition products of an augmented operad.
\end{obsv}

Thus the structure of an algebra over $\tilde{\FOp}$
includes partial composition operations of an augmented operad structure.
In a converse direction:

\begin{lemm}[{see~\cite[\S 1.2]{GinzburgKapranov}}]\label{TreeOperadStructures:TreeDecompositionImplication}
Let $\tilde{\POp}$ be any $\Sigma_*$-object equipped with partial composition operations
$\circ_e: \tilde{\POp}(r)\otimes\tilde{\POp}(s)\rightarrow\tilde{\POp}(r+s-1)$, $e = 1,\dots,r$.
\begin{enumerate}
\item
If we assume the equivariance axioms of partial composites of augmented operads,
then these operations amount to composites $\lambda_{\tau}: \tau(\tilde{\POp})\rightarrow\tilde{\POp}(I)$,
where $\tau\in\Theta_2(I)$ ranges over trees with two vertices.
\item
If the associativity axioms of partial composites of augmented operads hold in $\tilde{\POp}$,
then we have a unique collection of tree composition products $\lambda_{\tau}: \tau(\tilde{\POp})\rightarrow\tilde{\POp}(I)$, $\tau\in\Theta(I)$,
extending these two-vertex tree composites
and such that properties (\ref{UnitTreeComposites}-\ref{AssociativityTreeComposites})
of~\S\ref{TreeOperadStructures:ReducedFreeOperad} hold.
\end{enumerate}
\end{lemm}

\begin{proof}[Proof (sketch)]
The associativity properties of~\S\ref{TreeOperadStructures:ReducedFreeOperad}
imply that the composition product over any tree $\tau$
is determined by a decomposition into composition products over trees with two vertices.

The composition product over a subtree with two vertices
is equivalent to the contraction of an internal edge of~$\tau$.
The decompositions of a tree composition product $\lambda_{\tau}: \tau(\tilde{\POp})\rightarrow\tilde{\POp}(I)$
into composition products over trees with two vertices
are associated to orderings of the internal edges of~$\tau$.
Observe that all orderings yield the same composition product $\lambda_{\tau}: \tau(\tilde{\POp})\rightarrow\tilde{\POp}(I)$
whenever we have the associativity relations of figures~(\ref{Fig:TreeLinearAssociativity}-\ref{Fig:TreeRamifiedAssociativity})
for three-fold composites.

From this coherence property,
we deduce readily that the decomposition process provides the object $\tilde{\POp}$
with a well-defined and uniquely-determined structure of an algebra over $\tilde{\FOp}$.
\end{proof}

To summarize,
observation~\ref{TreeOperadStructures:AssociativityInterpretation} and lemma~\ref{TreeOperadStructures:TreeDecompositionImplication}
give:

\begin{prop}[{see~\cite[\S 1.2]{GinzburgKapranov} and~\cite[Theorem 40]{MarklSurvey}}]\label{TreeOperadStructures:PartialCompositeEquivalenceStatement}
The structure of an algebra over $\tilde{\FOp}$
amounts to a collection of operations
\begin{equation*}
\circ_e: \tilde{\POp}(r)\otimes\tilde{\POp}(s)\rightarrow\tilde{\POp}(r+s-1),\quad\text{$e = 1,\dots,r$},
\end{equation*}
such that the equivariance and associativity relations
of operadic partial composites hold.\qed
\end{prop}

In the next paragraphs,
we check that the structure morphisms
\begin{equation*}
\lambda_{\tau}: \tau(\tilde{\POp})\rightarrow\tilde{\POp}(I)
\end{equation*}
include all components of an operadic composition product $\mu: \POp\circ\POp\rightarrow\POp$ as natural summands,
and not only partial composites.
This approach is used in~\S\ref{CylinderOperads} (in a dual context)
for the definition of the explicit cylinder object
associated to an operad.

To begin with,
we rephrase the definition of the composition product $\mu: \POp\circ\POp\rightarrow\POp$ of an augmented operad $\POp$
in terms of a composition structure on the augmentation ideal of~$\POp$.

\subsubsection{Trees of height $2$}\label{TreeOperadStructures:HeightTwoTrees}
Define the height of a vertex $v$ in a tree $\tau$ as the length $l$
of the unique chain of edges
\begin{equation*}
v\xrightarrow{e_1} v_1\xrightarrow{e_2}\cdots\xrightarrow{e_{l-1}} v_{l-1}\xrightarrow{e_l} 0
\end{equation*}
which connects the vertex to the output of the tree.
Define the height of a tree $\tau$
as the maximal height of vertices $v\in V(\tau)$.
By definition of a tree structure,
the source of the outgoing edge of a tree $\tau$
is the unique vertex of height $1$
in $\tau$.

The set of $I$-trees of height $h$, for which we adopt the notation $\Psi_h(I)$,
forms clearly a subgroupoid of the groupoid of $I$-trees $\Theta(I)$.
These subgroupoids $\Psi_h(I)$ are also preserved by the action of bijections $u: I\rightarrow J$.

In~\S\ref{OperadDefinition:LevelTrees},
we define composite $\Sigma_*$-objects as dg-modules of tensors arranged on trees
with two levels of vertices.
The structure of an $I$-tree with two levels, as defined in~\S\ref{OperadDefinition:LevelTrees},
is equivalent to an $I$-tree $\tau$
equipped with a decomposition $V(\tau) = V_0(\tau)\amalg V_1(\tau)$ of the set of vertices $V(\tau)$
such that $V_0(\tau)$
is reduced to the source $v_0 = s(e)$ of the outgoing edge of the tree (the edge $e$ such that $t(e) = 0$),
the inputs of $v_0$ are given by the vertex set $I_{v_0} = V_1(\tau)$
and we have $I_v\subset I$ for every $v\in V_1(\tau)$.
The component $V_i(\tau)$ represents the subset of vertices at level $i = 0,1$.
Each vertex $v\in V_1(\tau)$
has height $2$.
Finally,
a tree with two levels of vertices can be identified with a tree of height $2$
such that every ingoing edge targets to a vertex $v$
of height $2$.

\subsubsection{The composition structure of augmented operads}\label{TreeOperadStructures:ReducedCompositionStructure}
Recall that we adopt the notation $\Lev_2(I)$
for the groupoid of $I$-trees with two levels of vertices.
According to the definition of~\S\ref{OperadDefinition:LevelTrees},
we have
\begin{equation*}
\POp\circ\POp(I) = \bigoplus_{\xi\in\Lev_2(I)} \xi(\POp,\POp)/\equiv.
\end{equation*}

Any $I$-tree $\tau$ of height $2$
has a completion $\widehat{\tau}\in\Lev_2(I)$
defined by the insertion of unital vertices $\xymatrix@H=6pt@W=3pt@M=2pt@!R=1pt@!C=1pt{ \ar[r] & *+<3mm>[o][F]{1}\ar[r] & }$
on edges $e$ going directly from a tree input $s(e) = i$ to the vertex of height $1$ of $\tau$.
A tree tensor $\varpi\in\tau(\tilde{\POp})$
is naturally associated to a two level tree tensor $\widehat{\varpi}\in\widehat{\tau}(\POp,\POp)$
in the sense of~\S\ref{OperadDefinition:LevelTrees}:
just label the added vertices $\xymatrix@H=6pt@W=3pt@M=2pt@!R=1pt@!C=1pt{ \ar[r] & *+<3mm>[o][F]{1}\ar[r] & }$
by unit elements $1\in\IOp(1)$.

Graphically,
we have a morphism $\tau(\tilde{\POp})\rightarrow\widehat{\tau}(\POp,\POp)$
which associates the two level tree tensors of the form
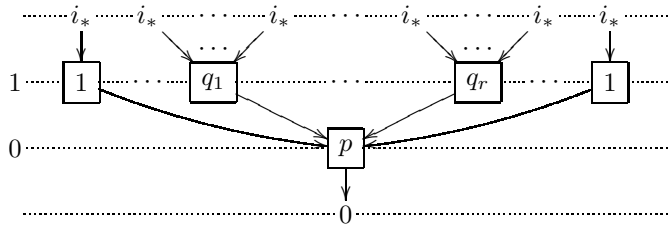
\begin{equation*}
\vcenter{\xymatrix@H=6pt@W=3pt@M=2pt@!R=1pt@!C=1pt{ \ar@{.}[r] & i_*\ar[d]\ar@{.}[r] &
i_*\ar[dr]\ar@{.}[r] & \cdots\ar@{.}[r]\ar@{}[d]|{\displaystyle{\cdots}} & i_*\ar[dl]\ar@{.}[r] &
\cdots\ar@{.}[r]
& i_*\ar[dr]\ar@{.}[r] & \cdots\ar@{.}[r]\ar@{}[d]|{\displaystyle{\cdots}} & i_*\ar[dl]\ar@{.}[r] &
i_*\ar[d]\ar@{.}[r] & \\
1\ar@{.}[r] & *+<3mm>[F]{1}\ar@/_4pt/[drrrr]\ar@{.}[r] & \cdots\ar@{.}[r]
& *+<3mm>[F]{q_1}\ar[drr]\ar@{.}[rr] &
& \cdots\ar@{.}[rr] &
& *+<3mm>[F]{q_r}\ar[dll]\ar@{.}[r] & \cdots\ar@{.}[r]
& *+<3mm>[F]{1}\ar@/^4pt/[dllll]\ar@{.}[r] & \\
0\ar@{.}[rrrrr] &&&&& *+<3mm>[F]{p}\ar[d]\ar@{.}[rrrrr] &&&&& \\
\ar@{.}[rrrrr] &&&&& 0\ar@{.}[rrrrr] &&&&& }},
\end{equation*}
to the tree tensors
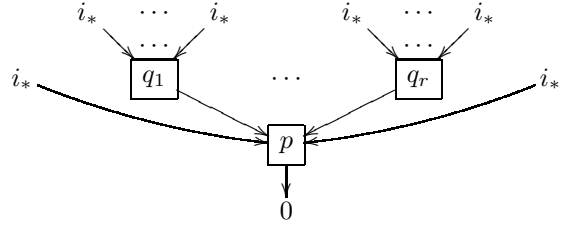
\begin{equation*}
\vcenter{\xymatrix@H=6pt@W=3pt@M=2pt@!R=1pt@!C=1pt{ &
i_*\ar[dr] & \cdots\ar@{}[d]|{\displaystyle{\cdots}} & i_*\ar[dl] &
& i_*\ar[dr] & \cdots\ar@{}[d]|{\displaystyle{\cdots}} & i_*\ar[dl] & \\
i_*\ar@/_4pt/[drrrr] &
& *+<3mm>[F]{q_1}\ar[drr] &
& \cdots &
& *+<3mm>[F]{q_r}\ar[dll] &
& i_*\ar@/^4pt/[dllll] \\
&&&& *+<3mm>[F]{p}\ar[d] &&&& \\ &&&& 0 &&&& }},
\end{equation*}
where $p\in\tilde{\POp}$ and $q_1,\dots,q_r\in\tilde{\POp}$.

Note that a tree of height $2$
has at least one vertex of height $2$.
Accordingly,
for an augmented operad $\POp$
equipped with a splitting $\POp = \IOp\oplus\tilde{\POp}$,
we have an identity
\begin{equation*}
\POp\circ\POp(I) = \downarrow^{\sharp}(\IOp)(I)\oplus\psi^{\flat}(\tilde{\POp})\oplus\psi^{\sharp}(\tilde{\POp})
\oplus\Bigl\{\bigoplus_{\tau\in\Psi_2(I)} \tau(\tilde{\POp})/\equiv\Bigr\},
\end{equation*}
where:
\begin{itemize}
\item
the summand $\downarrow^{\sharp}(\IOp)(I)\simeq\IOp(I)$
consists of two level tree tensors of the form
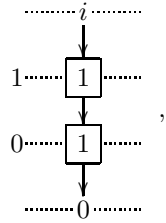
\begin{equation*}
\vcenter{\xymatrix@H=6pt@W=3pt@M=2pt@!R=1pt@!C=1pt{ \ar@{.}[r] & i\ar[d]\ar@{.}[r] & \\
*+<2pt>{1}\ar@{.}[r] & *+<8pt>[F]{1}\ar[d]\ar@{.}[r] & \\
*+<2pt>{0}\ar@{.}[r] & *+<8pt>[F]{1}\ar[d]\ar@{.}[r] & \\ \ar@{.}[r] & 0\ar@{.}[r] & }},
\end{equation*}
\item
the summand $\psi^{\flat}(\tilde{\POp})(I)\simeq\tilde{\POp}(I)$
consists of two level tree tensors of the form
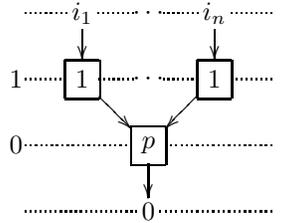
\begin{equation*}
\vcenter{\xymatrix@H=6pt@W=3pt@M=2pt@!R=1pt@!C=1pt{ \ar@{.}[r] & i_1\ar[d]\ar@{.}[r] & *{\cdots}\ar@{.}[r] & i_n\ar[d]\ar@{.}[r] & \\
*+<2pt>{1}\ar@{.}[r] & *+<8pt>[F]{1}\ar[dr]\ar@{.}[r] & *{\cdots}\ar@{.}[r] & *+<8pt>[F]{1}\ar[dl]\ar@{.}[r] & \\
*+<2pt>{0}\ar@{.}[rr] && *+<8pt>[F]{p}\ar[d]\ar@{.}[rr] && \\ \ar@{.}[rr] && 0\ar@{.}[rr] && }},
\end{equation*}
for an element $p\in\tilde{\POp}(I)$,
\item
the summand $\psi^{\sharp}(\tilde{\POp})(I)\simeq\tilde{\POp}(I)$
consists of two level tree tensors of the form
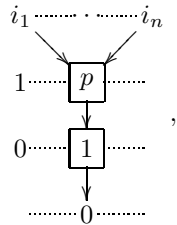
\begin{equation*}
\vcenter{\xymatrix@H=6pt@W=3pt@M=2pt@!R=1pt@!C=1pt{ i_1\ar[dr]\ar@{.}[r] & *{\cdots}\ar@{.}[r] & i_n\ar[dl] \\
*+<2pt>{1}\ar@{.}[r] & *+<8pt>[F]{p}\ar[d]\ar@{.}[r] & \\
*+<2pt>{0}\ar@{.}[r] & *+<8pt>[F]{1}\ar[d]\ar@{.}[r] & \\ \ar@{.}[r] & 0\ar@{.}[r] & }},
\end{equation*}
for an element $p\in\tilde{\POp}(I)$,
\item
and the other summands are associated to $I$-trees of height two $\tau\in\Psi_2(I)$.
\end{itemize}

The summand $\psi^{\flat}(\tilde{\POp})$
is identified with the image of the composite $\Sigma_*$-object $\tilde{\POp}\simeq\tilde{\POp}\circ I$ in $\POp\circ\POp$,
the summand $\psi^{\sharp}(\tilde{\POp})$ with the image of~$\tilde{\POp}\simeq I\circ\tilde{\POp}$,
and $\downarrow^{\sharp}(\IOp)$ with the image of~$I\circ I$.
The composition product $\mu: \POp\circ\POp\rightarrow\POp$ is determined on these summands by the unit axiom of operads.
Therefore,
we obtain that the composition structure of an augmented operad $\POp$
can be determined by a collection of morphisms $\lambda: \tau(\tilde{\POp})\rightarrow\tilde{\POp}(I)$,
included in the structure of an algebra over the monad~$\tilde{\FOp}$,
where $\tau$ runs over trees of height $2$.

We have moreover:

\begin{lemm}
The tree composition products $\lambda: \tau(\tilde{\POp})\rightarrow\tilde{\POp}(I)$, $\tau\in\Psi_2(I)$,
included in the monad action $\lambda: \tilde{\FOp}(\tilde{\POp})\rightarrow\tilde{\POp}$,
determine an operadic composition product $\mu: \POp\circ\POp\rightarrow\POp$
for which the associativity property of~\S\ref{TreeOperadStructures:ReducedFreeOperad}
holds.
\end{lemm}

\begin{proof}[Proof (sketch)]
The elements of the composite $\POp\circ\POp\circ\POp$
can be represented by tensors arranged on trees with three levels $\xi\in\Lev_3(I)$.
The associativity of~$\mu$
amounts to the identity of morphisms $\xi(\POp,\POp,\POp)\rightrightarrows\POp(I)$,
where $\xi(\POp,\POp,\POp)$ is the module of tree tensors associated to $\xi\in\Lev_3(I)$.

The splitting of~\S\ref{TreeOperadStructures:ReducedCompositionStructure} can be extended to the modules of tree tensors $\xi(\POp,\POp,\POp)$
and the associativity of the composition product $\mu: \POp\circ\POp\rightarrow\POp$
reduces to the associativity property of~\S\ref{TreeOperadStructures:ReducedFreeOperad}.
\end{proof}

Hence:

\begin{prop}\label{TreeOperadStructures:CompositionProductInclusion}
The structure of an algebra over~$\tilde{\FOp}$ includes an augmented operad structure.\qed
\end{prop}

Note that we retrieve the representation of partial composition products
if we apply the construction of~\S\ref{TreeOperadStructures:ReducedCompositionStructure}
to trees with two vertices (which belong to the set of trees of height $2$).
Hence,
the equivalence of proposition~\ref{TreeOperadStructures:PartialCompositeEquivalenceStatement}
and the implication of proposition~\ref{TreeOperadStructures:CompositionProductInclusion}
give finally:

\begin{prop}[{see~\cite[Theorem 40]{MarklSurvey} or~\cite[Proposition 1.12]{GetzlerJones}}]\label{TreeOperadStructures:CompositionProductEquivalence}
We have an equivalence between the category of algebras over~$\tilde{\FOp}$
and the category of augmented operads.\qed
\end{prop}

This proposition is established by other arguments in~\cite{MarklSurvey}.

\subsubsection{Connected operads}\label{TreeOperadStructures:ConnectedOperads}
To ensure good homotopical properties,
we restrict certain operadic constructions to operads $\POp$ such that $\POp(0) = 0$ and $\POp(1) = \kk$.
Such operads are called connected.
The category of connected operads is denoted by $\Op_1$.

Any connected operad comes equipped with an augmentation $\epsilon: \POp\rightarrow \IOp$,
simply given by the identity of~$\POp(1) = \IOp(1) = \kk$ in arity $1$,
and a morphism of connected operads commutes automatically with this augmentation.
Obviously,
the augmentation ideal of a connected operad is defined by:
\begin{equation*}
\tilde{\POp}(n) = \begin{cases} 0, & \text{if $n = 0,1$}, \\
\POp(n), & \text{otherwise}. \end{cases}
\end{equation*}

Let $\M_1$ be the category formed the category of $\Sigma_*$-objects $M$
such that $M(0) = M(1) = 0$ (we say that $M$ is reduced).
The free operad $\FOp(M)$ associated to a $\Sigma_*$-object $M$
is connected if and only if $M\in\M_1$.
The restriction of the free operad to $\M_1$
defines clearly a left adjoint of the functor $(-)^{\sim}: \Op_1\rightarrow\M_1$
which maps a connected operad $\POp$ to its augmentation ideal $\tilde{\POp}$.
Moreover,
the augmentation ideal of the free operad $\tilde{\FOp}: \M\rightarrow\M$
restricts clearly to a monad on $\M_1$
so that the category of algebras over $\tilde{\FOp}$
is equivalent to the category of connected operads.

\subsection{The homotopy of operads in dg-modules}\label{OperadHomotopy}
The purpose of this subsection is to recall the definition of the model structure of operads in dg-modules.

The applications of homotopical algebra to operads in dg-modules
go back to~\cite{GetzlerJones}.
In this reference,
the authors deal with operads and algebras in non-negatively graded dg-modules
over a field of characteristic zero.
The generalization of the definition of~\cite[\S 4]{GetzlerJones}
in the context of unbounded dg-modules over a ring
is addressed in~\cite{HinichHomotopy}.
In~\cite{BergerMoerdijkModel,Spitzweck},
the authors study applications of homotopical algebra to operads in the axiomatic setting of symmetric monoidal model categories,
which include the example of dg-module categories.

The articles~\cite{BergerMoerdijkModel,GetzlerJones,HinichHomotopy,Spitzweck}
can be used as references for the recollections of this subsection.
For applications of model structures to algebras over operads,
we also refer to the book~\cite{FresseModuleBook}.

\subsubsection{Adjoint model structures}\label{OperadHomotopy:AdjointModelStructures}
The model structure of the category of operads
is an instance of a model structure
produced by adjunction from a well-defined model category.

In general,
we have an adjunction $F: \X\rightleftarrows\A :U$,
such that $\X$ is a cofibrantly generated model category
and $\A$ is a category equipped with colimits and limits.
Define classes of weak-equivalence, cofibrations and fibrations
in~$\A$ by:
\begin{itemize}
\item
the weak-equivalences, respectively fibrations, are the morphisms~$f\in\Mor\A$
which are mapped to weak-equivalences, respectively fibrations, by the functor $U: \A\rightarrow\X$
(we say that the functor $U: \A\rightarrow\X$ creates weak-equivalences and fibrations in $\A$);
\item
the cofibrations are the morphisms which have the right lifting property with respect to acyclic fibrations.
\end{itemize}

In good cases,
this definition provides the category $\A$ with a well-defined model structure,
for which all axioms hold without restriction.
In less good cases,
the axioms of model categories hold provided that we restrict applications
of the lifting and factorization axioms to morphisms
with a cofibrant domain.
In this situation,
one says that $\A$ forms a semi-model category (see~\cite{HoveySemiModel}).
The axioms of semi-model categories
are enough for most constructions of homotopical algebra.

In all cases,
the morphisms $F(i): F(C)\rightarrow F(D)$
such that $i$ ranges over the generating (acyclic) cofibrations of~$\X$
define a set of generating (acyclic) cofibrations in~$\A$.

\subsubsection{The model structures of $\Sigma_*$-objects and operads}\label{OperadHomotopy:SigmaObjectModelStructure}
To define the model structure of the category of operads,
we use a composite adjunction
\begin{equation*}
\xymatrix{ \C^{\NN}\ar@<+2pt>[r]^{\Sigma_*\otimes-} & \M\ar@<+2pt>[r]^{\FOp}\ar@<+2pt>[l]^{V} & \Op\ar@<+2pt>[l]^{U} },
\end{equation*}
where $\C^{\NN}$ refers to the category of $\NN$-graded collections of dg-modules.

The functor $V: \M\rightarrow\C^{\NN}$
simply forgets the action of symmetric groups
and maps a $\Sigma_*$-object $M$
to the underlying collection of dg-modules $\{M(n)\}_{n\in\NN}\in\C^{\NN}$.
The functor $\Sigma_*\otimes-: \C^{\NN}\rightarrow\M$,
adjoint to $V$,
maps a collection $K\in\C^{\NN}$
to the $\Sigma_*$-object such that $(\Sigma_*\otimes K)(n) = \Sigma_n\otimes K(n)$.
The functors $\FOp: \M\rightleftarrows\Op :U$
are the free operad and forgetful functors,
whose definition is reviewed in~\S\S\ref{OperadDefinition}-\ref{TreeOperadStructures}.

The category $\C^{\NN}$ has an obvious model structure for which a morphism $f: K\rightarrow L$
is a weak-equivalence (respectively, cofibration, fibration)
if and only if all its components $f: K(n)\rightarrow L(n)$, $n\in\NN$,
are weak-equivalences (respectively, cofibrations, fibrations) of dg-modules.
Moreover,
we have an obvious set of generating (acyclic) cofibrations
in $\C^{\NN}$.

According to~\cite{HinichHomotopy,Spitzweck} (see also~\cite[\S 11.4, \S 12.2]{FresseModuleBook}),
the application of the construction of~\S\ref{OperadHomotopy:AdjointModelStructures}
to the adjunctions $\C^{\NN}\rightleftarrows\M\rightleftarrows\Op$
gives the following results:

\begin{fact}[{see for instance~\cite[\S 11.4]{FresseModuleBook}}]
The category of $\Sigma_*$-objects $\M$
inherits a full model structure
such that the forgetful functor $V: \M\rightarrow\C^{\NN}$
creates weak-equivalences and fibrations.
\end{fact}

\begin{fact}[see~\cite{HinichHomotopy,Spitzweck}]
The category of operads~$\Op$
inherits a semi-model structure
such that the forgetful functor $U: \Op\rightarrow\M$
creates weak-equivalences and fibrations.
\end{fact}

Usually,
we say that an operad morphism $f: \POp\rightarrow\QOp$
is a $\Sigma_*$-cofibration
if its image under the forgetful functor $U: \Op\rightarrow\M$
defines a cofibration in the category of $\Sigma_*$-objects
and an operad $\POp$ is $\Sigma_*$-cofibrant
if the initial morphism $\eta: \IOp\rightarrow\POp$
is a $\Sigma_*$-cofibration.
Similarly,
we say that a morphism of $\Sigma_*$-objects  $f: \POp\rightarrow\QOp$
is a $\C$-cofibration
if its image under the forgetful functor $V: \M\rightarrow\C^{\NN}$
defines a cofibration in $\C^{\NN}$.

According to~\cite{Spitzweck},
the forgetful functor $U: \Op\rightarrow\M$ maps cofibrations with a $\Sigma_*$-cofibrant domain to cofibrations.
Hence,
any cofibrant operad is automatically $\Sigma_*$-cofibrant.
The forgetful functor $V: \M\rightarrow\C^{\NN}$
preserves cofibrations too.

\subsubsection{The restriction of model structures to the subcategory of connected operads}\label{OperadHomotopy:ConnectedOperads}
The construction of~\S\ref{OperadHomotopy:AdjointModelStructures}
can also be applied
to the adjunction
\begin{equation*}
\FOp: \M_1\rightleftarrows\Op_1 :U
\end{equation*}
between reduced $\Sigma_*$-objects and connected operads.
In fact,
the category of connected operads $\Op_1$
inherits the weak-equivalences (respectively, cofibrations, fibrations) $f\in\Mor\Op_1$
such that $f$ forms a weak-equivalence (respectively, a cofibration, a fibration) in $\Op$.
The category of reduced $\Sigma_*$-objects itself $\M_1$
forms naturally a model subcategory of the category of $\Sigma_*$-objects $\M$.

The whole axioms of operads hold in the subcategory of connected operads (see~\cite{BergerMoerdijkModel}).

\subsubsection{Generating cofibrations}\label{OperadHomotopy:GeneratingCofibrations}
For the sake of completeness,
we review quickly the definition of generating (acyclic) cofibrations
in $\C^{\NN}$, in the category of $\Sigma_*$-objects $\M$
and in the category of operads $\Op$.

The category of collections $\C^{\NN}$ and the category of $\Sigma_*$-objects $\M$
are naturally tensored over the category of dg-modules.
In both cases,
the tensor product of an object $M$ with a dg-module $C$
is defined by the obvious formula
\begin{equation*}
(C\otimes M)(n) = C\otimes M(n),\quad\text{for $n\in\NN$}.
\end{equation*}
The objects $G_r\in\C^{\NN}$
such that
\begin{equation*}
G_r(n) = \begin{cases} \kk, & \text{if $n = r$}, \\ 0, & \text{otherwise}, \end{cases}
\end{equation*}
are in an enriched sense generators of the category of collections $\C^{\NN}$.
The associated $\Sigma_*$-objects $\Sigma_*\otimes G_r$,
also defined by the formula
\begin{equation*}
(\Sigma_*\otimes G_r)(n) = \begin{cases} \kk[\Sigma_r], & \text{if $n = r$}, \\ 0, & \text{otherwise}, \end{cases}
\end{equation*}
where $\kk[\Sigma_r]$ represents the regular representation of $\Sigma_r$,
are generators of~$\M$.

The category~$\C^{\NN}$
is equipped with generating (acyclic) cofibrations
of the form
\begin{equation*}
i\otimes G_r: C\otimes G_r\rightarrow D\otimes G_r,
\end{equation*}
where $i$ range over generating (acyclic) cofibrations of dg-modules.
By construction,
the morphisms of $\Sigma_*$-objects $\Sigma_*\otimes(i\otimes G_r)$
associated to these generating (acyclic) cofibrations $i\otimes G_r$
are the generating (acyclic) cofibrations of the model category of $\Sigma_*$-objects.
The morphisms of free operads $\FOp(\Sigma_*\otimes(i\otimes G_r))$
induced by these morphisms $\Sigma_*\otimes(i\otimes G_r)$
form the generating (acyclic) cofibrations of the semi-model category of operads.

The functor $\Sigma_*\otimes -$
commutes clearly with the tensor product over dg-modules.
Hence,
we also have an identity $\Sigma_*\otimes(i\otimes G_r) = i\otimes(\Sigma_*\otimes G_r)$
for the generating (acyclic) cofibrations of the model category of $\Sigma_*$-objects.

\subsubsection{Cofibrant cell complexes of operads}\label{OperadHomotopy:CofibrantCellComplexes}
In the next subsection,
we review the notion of a quasi-free operad
which gives an effective representation of the structure of cofibrant cell operads
in dg-modules.

For technical reasons,
we consider an extended notion of cofibrant cell complex in the context of operads:
in our constructions,
we get morphisms $j: \POp\rightarrow\QOp$
which splits into composites
\begin{equation*}
\POp = \QOp_0\xrightarrow{j_1}\dots
\rightarrow\QOp_{\lambda-1}\xrightarrow{j_{\lambda}}\QOp_{\lambda}\rightarrow\dots
\rightarrow\colim_{\lambda}\QOp_{\lambda} = \QOp
\end{equation*}
of pushouts
\begin{equation*}
\xymatrix{ \FOp(M_{\lambda})\ar[r]\ar[d]_{\FOp(i_{\lambda})} & \QOp_{\lambda-1}\ar@{.>}[d]^{j_{\lambda}} \\
\FOp(N_{\lambda})\ar@{.>}[r] & \QOp_{\lambda} }
\end{equation*}
such that $i_{\lambda}: M_{\lambda}\rightarrow N_{\lambda}$ is any cofibration of $\Sigma_*$-objects
and not necessarily a direct sum of generating cofibrations.
Recall that the class of cofibrations in a semi-model category is closed under pushouts and composites.
Observe that any morphism of free operads $\FOp(i): \FOp(M)\rightarrow \FOp(N)$
induced by a cofibration of $\Sigma_*$-objects $i: M\rightarrow N$
forms a cofibration in the category of operads (immediate by adjunction)
to conclude that these generalized cell structures also define cofibrations of operads.

Similar observations hold for acyclic cofibrations.

\subsection{Quasi-free operads}\label{QuasiFreeOperads}
In short,
a quasi-free operad is a dg-operad $\QOp$
defined by the addition of a twisting homomorphism $\partial: \FOp(M)\rightarrow \FOp(M)$
to the natural differential of a free operad $\FOp(M)$.
The first purpose of this section is to review an explicit construction
of the twisting homomorphisms of quasi-free operads.
Then we prove that the quasi-free operads $\QOp = (\FOp(M),\partial)$
generated by a reduced and cofibrant $\Sigma_*$-object $M$
form cofibrant objects of the category of operads.

For our needs,
we also study morphisms of quasi-free operads $\phi_f: (\FOp(M),\partial)\rightarrow(\FOp(N),\partial)$
induced by morphisms of reduced $\Sigma_*$-objects $f: M\rightarrow N$.
We prove that $\phi_f$ is a cofibration (respectively, an acyclic cofibration) of operads
if $f$ is a cofibration (respectively, an acyclic cofibration) of $\Sigma_*$-objects.

\subsubsection{Homomorphisms of $\Sigma_*$-objects}
The category of~$\Sigma_*$-objects
is naturally enriched over the category of dg-modules,
with as homomorphisms $f\in\Hom_{\M}(M,N)$
the collections of homomorphisms of dg-modules
\begin{equation*}
f\in\Hom_{\C}(M(n),N(n)),\quad\text{$n\in\NN$},
\end{equation*}
that commute with the action of permutations.
The differential of a homomorphism in~$\Hom_{\M}(M,N)$
is defined termwise by the differential of the dg-homs $\Hom_{\C}(M(n),N(n))$, $n\in\NN$.

Obviously,
a morphism of~$\Sigma_*$-objects
is a homomorphism of degree $0$ such that $\delta(f) = 0$.
Homomorphisms which are not morphisms
occur naturally in the constructions of this subsection.

\subsubsection{Derivations of operads, twisted dg-operads and quasi-free operads}\label{QuasiFreeOperads:TwistingDerivations}
One says that a collection of homogeneous homomorphisms
\begin{equation*}
\theta\in\Hom_{\C}(\POp(n),\POp(n)),\quad\text{$n\in\NN$},
\end{equation*}
defines an operad derivation
if these homomorphisms commute with the action of permutations on $\POp(n)$
and satisfy the derivation relation
\begin{equation}
\theta(p(q_1,\dots,q_r)) = \theta(p)(q_1,\dots,q_r) + \sum_{i=1}^{r} p(q_1,\dots,\theta(q_i),\dots,q_r)
\end{equation}
for any composite $p(q_1,\dots,q_r))\in\POp$.

By definition,
the composition products of an operad in dg-modules
$\mu: \POp(r)\otimes\POp(n_1)\otimes\dots\otimes\POp(n_r)\rightarrow\POp(n_1+\dots+n_r)$
are morphisms of dg-modules.
This assumption amounts to the requirement that the differentials of the dg-modules~$\POp(n)$
satisfy the derivation relation (*).
The internal differential of each~$\POp(n)$
is also assumed to commute with the action of permutations $\POp(n)$.
Hence,
the differentials of an operad consist of homomorphisms
of degree $-1$
\begin{equation*}
\delta: \POp(n)\rightarrow\POp(n)
\end{equation*}
which satisfy the identity $\delta^2 = 0$ and form an operad derivation.

Clearly,
the differentials $\delta+\partial: \POp(n)\rightarrow\POp(n)$
obtained by the addition of twisting homomorphisms $\partial\in\Hom_{\C}(\POp(n),\POp(n))$
to the internal differential of each~$\POp(n)$
define an operad differential
if and only if $\partial$ is an operad derivation.
Thus,
we obtain that a collection of twisted dg-modules $(\POp,\partial) = \{(\POp(n),\partial)\}_{n\in\NN}$
inherits an operad structure from $\POp$
if and only if the twisting homomorphisms $\partial$ form an operad derivation.

A quasi-free operad is a twisted dg-operad of the form $\QOp = (\FOp(M),\partial)$,
where $\POp = \FOp(M)$ is a free operad in dg-modules.

\subsubsection{Derivation of free operads}\label{QuasiFreeOperads:FreeOperadDerivations}
The derivation relation of~\S\ref{QuasiFreeOperads:TwistingDerivations}
implies that any derivation $\partial: \FOp(M)\rightarrow \FOp(M)$
on a free operad is determined by its restriction to~$M$
since a free operad $\FOp(M)$
is spanned by formal composites of elements of~$M$.

The purpose of this paragraph is to review the construction of a derivation~$\partial_{\alpha}: \FOp(M)\rightarrow \FOp(M)$
such that $\partial_\alpha|_M = \alpha$
for a given homomorphism of $\Sigma_*$-objects $\alpha: M\rightarrow \FOp(M)$.
Since $\FOp(M) = \bigoplus_{\sigma} \sigma(M)/\equiv$,
the homomorphism $\alpha: M\rightarrow \FOp(M)$
splits into a sum of homomorphisms $\alpha_\sigma: M(I)\rightarrow\sigma(M)$,
naturally associated to each $I$-tree $\sigma\in\Theta(I)$.

The derivation~$\partial_{\alpha}: \FOp(M)\rightarrow \FOp(M)$
has a component $\partial_\alpha: \tau(M)\rightarrow\theta(M)$
for each application of a homomorphism~$\alpha_\sigma: M(I_v)\rightarrow\sigma(M)$
to a vertex~$v$ of a tree~$\tau$,
and for each tree $\tau\in\Theta(I)$.
The tree $\theta$ is defined by blowing up the vertex~$v$ into~$\sigma$
within the tree $\tau$.

Note that the input set of~$\sigma$ is by assumption the input set~$I_v$
of the vertex~$v$ in~$\tau$.
For convenience,
we identify the output of~$\sigma$ with the target of the edge $e_v\in E(\tau)$
such that $s(e_v) = v$.

Formally,
the $I$-tree $\theta$ is defined by the vertex set $V(\theta) = V(\tau)\setminus\{v\}\coprod V(\sigma)$
and the edge set $E(\theta) = E(\tau)\setminus(s^{-1}\{v\}\cup t^{-1}\{v\})\coprod E(\sigma)$.
The source (respectively, the target) of an edge~$e$ in~$\theta$
is defined by the source (respectively, the target) of~$e$ in~$\tau$
if~$e\in E(\tau)\setminus(s^{-1}\{v\}\cup t^{-1}\{v\})$,
by the source (respectively, the target) of~$e$ in~$\sigma$
if~$e\in E(\sigma)$.
An example is represented in figure~\ref{Fig:VertexBlowUp}.
The tree $\sigma$ is the framed subtree of the right-hand side.
\begin{figure}
\begin{equation*}
\vcenter{\xymatrix@W=2pt@H=4pt@R=12pt@C=6pt@M=2pt{i\ar[dr] && j\ar[dl] & k\ar[ddr] && l\ar[ddl] & m\ar[dr] && n\ar[dl] \\ 
& *+<3mm>[o][F]{v}\ar@/_4pt/[drrr] &&&&&& *+<3mm>[o][F]{w}\ar@/^4pt/[dlll] & \\ 
&&&& *+<3mm>[o][F]{u}\ar[d] &&&& \\ 
&&&& 0 &&&& 
}}
\mapsto\vcenter{\xymatrix@W=2pt@H=4pt@R=12pt@C=6pt@M=2pt{ i\ar[dr] && j\ar[dl] & k\ar[dd] & m\ar[dr] && n\ar[dl] & \\ 
& *+<3mm>[o][F]{v}\ar@/_4pt/[ddrr] &&&& *+<3mm>[o][F]{w}\ar[dll] && \\ 
&&& *+<3mm>[o][F]{t}\ar[d]\save []!C.[d]!C *+<8pt>[F-:<6pt>]\frm{}\restore &&&& l\ar@/^4pt/[dllll] \\ 
&&& *+<3mm>[o][F]{s}\ar[d] &&&& \\ 
&&& 0 &&&& 
}}
\end{equation*}
\caption{}\label{Fig:VertexBlowUp}\end{figure}

We have identities
\begin{equation*}
\tau(M) = M(I_v)\otimes(\tau\setminus v)(M)
\quad\text{and}\quad\theta(M) = \sigma(M)\otimes(\theta\setminus\sigma)(M),
\end{equation*}
where we set
\begin{equation*}
(\theta\setminus\sigma)(M) = (\tau\setminus v)(M) = \bigotimes_{u\in V(\tau)\setminus\{v\}} M(I_u),
\end{equation*}
and the component $\partial_{\alpha}: \tau(M)\rightarrow\theta(M)$ of~$\partial_\alpha$
is defined by the tensor product of~$\alpha_{\sigma}: M(I_v)\rightarrow\sigma(M)$
with the identity on $(\theta\setminus\sigma)(M) = (\tau\setminus v)(M)$.
Note that the homomorphisms $\partial_{\alpha}: \tau(M)\rightarrow\theta(M)$
sum up to a well defined homomorphism
\begin{equation*}
\underbrace{\bigoplus_{\tau}\tau(M)}_{\FOp(M)}\xrightarrow{\partial_{\alpha}}\underbrace{\bigoplus_{\theta}\theta(M)}_{\FOp(M)}
\end{equation*}
essentially because we deal with finite trees
and the homomorphisms $\alpha_\sigma: M(I_v)\rightarrow\sigma(M)$
are given as components of a well-defined homomorphism $\alpha: M\rightarrow \FOp(M)$.

We check easily that the homomorphism $\partial_{\alpha}$ defined by the construction of this paragraph
forms an operad derivation.
We have clearly $\partial_{\alpha}|_M = \alpha$,
because for an $I$-tree with a single vertex $\psi$ the components of~$\partial_\alpha$
reduce to an application of the homomorphisms~$\alpha_{\sigma}$.

Since we observe at the beginning of this paragraph
that the relation $\partial_{\alpha}|_M = \alpha$
fully determines the derivation $\partial_{\alpha}: \FOp(M)\rightarrow \FOp(M)$,
we obtain finally:

\begin{prop}\label{QuasiFreeOperads:DerivationConstruction}
We have a bijective correspondence between operad derivations $\partial_\alpha: \FOp(M)\rightarrow \FOp(M)$
and homogeneous homomorphisms $\alpha: M\rightarrow \FOp(M)$
such that $\partial_{\alpha}|_M = \alpha$.\qed
\end{prop}

For our needs,
we also record:

\begin{prop}\label{QuasiFreeOperads:TwistingDerivationConstruction}
The operad derivation $\partial_\alpha: \FOp(M)\rightarrow \FOp(M)$
associated to a homomorphism of $\Sigma_*$-objects of degree $-1$
satisfies the equation of a twisting homomorphism $\delta(\partial_\alpha) + \partial_\alpha^2 = 0$
if and only if we have the identity $\delta(\alpha) + \partial_\alpha\cdot\alpha = 0$
in $\Hom_{\M}(M,\FOp(M))$.
\end{prop}

\begin{proof}
Straightforward inspection from the explicit definition of the derivation~$\partial_\alpha$.
\end{proof}

\subsubsection{Homomorphisms on free operads}\label{QuasiFreeOperads:FreeHomomorphisms}
By definition of the free operad,
we have a bijective correspondence between operad morphisms $\phi_f: \FOp(M)\rightarrow\POp$
and morphisms of $\Sigma_*$-objects $f: M\rightarrow\POp$
such that $\phi_f|_M = f$.
For a morphism toward a free operad $f: M\rightarrow\FOp(N)$
and such that $f(M)\subset N$,
this morphism $\phi_f: \FOp(M)\rightarrow\FOp(N)$
is identified with the morphism associated to $f: M\rightarrow N$
by the free operad functor $\FOp: \M\rightarrow\Op$.

The correspondence $f\mapsto\phi_f$
has an obvious extension to homomorphisms
of degree~$0$.
The homomorphism $\phi_f: \FOp(M)\rightarrow\POp$
associated to a homomorphism $f: M\rightarrow \FOp(M)$
is represented by a composite
\begin{equation*}
\FOp(M)\xrightarrow{f} \FOp(\ROp)\xrightarrow{\lambda}\ROp,
\end{equation*}
where $\lambda: \FOp(\ROp)\rightarrow\ROp$ represents the universal composition product of the operad~$\ROp$.
Note simply that, according to the explicit construction of the free operad in~\S\ref{TreeOperadStructures:FreeOperadConstruction},
the free operad functor $\FOp(M)$
has a natural extension to homomorphisms of~$\Sigma_*$-objects of degree~$0$.

For a quasi-free operad,
we obtain:

\begin{prop}\label{QuasiFreeOperads:MorphismConstruction}
Let $\QOp = (\FOp(M),\partial_\alpha)$
be a quasi-free operad.
Let $\ROp$ be any operad.
We have a bijective correspondence between the operad morphisms
\begin{equation*}
(\FOp(M),\partial_\alpha)\xrightarrow{\phi_f}\ROp
\end{equation*}
and the homomorphisms of degree $0$
\begin{equation*}
f\in\Hom_{\M}(M,\ROp)
\end{equation*}
such that $\delta(f) = \phi_f\cdot\alpha$.
\end{prop}

\begin{proof}
The homomorphism $\phi_f: \FOp(M)\rightarrow\ROp$
associated to a homomorphism $f: M\rightarrow\ROp$ of degree $0$
preserves composition products by construction.
The commutation of~$\phi_f$
with differential reads $\delta\cdot\phi_f = \phi_f\cdot\delta + \phi\cdot\partial_{\alpha}$.
Use the explicit construction of~$\phi_f$
and $\partial_\alpha$
to check that this equation is equivalent to the identity $\delta(f) = \phi_f\cdot\alpha$
in $\Hom_{\M}(M,\ROp)$.
\end{proof}

This proposition implies in a particular case:

\begin{prop}\label{QuasiFreeOperads:InducedMorphisms}
A morphism of $\Sigma_*$-objects $f: M\rightarrow N$
defines a morphism between quasi-free operads
\begin{equation*}
\underbrace{(\FOp(M),\partial_\alpha)}_{\POp}\xrightarrow{\FOp(f)}\underbrace{(\FOp(N),\partial_\beta)}_{\QOp}.
\end{equation*}
if and only if we have the identity $\beta\cdot f = \FOp(f)\cdot\alpha$
in $\Hom_{\M}(M,\FOp(N))$.
\end{prop}

\begin{proof}
In the equation of proposition~\ref{QuasiFreeOperads:MorphismConstruction},
the differential of~$f$
has to be replaced by the commutator $(\delta + \partial_\beta)\cdot f + f\cdot\delta = \delta(f) + \beta\cdot f$
since the operad $\QOp = (\FOp(N),\partial_\beta)$
is equipped with the differential $\delta+\partial_\beta: \FOp(N)\rightarrow \FOp(N)$.
Observe that $\delta(f) = 0$
since $f$ is supposed to form a morphism of $\Sigma_*$-objects in dg-modules.
For a morphism $f: M\rightarrow N$,
we also have an identity $\phi_f = \FOp(f)$.
Hence the equation of proposition~\ref{QuasiFreeOperads:MorphismConstruction}
reduces to the identity $\beta\cdot f = \FOp(f)\cdot\alpha$.
\end{proof}

\subsubsection{The canonical filtration of a quasi-free operad}\label{QuasiFreeOperads:FiltrationDefinition}
Recall that the underlying $\Sigma_*$-object of the free operad
has a natural splitting
$\FOp(M) = \bigoplus_{r=0}^{\infty} \FOp_r(M)$
such that
\begin{equation*}
\FOp_r(M)(I) = \bigoplus_{\tau\in\Theta_r(I)} \tau(M)/\equiv,
\end{equation*}
where $\Theta_r(I)$ is the groupoid of $I$-trees with $r$ vertices (see~\S\ref{TreeOperadStructures:AugmentedFreeOperad}).
We have moreover $\FOp_0(M) = \IOp$ and $\FOp_1(M) = M$.
The projection onto $\FOp_0(M)$
provides the free operad with a natural augmentation $\epsilon: \FOp(M)\rightarrow\IOp$
such that $\tilde{\FOp}(M) = \bigoplus_{r=1}^{\infty} \FOp_r(M)$
represents the augmentation ideal of~$\FOp(M)$.

Any homomorphism $\alpha: M\rightarrow \FOp(M)$
admits a splitting $\alpha = \sum_{r=0}^{\infty} \alpha_r$ such that $\alpha_r(M)\subset \FOp_r(M)$,
to which corresponds a splitting $\partial_{\alpha} = \sum_{r=0}^{\infty} \partial_{\alpha_r}$
of the operad derivation associated to~$\alpha$.
We assume usually $\alpha_0 = 0$, or equivalently $\alpha(M)\subset\tilde{\FOp}(M)$.

For a quasi-free operad $\QOp = (\FOp(M),\partial_{\alpha})$,
we also assume $\alpha_1 = 0$,
because $M\xrightarrow{\alpha_1}\FOp_1(M) = M$
determine a twisting homomorphism of~$\Sigma_*$-objects
which can be embodied into the internal differential of~$M$.
These assumptions amount to the requirement $\alpha(M)\subset \FOp_{\geq 2}(M)$,
where we set $\FOp_{\geq 2}(M) = \bigoplus_{r\geq 2} \FOp_r(M)$.

Let $M_{\leq r}$
be the $\Sigma_*$-object
such that
\begin{equation*}
M_{\leq r}(n) = \begin{cases} M(n), & \text{if $n\leq r$}, \\ 0, & \text{otherwise}. \end{cases}
\end{equation*}
We study the suboperads $\FOp(M_{\leq r})\subset\FOp(M)$.
We aim to prove that these suboperads give a filtration of~$\QOp = (\FOp(M),\partial_{\alpha})$.

To begin with, we observe:

\begin{lemm}\label{QuasiFreeOperads:FiltrationDifferential}
Let $\partial_\alpha: \FOp(M)\rightarrow \FOp(M)$ be a derivation of operads
associated to a homomorphism
such that $\alpha(M)\subset \FOp_{\geq 2}(M)$.
If $M$ is reduced,
then we have the relation $\alpha(M_{\leq r})\subset \FOp(M_{\leq r-1})$.
\end{lemm}

\begin{proof}
The assumption $M(0) = M(1) = 0$
implies that the object $\FOp(M)(I) = \bigoplus_{\tau\in\Theta(I)} \tau(M)/\equiv$
reduce to summands $\tau(M)$ such that every vertex of~$\tau$ has at least $2$ inputs.
If $\tau$ has at least $2$ vertices and no more than $r$ inputs,
then this requirement implies that every vertex of~$\tau$
has strictly less than $r$ inputs.
Therefore we have $\alpha(M)\subset \FOp_{\geq 2}(M)\Rightarrow\alpha(M_{\leq r})\subset \FOp(M_{\leq r-1})$.
\end{proof}

\begin{lemm}\label{QuasiFreeOperads:QuasiFreeFiltration}
Let $\QOp = (\FOp(M),\partial)$ be a quasi-free connected operad
such that $\partial(M)\subset \FOp_{\geq 2}(M)$.
The pairs $\QOp_{\leq r} = (\FOp(M_{\leq r}),\partial)$
form a nested collections of suboperads of~$\QOp$
such that $\QOp = \colim_r\QOp_{\leq r}$.
Moreover,
the embeddings $j_r: \QOp_{\leq r-1}\hookrightarrow\QOp_{\leq r}$
fit pushout diagrams
\begin{equation}
\vcenter{\xymatrix{ \FOp(B[0]\otimes M_{(r)})\ar[r]^(0.6){\phi_f}\ar[d]_{\FOp(i[0]\otimes M_{(r)})} & \QOp_{\leq r-1}\ar@{.>}[d]^{j_r} \\
\FOp(E[0]\otimes M_{(r)})\ar@{.>}[r]_(0.6){\phi_g} & \QOp_{\leq r-1} }},
\end{equation}
where $i[0]: B[0]\rightarrow E[0]$ is the degree $0$ generating cofibration of dg-modules
and $M_{(r)}$ is the $\Sigma_*$-object such that
\begin{equation*}
M_{(r)}(n) = \begin{cases} M(n), & \text{if $n = r$}, \\ 0, & \text{otherwise}. \end{cases}
\end{equation*}
\end{lemm}

\begin{proof}
According to lemma~\ref{QuasiFreeOperads:FiltrationDifferential},
we have $\alpha(M_{\leq r})\subset \FOp(M_{\leq r-1})\subset \FOp(M_{\leq r})$,
from which we deduce the relation $\partial_{\alpha}(\FOp(M_{\leq r}))\subset\FOp(M_{\leq r})$
for the derivation associated to $\alpha$.
Hence
we have a well-defined suboperad of~$\QOp = (\FOp(M),\partial_{\alpha})$
defined by the free operad $\FOp(M_{\leq r})$
and the restriction of the derivation $\partial_{\alpha}: \FOp(M)\rightarrow \FOp(M)$
to $\FOp(M_{\leq r})$.

We have by definition $E[0] = \kk e\oplus\kk b$ and $B[0] = \kk b\subset E[0]$,
for generating elements such that $\deg(e) = 0$, $\deg(b) = -1$ and $\delta(e) = b$.

The homomorphism $\alpha: M(r)\rightarrow \FOp(M_{\leq r-1})$, which has degree $-1$,
yields a homomorphism of degree $0$
\begin{equation*}
f: B[0]\otimes M_{(r)}\rightarrow \FOp(M_{\leq r-1}).
\end{equation*}
The relation $\delta(\alpha) + \partial_{\alpha}\cdot\alpha = 0$
implies $(\delta + \partial_{\alpha})\cdot f = f\cdot\delta$,
from which we deduce that $f$ gives a well-defined morphism of $\Sigma_*$-objects $f: B[0]\otimes M_{(r)}\rightarrow\QOp_{\leq r-1}$.
Form the operad morphism $\phi_f: \FOp(B[0]\otimes M_{(r)})\rightarrow\QOp_{\leq r-1}$
associated to~$f$.

Let
\begin{equation*}
g: E[0]\otimes M_{(r)}\rightarrow \FOp(M_{\leq r}).
\end{equation*}
be the extension of the homomorphism
\begin{equation*}
B[0]\otimes M_{(r)}\xrightarrow{f} \FOp(M_{\leq r-1})\rightarrow \FOp(M_{\leq r})
\end{equation*}
defined by the canonical embedding of~$M(r)$ into $\FOp(M_{\leq r})$
on the summand $\kk e\otimes M_{(r)}$.
One checks readily that $g$, like $f$,
gives a well-defined morphism of $\Sigma_*$-objects $g: E[0]\otimes M_{(r)}\rightarrow\QOp_{\leq r}$.
We form the operad morphism $\phi_g: \FOp(E[0]\otimes M_{(r)})\rightarrow\QOp_{\leq r}$
associated to~$g$.

We have by construction $g\cdot \FOp(i[0]\otimes M_{(r)}) = j_r\cdot f$.
Use the characterization of morphisms on quasi-free operads
to check that the commutative square (*)
formed by these morphisms is a pushout of operads.
\end{proof}

\begin{thm}\label{QuasiFreeOperads:CofibrantOperads}
Let $\QOp = (\FOp(M),\partial_{\alpha})$ be a quasi-free connected operad
such that $\partial_{\alpha}(M)\subset\FOp_{\geq 2}(M)$.
If $M$ is $\Sigma_*$-cofibrant,
then $\QOp$ is cofibrant as an operad.
\end{thm}

\begin{proof}
Each object $M_{(r)}$ is $\Sigma_*$-cofibrant if $M$ is so (for instance, observe that $M_{(r)}$ forms a retract of~$M$).
From this assertion,
we deduce that each morphism $i[0]\otimes M_{(r)}$
forms a cofibration of $\Sigma_*$-objects,
because the tensor product of $\Sigma_*$-objects with dg-modules
satisfies the pushout-product axiom of symmetric monoidal model categories (see~\cite[Proposition 11.4.B]{FresseModuleBook}).
Hence,
the observations of~\S\ref{OperadHomotopy:CofibrantCellComplexes}
imply that the morphism $\eta: \IOp\rightarrow\QOp$
forms a cofibration
and the proposition follows.
\end{proof}

\begin{prop}\label{QuasiFreeOperads:Cofibrations}
Suppose we have a morphism of reduced $\Sigma_*$-objects $f: M\rightarrow N$
which determines a morphism of quasi-free operads
\begin{equation*}
\underbrace{(\FOp(M),\partial_{\alpha})}_{\POp}\xrightarrow{\FOp(f)}\underbrace{(\FOp(N),\partial_{\beta})}_{\QOp},
\end{equation*}
for which we have $\alpha(M)\subset\FOp_{\geq 2}(M)$ and $\beta(N)\subset\FOp_{\geq 2}(N)$.
This morphism $\FOp(f)$ is a cofibration (respectively, an acyclic cofibration of operads
whenever $f$ is a cofibration (respectively, an acyclic cofibration)
of $\Sigma_*$-objects.
\end{prop}

\begin{proof}
This morphism $\FOp(f)$ preserves clearly the natural filtration of quasi-free operads.
Thus we have a diagram of morphisms
\begin{equation*}
\xymatrix{ \POp\ar[r]^{=} & \POp\ar[r]^{=} & \ar@{.}[r] &
\ar[r]^{=} & \POp\ar[r]^{=} & \ar@{.}[r] &
\ar[r]^{=} & \POp \\
\IOp\ar[u]\ar[d]_{=}\ar[r]^{=} & \POp_{\leq 0}\ar[u]^{\FOp(f)}\ar@{^{(}->}[]!R+<4pt,0pt>;[r]\ar@{.>}[d] & \ar@{.}[r] &
\ar@{^{(}->}[]!R+<4pt,0pt>;[r] & \POp_{\leq r}\ar[u]^{\FOp(f)}\ar@{^{(}->}[]!R+<4pt,0pt>;[r]\ar@{.>}[d] & \ar@{.}[r] &
\ar@{^{(}->}[]!R+<4pt,0pt>;[r] & \POp\ar[u]_{=}\ar[d]^{\FOp(f)} \\
\IOp\ar[r]^{=} & \QOp_{\leq 0}\ar@{^{(}->}[]!R+<4pt,0pt>;[r] & \ar@{.}[r] &
\ar@{^{(}->}[]!R+<4pt,0pt>;[r] & \QOp_{\leq r}\ar@{^{(}->}[]!R+<4pt,0pt>;[r] & \ar@{.}[r] &
\ar@{^{(}->}[]!R+<4pt,0pt>;[r] & \QOp }.
\end{equation*}
that can be patched together to give a decomposition of~$\FOp(f)$
of the form:
\begin{multline*}
\POp = \POp\bigvee_{\POp_{\leq 0}}\QOp_{\leq 0}\rightarrow\dots
\rightarrow\POp\bigvee_{\POp_{\leq r-1}}\QOp_{\leq r-1}\xrightarrow{j_r}\POp\bigvee_{\POp_{\leq r}}\QOp_{\leq r}\rightarrow\dots\\
\dots\rightarrow\colim_r\POp\bigvee_{\POp_{\leq r}}\QOp_{\leq r} = \QOp.
\end{multline*}
Observe that the morphisms $j_r$ which occur in this decomposition
fit pushouts
\begin{equation*}
\xymatrix{ \FOp(E[0]\otimes M_{(r)}\bigoplus_{B[0]\otimes M_{(r)}} B[0]\otimes N_{(r)})\ar[r]\ar[d]_{\FOp(i[0]\square f)} &
\POp\bigvee_{\POp_{\leq r-1}}\QOp_{\leq r-1}\ar@{.>}[d]^{j_r} \\
\FOp(E[0]\otimes N_{(r)})\ar@{.>}[r] & \POp\bigvee_{\POp_{\leq r}}\QOp_{\leq r} },
\end{equation*}
where
\begin{equation*}
i[0]\square f: E[0]\otimes M_{(r)}\bigoplus_{B[0]\otimes M_{(r)}} B[0]\otimes N_{(r)}\rightarrow E[0]\otimes N_{(r)}
\end{equation*}
is the pushout-product of~$i[0]: B[0]\rightarrow E[0]$
and $f: M_{(r)}\rightarrow N_{(r)}$.

Observe that each morphism $f: M_{(r)}\rightarrow N_{(r)}$
forms a cofibration (respectively, an acyclic cofibration) of~$\Sigma_*$-objects
if $f$ is so
(for instance, because $f: M_{(r)}\rightarrow N_{(r)}$ forms a retract of $f: M\rightarrow N$ in the category of $\Sigma_*$-objects).
Recall that the tensor product of $\Sigma_*$-objects with dg-modules
satisfies the pushout-product axiom of monoidal model categories (see~\cite[Proposition 11.4.B]{FresseModuleBook}).
Thus the morphism $i[0]_*\square f_*$
forms a cofibration (respectively, an acyclic cofibration) of~$\Sigma_*$-objects
if $f$ is so.

The conclusion of the proposition follows since (see~\S\ref{OperadHomotopy:CofibrantCellComplexes}):
any morphism of free operads $\FOp(i): \FOp(M)\rightarrow \FOp(N)$
induced by a cofibration (respectively, an acyclic cofibration) of $\Sigma_*$-objects $i: M\rightarrow N$
forms a cofibration (respectively, an acyclic cofibration) in the category of operads (immediate by adjunction);
the class of (acyclic) cofibrations in a semi-model category is closed under pushouts and composites.
\end{proof}

\section{The category of algebras associated to an operad}\label{OperadAlgebras}
\renewcommand{\thesubsubsection}{\thesection.\arabic{subsubsection}}

The first purpose of this section is to review the definition of the category of algebras associated to an operad.

The categories of algebras associated to $\Sigma_*$-cofibrant operads
inherit a semi-model structure from dg-modules.
This model structure is defined in~\cite[\S 4]{GetzlerJones}
for algebras over operads in non-negatively graded dg-modules
over a field of characteristic zero.
The generalizations of the definition of~\cite{GetzlerJones}
in the context of unbounded dg-modules over a ring
are addressed in~\cite{HinichHomotopy} and \cite{BergerMoerdijkModel,Harper,Spitzweck}.
We review briefly this definition.

In the setting of~\cite{GetzlerJones},
the cofibrant objects are retracts of quasi-free objects of the category of $\POp$-algebras,
for any operad $\POp$.
In the context of unbounded dg-modules over a ring,
we prove that the quasi-free $\POp$-algebras equipped with a good filtration
are cofibrant $\POp$-algebras.

\subsubsection{The functors associated to $\Sigma_*$-objects}\label{OperadAlgebras:SigmaObjectFunctor}
To any $\Sigma_*$-object $M$,
we associate a functor $S(M): \C\rightarrow\C$
which maps a dg-module $C\in\C$
to the dg-module of symmetric tensors with coefficients in $M$:
\begin{equation*}
S(M,C) = \bigoplus_{n=0}^{\infty} (M(n)\otimes C^{\otimes n})_{\Sigma_n}.
\end{equation*}
In this expansion,
the quotient by the action of symmetric groups $\Sigma_n$
identifies the action of permutations on tensor powers~$C^{\otimes n}$
with the natural $\Sigma_n$-action on~$M(n)$.

The element of~$S(M,C)$
represented by the tensor~$x\otimes c_1\otimes\dots\otimes c_n\in M(n)\otimes C^{\otimes n}$
is denoted by $x(c_1,\dots,c_n)$.
The $\Sigma_n$-coinvariance
relation reads
\begin{equation*}
(w x)(c_1,\dots,c_n) = x(c_{w(1)},\dots,c_{w(n)}),
\end{equation*}
for $w\in\Sigma_n$.
The tensor $x(c_1,\dots,c_n)$
can also be represented by a tree with one vertex labeled by $x$
whose inputs are labeled by the elements $c_1,\dots,c_n\in C$:
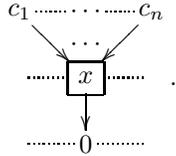
\begin{equation*}
\vcenter{\xymatrix@H=6pt@W=3pt@M=2pt@!R=1pt@!C=1pt{ c_1\ar@{.}[r]\ar[dr] & \cdots\ar@{.}[r]\ar@{}[d]|{\displaystyle{\cdots}} & c_n\ar[dl] \\
\ar@{.}[r] & *+<8pt>[F]{x}\ar@{.}[r]\ar[d] & \\ \ar@{.}[r] & 0\ar@{.}[r] & }}.
\end{equation*}
The $\Sigma_n$-coinvariance relation
is encoded in the structure of the tree.

Let $\F$ denote the category of functors $F: \C\rightarrow\C$.
The map $S: M\mapsto S(M)$ defines clearly a functor $S: \M\rightarrow\F$.
Moreover:
\begin{itemize}
\item
the functor associated to the unit $\Sigma_*$-object $\IOp$
is isomorphic to the identity functor $\Id$,
\item
the functor associated to a composite $\Sigma_*$-object $M\circ N$
is isomorphic to the composite functor $S(M)\circ S(N)$,
\item
and the isomorphisms that give these relations $S(\IOp)\simeq\Id$ and $S(M\circ N)\simeq S(M)\circ S(N)$
satisfy natural coherence identities
with respect to the associativity and unit isomorphisms
of composition structures
\end{itemize}
(see for instance the survey of~\cite[\S 1]{FresseKoszulDuality}).
To summarize these observations,
one says that the map $S: M\mapsto S(M)$ defines a monoidal functor $S: (\M,\circ,\IOp)\rightarrow(\F,\circ,\Id)$.

\subsubsection{The category of algebras associated to an operad}\label{OperadAlgebras:CategoryDefinition}
An algebra over an operad $\POp$
is a dg-module $A$
together with $\Sigma_*$-invariant evaluation products
\begin{equation}
\lambda: \POp(n)\otimes A^{\otimes n}\rightarrow A,\quad\text{$n\in\NN$},
\end{equation}
that satisfy natural unit and associativity relations
with respect to the unit and composition product of~$\POp$.
The $\Sigma_*$-invariance requirement implies that these evaluation products
assemble to a morphism $\lambda: S(\POp,A)\rightarrow A$.
The image of~$p\otimes a_1\otimes\dots\otimes a_n\in\POp(n)\otimes A^{\otimes n}$
under the evaluation product of~$A$
is denoted by~$p(a_1,\dots,a_n)$.
Naturally,
a morphism of $\POp$-algebras is a morphism of dg-modules $f: A\rightarrow B$
which preserves the evaluation products (*).

Usual categories of algebras, like the category of associative and commutative algebras,
the category of associative algebras,
and the category of Lie algebras,
are associated to operads (called the commutative operad $\COp$,
the associative operad $\AOp$,
and the Lie operad $\LOp$ in the above-mentioned examples).

\subsubsection{Operads and monads}\label{OperadAlgebras:OperadMonad}
In fact,
the functor $S(\POp): \C\rightarrow\C$
associated to an operad $\POp$ inherits a monad structure
since the operad unit $\eta: \IOp\rightarrow\POp$ and the operad composition product $\mu: \POp\circ\POp\rightarrow\POp$
induce functor morphisms
\begin{equation*}
\Id\simeq S(\IOp)\xrightarrow{S(\eta)} S(\POp)
\quad\text{and}
\quad S(\POp)\circ S(\POp)\simeq S(\POp\circ\POp)\xrightarrow{S(\mu)} S(\POp)
\end{equation*}
that still satisfy unit and associativity relations.
For an algebra over $\POp$,
the unit and associativity relation of the evaluation products~(*)
amount to unit and associativity relations
with respect to the unit and composition product of the monad $S(\POp)$.
Hence
we obtain that the structure of an algebra over an operad $\POp$
amounts to the structure of an algebra over the monad $S(\POp)$
associated to $\POp$.
Intuitively,
the elements of~$S(\POp,A)$ represents formal composites $p(a_1,\dots,a_n)$
and the morphism $\lambda: S(\POp,A)\rightarrow A$
performs the evaluation of these composite elements.

The relationship between operads and monads was a motivation for the original definition of an operad in~\cite{May}.
The abstract interpretation of this relationship in terms of a composition structure on the category of $\Sigma_*$-objects
goes back to~\cite{Smirnov}.

\subsubsection{Algebras over augmented operads}
For an augmented operad $\POp$, for which we have a splitting $\POp = \IOp\oplus\tilde{\POp}$,
the evaluation structure of a $\POp$-algebra $A$
is fully determined by the action of operations of the augmentation ideal on $A$,
because the unit operation $1\in\POp(1)$
is supposed to verify the identity $1(a) = a$, for every $a\in A$.

The associativity relation of a $\POp$-algebra structure
is also implied by the identities
\begin{equation*}
(p\circ_e q)(a_1,\dots,a_{r+s-1}) = p(a_1,\dots,q(a_e,\dots,a_{e+s-1}),\dots,a_{r+s-1}),
\end{equation*}
which only involve the action and the structure of the augmentation ideal of~$\POp$.
This representation of the associativity relation
will be sufficient for our needs.

\subsubsection{Free algebras}\label{OperadAlgebras:FreeAlgebras}
The monadic definition of the structure of a $\POp$-algebra implies that the object $\POp(C) = S(\POp,C)$
associated to a dg-module $C\in\C$
represents the free $\POp$-algebra generated by $C$ in the usual sense:
we have a morphism of dg-modules $\eta: C\rightarrow\POp(C)$
which is universal in the sense that any morphism $f: \POp(C)\rightarrow A$ toward a $\POp$-algebra $A$
has a unique factorization
\begin{equation*}
\xymatrix{ C\ar[rr]^{f}\ar[dr]_{\eta} && A \\ & \POp(C)\ar@{.>}[ur]_{\exists!\phi_f} & }
\end{equation*}
such that $\phi_f$ is a morphism of $\POp$-algebras.
The evaluation product of~$\POp(C) = S(\POp,C)$
is defined by the morphism
\begin{equation*}
S(\POp,S(\POp,C))\simeq S(\POp\circ\POp,C)\xrightarrow{S(\mu,C)} S(\POp,C)
\end{equation*}
induced by the composition product of~$\POp$.
The universal morphism $\eta: C\rightarrow\POp(C)$
is defined by the morphism
\begin{equation*}
C\simeq S(\IOp,C)\xrightarrow{S(\eta,C)} S(\POp,C)
\end{equation*}
induced by the unit of~$\POp$.

In principle,
we adopt the notation $\POp(C)$ to refer to the object $S(\POp,C)$ equipped with this free $\POp$-algebra structure,
but we keep the notation $S(\POp,C)$ for the underlying dg-module of~$\POp(C)$.

In the sequel,
we use the canonical morphism $\eta: C\rightarrow\POp(C)$
to identify the dg-module $C$
with a submodule of~$\POp(C)$.
The formal composites $p(c_1,\dots,c_n)\in\POp(C)$ of~\S\ref{OperadAlgebras:OperadMonad}
represent the evaluation of operations $p\in\POp(n)$
on elements $c_1,\dots,c_n\in C$.
The morphism of $\POp$-algebras $\phi_f: \POp(C)\rightarrow A$
associated to a morphism of dg-modules $f: C\rightarrow A$
is given by the explicit formula
\begin{equation*}
\phi_f(p(c_1,\dots,c_n)) = p(f(c_1),\dots,f(c_n)),
\end{equation*}
for every $p(c_1,\dots,c_n)\in\POp(C)$.

\subsubsection{The semi-model category of algebras over an operad}\label{OperadAlgebras:SemiModelStructure}
The category of algebras over an operad $\POp$
is denoted by ${}_{\POp}\C$.
The free $\POp$-algebra functor $\POp(-): \C\rightarrow{}_{\POp}\C$
defines a left adjoint of the obvious forgetful functor $U: {}_{\POp}\C\rightarrow\C$.
The construction of~\S\ref{OperadHomotopy:AdjointModelStructures}
is applied to the adjunction $\POp(-): \C\rightleftarrows{}_{\POp}\C :U$
in order to provide the category of $\POp$-algebras with a model structure
such that the forgetful functor $U: {}_{\POp}\C\rightarrow\C$
creates weak-equivalences and fibrations.
The construction returns the following result:

\begin{fact}
The category of $\POp$-algebras inherits a semi-model structure if the operad $\POp$ is $\Sigma_*$-cofibrant
(see~\cite{FresseModuleBook,Mandell,Spitzweck}),
a full model structure if $\POp$ is cofibrant as an operad (and in many other good cases, see~\cite{BergerMoerdijkModel,Harper,HinichHomotopy}).
\end{fact}

By construction,
the semi-model category of $\POp$-algebras is cofibrantly generated by the morphisms
of free $\POp$-algebras $\POp(i): \POp(C)\rightarrow\POp(D)$
such that $i: C\rightarrow D$ is a generating (acyclic) cofibration
of the category of dg-modules.
The goal of the next paragraphs is to identify the cofibrant cell complexes
of the category of $\POp$-algebras.
In summary,
we prove that these cofibrant cell complexes
are quasi-free operads equipped with a good filtration.
First of all,
we review the definition of a quasi-free object in the category of algebras over an operad.

\subsubsection{Derivations and twisted dg-algebras over operads}\label{OperadAlgebras:Derivations}
By definition,
the evaluation products $\lambda: \POp(r)\otimes A^{\otimes r}\rightarrow A$ of a $\POp$-algebra in dg-modules $A$
are morphisms of dg-modules.
This assumption amounts to the requirement that the differential of~$A$
satisfies the derivation relation
\begin{equation}
\delta(p(a_1,\dots,a_r)) = \delta(p)(a_1,\dots,a_r) + \sum_{i=1}^{r} \pm p(a_1,\dots,\delta(a_i),\dots,a_r)
\end{equation}

Say that a homogeneous homomorphisms $\theta\in\Hom_{\C}(A,A)$
defines a $\POp$-algebra derivation
if we have the relation
\begin{equation}\renewcommand{\theequation}{**}
\theta(p(a_1,\dots,a_r)) = \sum_{i=1}^{r} \pm p(a_1,\dots,\theta(a_i),\dots,a_r)
\end{equation}
for every composite $p(a_1,\dots,a_r))\in A$.

Clearly,
the differential $\delta+\partial: A\rightarrow A$
obtained by the addition of a twisting homomorphism $\partial\in\Hom_{\C}(A,A)$
to the internal differential of a $\POp$-algebra~$A$
satisfies the derivation relation (*)
if and only if $\partial$ is a $\POp$-algebra derivation in the sense (**).
Thus,
a twisted dg-module $(A,\partial)$ associated to a $\POp$-algebra $A$
inherits a $\POp$-algebra structure
if and only if the twisting homomorphism $\partial$ forms a $\POp$-algebra derivation.

\subsubsection{Quasi-free algebras over operads}\label{OperadAlgebras:QuasiFreeAlgebras}
A quasi-free $\POp$-algebra is a twisted $\POp$-algebra $(\POp(C),\partial)$
formed from a free $\POp$-algebra $\POp(C)$.

The derivation relation of~\S\ref{OperadAlgebras:Derivations}
implies that any derivation on a free $\POp$-algebra $\partial: \POp(C)\rightarrow\POp(C)$
is determined by its restriction to the generating dg-module $C$.
In the converse direction,
we have a well-defined derivation $\partial_{\alpha}: \POp(C)\rightarrow\POp(C)$
associated to any homomorphism $\alpha: C\rightarrow\POp(C)$
and such that $\partial_{\alpha}|_{C} = \alpha$.
This derivation is defined by the formula
\begin{equation*}
\partial_{\alpha}(p(c_1,\dots,c_r)) = \sum_{i=1}^{r} \pm p(c_1,\dots,\alpha(c_i),\dots,c_r),
\end{equation*}
for any element $p(c_1,\dots,c_r)\in\POp(C)$.
Hence:

\begin{prop}\label{OperadAlgebras:DerivationConstruction}
For a free $\POp$-algebra $\POp(C)$,
we have a bijective correspondence between derivations $\partial_{\alpha}: \POp(C)\rightarrow\POp(C)$
and homomorphisms $\alpha: C\rightarrow\POp(C)$.\qed
\end{prop}

By an easy verification, we obtain moreover:

\begin{prop}\label{OperadAlgebras:TwistingDerivationConstruction}
The $\POp$-algebra derivation $\partial_\alpha: \POp(C)\rightarrow \POp(C)$
associated to a homomorphism $\alpha: C\rightarrow\POp(C)$ of degree $-1$
satisfies the equation of a twisting homomorphism $\delta(\partial_\alpha) + \partial_\alpha^2 = 0$
if and only if we have the identity
\begin{equation}
\delta(\alpha) + \partial_\alpha\cdot\alpha = 0
\end{equation}
in $\Hom_{\M}(C,\POp(C))$.\qed
\end{prop}

In~\cite[\S 2.2]{GetzlerJones},
a similar assertion is established for operads over a ring of characteristic $0$.
In this reference,
equation (*) is replaced by a condition involving the Lie bracket
$[\partial_{\alpha},\partial_{\alpha}] = 2\partial_{\alpha}\partial_{\alpha}$,
for a derivation $\partial_{\alpha}$ of degree $1$.
Naturally,
the condition of~\cite{GetzlerJones} is no more equivalent to (*)
when $2$-torsion occurs
and we can not use the Lie bracket of derivations
in our setting.

\medskip
The formula (*) of~\S\ref{OperadAlgebras:FreeAlgebras}
has an obvious extension
to homomorphisms of degree $0$
and yields a bijection between the homomorphisms of degree $0$
\begin{equation*}
f\in\Hom_{\M}(C,A)
\end{equation*}
and the homomorphisms
\begin{equation*}
\phi_f: \POp(C)\rightarrow A
\end{equation*}
which commutes with the action of operations.
For a quasi-free $\POp$-algebra,
we have further:

\begin{prop}\label{OperadAlgebras:QuasiFreeMorphismConstruction}
Let $(\POp(C),\partial_\alpha)$ be a quasi-free $\POp$-algebra.
Let $A$ be any $\POp$-algebra.
We have a bijective correspondence between the morphisms of $\POp$-algebras
\begin{equation*}
(\POp(C),\partial_{\alpha})\xrightarrow{\phi_f} A
\end{equation*}
and the homomorphisms of degree $0$
\begin{equation*}
f\in\Hom_{\M}(C,A)
\end{equation*}
such that $\delta(f) = \phi_f\cdot\alpha$.
\end{prop}

\begin{proof}
Easy from the explicit definition of~$\partial_\alpha$ and $\phi_f$.
\end{proof}

The next proposition gives a motivation
for the definition of quasi-free $\POp$-algebras:

\begin{prop}\label{OperadAlgebras:CofibrantAlgebras}
Let $A = (\POp(C),\partial_\alpha)$ be a quasi-free $\POp$-algebra.
Suppose we have a filtration
\begin{equation*}
0 = C_0\hookrightarrow\cdots\hookrightarrow C_{\lambda}\hookrightarrow\cdots\hookrightarrow\colim_{\lambda} C_{\lambda} = C
\end{equation*}
such that $\partial_\alpha(C_\lambda)\subset\POp(C_{\lambda-1})$
and each embedding $i_{\lambda}: C_{\lambda-1}\hookrightarrow C_{\lambda}$
forms a cofibration of dg-modules.
Then $A$ is cofibrant as a $\POp$-algebra.
\end{prop}

\begin{proof}
The assumption
\begin{equation*}
\partial_\alpha(C_\lambda)\subset\POp(C_{\lambda-1})\subset\POp(C_{\lambda})
\end{equation*}
implies the existence of a nested sequence of quasi-free $\POp$-algebras $A_{\lambda} = (\POp(C_{\lambda}),\partial_{\alpha})$
such that $A = \bigcup_{\lambda} A_{\lambda}$.

Recall that $E[0]$
is the dg-module such that $E[0] = \kk e\oplus\kk b$,
where $\deg(e) = 0$, $\deg(b) = -1$, $\delta(e) = b$,
and $B[0]$ is the submodule of~$E[0]$ spanned by $b$.
The homomorphism $\alpha: C_{\lambda}\rightarrow\POp(C_{\lambda-1})$,
defined by the restriction of~$\alpha: C\rightarrow\POp(C)$,
gives rise to a homomorphism of degree $0$
\begin{equation*}
f_{\lambda}: B[0]\otimes C_{\lambda}\rightarrow\POp(C_{\lambda-1}).
\end{equation*}
The relation $\delta(\alpha) + \partial_\alpha\cdot\alpha = 0$
implies that $f_{\lambda}$
forms a morphism of dg-modules $f_{\lambda}: B[0]\otimes C_{\lambda}\rightarrow A_{\lambda-1}$
and yields a morphism of $\POp$-algebras $\phi_{f_{\lambda}}: \POp(B[0]\otimes C_{\lambda})\rightarrow A_{\lambda-1}$.
We also have a morphism of dg-modules $g_{\lambda}: E[0]\otimes C_{\lambda}\rightarrow A_{\lambda}$
such that $g_{\lambda}(e\otimes c) = c\in\POp(C_{\lambda})$,
for each $c\in C_{\lambda}$
and an associated morphism of $\POp$-algebras $\phi_{g_{\lambda}}: \POp(E[0]\otimes C_{\lambda})\rightarrow A_{\lambda}$.

Check that the square
\begin{equation*}
\xymatrix{ \POp(B[0]\otimes C_{\lambda-1})\ar[r]\ar[d] & \POp(E[0]\otimes C_{\lambda-1})\ar@{.>}[d]^{\phi_{g_{\lambda-1}}} \\
\POp(B[0]\otimes C_{\lambda})\ar@{.>}[r]_{\phi_{f_{\lambda}}} & A_{\lambda-1} }
\end{equation*}
commutes
and observe that
\begin{equation*}
\POp(B[0]\otimes C_{\lambda}\bigoplus_{B[0]\otimes C_{\lambda-1}} E[0]\otimes C_{\lambda-1})
= \POp(B[0]\otimes C_{\lambda})\bigvee_{\POp(B[0]\otimes C_{\lambda-1})}\POp(E[0]\otimes C_{\lambda-1})
\end{equation*}
to deduce that $(\phi_{f_{\lambda}},\phi_{g_{\lambda-1}})$
defines a morphism of $\POp$-algebras
\begin{equation*}
\POp(B[0]\otimes C_{\lambda}\bigoplus_{B[0]\otimes C_{\lambda-1}} E[0]\otimes C_{\lambda-1})
\xrightarrow{(\phi_{f_{\lambda}},\phi_{g_{\lambda-1}})} A_{\lambda-1}.
\end{equation*}

Each embedding $j_{\lambda}: A_{\lambda-1}\rightarrow A_{\lambda}$
fits a commutative square
\begin{equation*}
\xymatrix{ \POp(B[0]\otimes C_{\lambda}\bigoplus_{B[0]\otimes C_{\lambda-1}} E[0]\otimes C_{\lambda-1})
\ar[rr]^(0.7){(\phi_{f_{\lambda}},\phi_{g_{\lambda-1}})}\ar[d]_{\POp(i[0]\square i_{\lambda})} &&
A_{\lambda-1}\ar@{.>}[d]^{j_{\lambda}} \\
\POp(E[0]\otimes C_{\lambda})\ar@{.>}[rr]_{\phi_{g_{\lambda}}} &&
A_{\lambda} }
\end{equation*}
where
\begin{equation*}
B[0]\otimes C_{\lambda}\bigoplus_{B[0]\otimes C_{\lambda-1}} E[0]\otimes C_{\lambda-1}
\xrightarrow{i[0]\square i_{\lambda}} C_{\lambda}
\end{equation*}
represents the pushout-product of~$i[0]: B[0]\rightarrow E[0]$ and $i_{\lambda}: C_{\lambda-1}\rightarrow C_{\lambda}$.
By an easy verification,
we see that these commutative squares form pushouts
in the category of $\POp$-algebras.

The pushout-product axiom
implies that $i[0]\square i_{\lambda}$ forms a cofibration of dg-modules.
By adjunction,
we obtain that $\POp(i[0]\square i_{\lambda})$
forms a cofibration in the category of $\POp$-algebras.
Since cofibrations (with a cofibrant domain) are stable under pushouts in semi-model categories,
we obtain that each $j_{\lambda}$
forms a cofibration of $\POp$-algebras,
from which we conclude that $A$ is cofibrant as a $\POp$-algebra.
\end{proof}

\subsubsection{Endomorphism operads}\label{OperadAlgebras:EndomorphismOperads}
Let $C\in\C$ be any dg-module.
The collection of dg-modules $\End_C(n) = \Hom_{\C}(C^{\otimes n},C)$
inherits a natural operad structure
and defines the endomorphism operad of~$C$.

The structure of a $\POp$-algebra $A$ is equivalent to a morphism of operads $\phi: \POp\rightarrow\End_A$,
because the evaluation morphisms of a $\POp$-algebra $\lambda: S(\POp,A)\rightarrow A$
are adjoint to morphisms of dg-modules $\phi: \POp(n)\rightarrow\Hom_{\C}(A^{\otimes n},A)$
preserving operad structures.

Usually,
we use the notation of the underlying dg-module $A$ to refer to a $\POp$-algebra,
because we assume that $A$ has a natural internal $\POp$-algebra structure.
This is no more the case in~\S\S\ref{CoalgebraModels}-\ref{CylinderHomotopyMorphisms}.
In this situation,
we represent a $\POp$-algebra by a pair $(A,\phi)$, where $A$ is the underlying dg-module
and $\phi: \POp\rightarrow\End_A$
is the operad morphism
which determines the $\POp$-algebra structure.

\part*{Applications of the operadic cobar construction}

\section{The operadic cobar construction}\label{CobarConstruction}
\renewcommand{\thesubsubsection}{\thesection.\arabic{subsubsection}}

The main purpose of this section is to review the definition of the cobar construction associated to a cooperad
(we adapt ideas of~\cite{GetzlerJones} to our setting).
To summarize the construction,
a cooperad is a $\Sigma_*$-object $\DOp$ equipped with a structure essentially dual to the composition product of an operad
and the operadic cobar construction of~$\DOp$ is a quasi-free operad, denoted by $B^c(\DOp)$,
whose twisting derivation
is determined by the cooperad structure of~$\DOp$.

In the second part of the section,
we study the morphisms $\phi: B^c(\DOp)\rightarrow\POp$ toward an operad $\POp$.
In the next section
we prove that an operad equivalence $\phi: B^c(\DOp)\xrightarrow{\sim}\POp$
determines the twisting homomorphism of a quasi-free replacements of $\POp$-algebras,
for any operad $\POp$ in unbounded dg-modules.
For the moment
we only check that certain twisted $\Sigma_*$-objects $(\POp\circ\DOp,\partial_{\theta})$
which represent these quasi-free replacements
are acyclic.

The bar duality of operads associates to any operad $\POp$
a cooperad $\DOp$
such that we have a weak-equivalence $\phi: B^c(\DOp)\xrightarrow{\sim}\POp$.
To complete the results of this section,
we explain briefly how to arrange this construction in order to form a $\Sigma_*$-cofibrant cooperad from~$\DOp$.
We also study applications of the Koszul duality of operads
and the usual examples of the associative, commutative and Lie operads.

\subsubsection{Cooperads}\label{CobarConstruction:Cooperads}
Basically,
a cooperad is a $\Sigma_*$-object $\DOp$
equipped with a counit $\epsilon: \DOp\rightarrow \IOp$
and a composition coproduct $\nu: \DOp\rightarrow\DOp\circ\DOp$
such that the diagrams
\begin{equation*}
\vcenter{\xymatrix{ & \ar[dl]_{\simeq}\DOp\ar[dr]^{\simeq}\ar[d]^{\nu} & \\
\IOp\circ\DOp & \DOp\circ\DOp\ar[l]^{\epsilon\circ\DOp}\ar[r]_{\DOp\circ\epsilon} & \DOp\circ \IOp }}
\qquad\text{and}
\qquad\vcenter{\xymatrix{ \DOp\ar[r]^{\nu}\ar[d]_{\nu} & \DOp\circ\DOp\ar[d]^{\DOp\circ\nu} \\
\DOp\circ\DOp\ar[r]_{\nu\circ\DOp} & \DOp\circ\DOp\circ\DOp }},
\end{equation*}
dual to the diagrams of~\S\ref{OperadDefinition},
commute.

In the sequel,
we represent the expansion of a cooperad coproduct
by an expression of the form:
\begin{equation*}
\nu\left\{\vcenter{\xymatrix@H=6pt@W=3pt@M=2pt@!R=1pt@!C=1pt{
i_1\ar[dr] & \cdots\ar@{}[d]|{\displaystyle{\cdots}} & i_n\ar[dl] \\
& *+<8pt>[F]{\gamma}\ar[d] & \\ & 0 & }}\right\}
= \sum_{\nu(\gamma)}
\left\{\vcenter{\xymatrix@H=6pt@W=3pt@M=2pt@!R=1pt@!C=1pt{ \ar@{.}[r] &
i_*\ar[dr]\ar@{.}[r] & \cdots\ar@{.}[r]\ar@{}[d]|{\displaystyle{\cdots}} & i_*\ar[dl]\ar@{.}[r] &
\cdots\ar@{.}[r] & i_*\ar[dr]\ar@{.}[r] & \cdots\ar@{.}[r]\ar@{}[d]|{\displaystyle{\cdots}} & i_*\ar[dl]\ar@{.}[r] & \\
*+<2pt>{1}\ar@{.}[rr] && *+<8pt>[F]{\gamma_*}\ar[drr]\ar@{.}[rr] && \cdots\ar@{.}[rr] && *+<8pt>[F]{\gamma_*}\ar[dll]\ar@{.}[rr] && \\
*+<2pt>{0}\ar@{.}[rrrr] &&&& *+<8pt>[F]{\gamma_*}\ar[d]\ar@{.}[rrrr] &&&& \\
\ar@{.}[rrrr] &&&& 0\ar@{.}[rrrr] &&&& }}\right\}.
\end{equation*}
In the summation set,
the notation $\gamma$ refers to the element of which we take the coproduct.
In the expansion,
we simply use the notation $\gamma_*$
to refer to the factors of the coproduct of~$\gamma$.

In this paper,
we only consider cooperads which are connected (in the sense that $\DOp(0) = 0$ and $\DOp(1) = \kk$),
because technical difficulties occur when this assumption is not satisfied.

The unit $\Sigma_*$-object $\IOp$
forms a cooperad,
and any connected cooperad is coaugmented
since the identity $\DOp(1) = \kk$
determines a cooperad morphism $\eta: \IOp\rightarrow\DOp$.
The connected $\Sigma_*$-object $\tilde{\DOp}$
such that
\begin{equation*}
\tilde{\DOp}(r) = \begin{cases} 0, & \text{if $r = 0,1$}, \\ \DOp(r), & \text{otherwise}, \end{cases}
\end{equation*}
represents the coaugmentation coideal of~$\DOp$.

In our constructions,
we use that the structure of a connected cooperad is equivalent to a collection of tree coproducts,
which are dual to the tree composition products of an operad.
This representation
is reviewed in the next paragraphs.

\subsubsection{The tree comonad}\label{CobarConstruction:TreeComonad}
The reduced module of tree tensors
\begin{equation*}
\tilde{\FOp}(M)(I) = \bigoplus_{\tau\in\tilde{\Theta}(I)} \tau(M)/\equiv
\end{equation*}
which represents the augmentation ideal of the free operad
comes also equipped with a comonad structure,
determined by a counit $\epsilon: \tilde{\FOp}(M)\rightarrow M$
and a coproduct $\nu: \tilde{\FOp}(M)\rightarrow\tilde{\FOp}(\tilde{\FOp}(M))$.
The counit $\epsilon: \tilde{\FOp}(M)\rightarrow M$
is given by the projection
onto the summand $\psi(M)$
associated to the one-vertex tree $\psi$ of~\S\ref{TreeOperadStructures:FreeOperadConstruction}.
Recall that the elements of~$\tilde{\FOp}(\tilde{\FOp}(M))$
are equivalent to tensors arranged on a big tree
equipped with a partition into small subtrees (see~\S\ref{TreeOperadStructures:ReducedFreeOperad}).
The coproduct $\nu: \tilde{\FOp}(M)\rightarrow\tilde{\FOp}(\tilde{\FOp}(M))$
is given on each component $\tau(M)$
by the sum over all partitions of~$\tau$
into small subtrees.

The structure of a coalgebra over $\tilde{\FOp}$
consists of a $\Sigma_*$-object $\tilde{\DOp}$
together with a coproduct $\rho: \tilde{\DOp}\rightarrow \FOp(\tilde{\DOp})$
which satisfy natural unit and associativity relations
with respect to the comonad structure of~$\tilde{\FOp}$.
From the definition of~$\tilde{\FOp}$,
it appears that the structure of an algebra over~$\tilde{\FOp}$
is fully determined by a collection of dg-module morphisms
$\rho_{\tau}: \tilde{\DOp}(I)\rightarrow\tau(\tilde{\DOp})$, $\tau\in\tilde{\Theta}(I)$,
commuting with the action of $I$-tree isomorphisms, with the action of reindexing bijections,
and so that:
\begin{enumerate}
\item\label{CounitTreeComposites}
for the one-vertex tree $\psi$ of~\S\ref{TreeOperadStructures:FreeOperadConstruction},
the morphism
\begin{equation*}
\rho_{\psi}: \tilde{\DOp}(I)\rightarrow\psi(\tilde{\DOp})
\end{equation*}
reduces to the canonical isomorphism $\tilde{\DOp}(I)\simeq\psi(\tilde{\DOp})$;
\item\label{CoassociativityTreeComposites}
for any collection of blow-ups of vertices $v$ of a tree $\sigma$
into subtrees $\tau_v$ of a big tree~$\tau$,
the composite
\begin{equation*}
\xymatrix{ \tilde{\DOp}(I)\ar[r]^{\rho_{\sigma}} & \sigma(\tilde{\DOp})\ar[r]^{\sigma(\rho_*)} & \tau(\tilde{\DOp}) },
\end{equation*}
where $\sigma(\rho_*)$ refers to the evaluation of the morphisms $\rho_{\tau_v}: \tilde{\DOp}(I_v)\rightarrow\tau_v(\tilde{\DOp})$
on the vertices of~$\sigma$,
agrees with $\rho_{\tau}: \tilde{\DOp}(I)\rightarrow\tau(\tilde{\DOp})$.
\end{enumerate}
Note simply that the morphisms $\rho_\tau: \tilde{\DOp}(I)\rightarrow\tau(\tilde{\DOp})$
sum up to a well-defined morphism $\rho: \tilde{\DOp}(I)\rightarrow\tilde{F}(\DOp)(I)$
though $\tilde{F}(\DOp)$ is defined by a direct sum,
because the assumption $\tilde{\DOp}(0) = \tilde{\DOp}(1) = 0$
implies that, for a fixed finite set $I$,
the dg-module $\tilde{F}(\DOp)(I)$
reduces to a sum over a finite subset of~$\Theta(I)$.

The $I$-trees $\tau\in\Theta(I)$
such that $|I_v|\geq 1$ for every $v\in V(\tau)$
have no automorphisms.
This property occurs crucially in the proof of the equivalence
between cooperads and coalgebras over~$\tilde{\FOp}$
(see next).

We represent the tree coproducts of an element $\gamma\in\tilde{\DOp}(I)$
by expressions such that:
\begin{equation*}
\rho_{\tau}
\left\{\vcenter{\xymatrix@H=6pt@W=4pt@M=2pt@R=8pt@C=4pt{ i\ar[drrr] && j\ar[dr] && k\ar[dl] && l\ar[dlll] \\
&&& *+<8pt>[F]{\gamma}\ar[d] &&& \\
&&& 0 &&& }}\right\}
= \sum_{\rho_{\tau}(\gamma)}
\left\{\vcenter{\xymatrix@H=6pt@W=4pt@M=2pt@R=8pt@C=4pt{ i\ar[dr] & & j\ar[dl] && k\ar[dr] & & l\ar[dl] \\
& *+<6pt>[F]{\gamma_*}\ar[drr] &&&& *+<6pt>[F]{\gamma_*}\ar[dll] & \\
&&& *+<6pt>[F]{\gamma_*}\ar[d] &&& \\
&&& 0 &&& }}\right\}.
\end{equation*}
The notation $\tau$ in the summation set refers to the tree represented in the expansion of the formula.
Again,
we simply use the notation $\gamma_*$
to refer to the factors of the coproduct of~$\gamma$.
If we sum up tree coproducts $\rho_*(\gamma)\in\tau(\tilde{\DOp})$ over a certain set of trees $\tau\in\Psi(I)$,
then we add this summation set $\Psi(I)$
to the sum notation.

\subsubsection{Cooperads as coalgebras over the tree comonad}\label{CobarConstruction:TreeCooperadStructure}
The constructions of~\S\ref{TreeOperadStructures:ReducedCompositionStructure}
can be dualized.

For an element $\gamma\in\tilde{\DOp}(I)$,
we form the tree coproducts
\begin{equation*}
\sum_{\substack{\tau\in\Psi_2(I)\\ \rho_{\tau}(\gamma)}}
\left\{\vcenter{\xymatrix@H=6pt@W=4pt@M=2pt@R=8pt@C=4pt{ &
i_*\ar[dr] & \cdots\ar@{}[d]|{\displaystyle{\cdots}} & i_*\ar[dl] &
& i_*\ar[dr] & \cdots\ar@{}[d]|{\displaystyle{\cdots}} & i_*\ar[dl] & \\
i_*\ar@/_4pt/[drrrr] &
& *+<6pt>[F]{\gamma_*}\ar[drr] &
& \cdots &
& *+<6pt>[F]{\gamma_*}\ar[dll] &
& i_*\ar@/^4pt/[dllll] \\
&&&& *+<6pt>[F]{\gamma_*}\ar[d] &&&& \\ &&&& 0 &&&& }}\right\},
\end{equation*}
over the set of trees of height $2$.
Insert unit vertices to obtain the two level tree tensors
associated to these elements (use the correspondence of~\S\ref{TreeOperadStructures:ReducedCompositionStructure})
and add the unital coproducts
\begin{equation*}
\vcenter{\xymatrix@H=6pt@W=4pt@M=2pt@R=8pt@C=4pt{ \ar@{.}[r] & i_1\ar[d]\ar@{.}[r] & *{\cdots}\ar@{.}[r] & i_n\ar[d]\ar@{.}[r] & \\
*+<2pt>{1}\ar@{.}[r] & *+<6pt>[F]{1}\ar[dr]\ar@{.}[r] & *{\cdots}\ar@{.}[r] & *+<6pt>[F]{1}\ar[dl]\ar@{.}[r] & \\
*+<2pt>{0}\ar@{.}[rr] && *+<6pt>[F]{\gamma}\ar[d]\ar@{.}[rr] && \\ \ar@{.}[rr] && 0\ar@{.}[rr] && }}
\quad\text{and}
\quad\vcenter{\xymatrix@H=6pt@W=4pt@M=2pt@R=8pt@C=4pt{ i_1\ar[dr]\ar@{.}[r] & *{\cdots}\ar@{.}[r] & i_n\ar[dl] \\
*+<2pt>{1}\ar@{.}[r] & *+<6pt>[F]{\gamma}\ar[d]\ar@{.}[r] & \\
*+<2pt>{0}\ar@{.}[r] & *+<6pt>[F]{1}\ar[d]\ar@{.}[r] & \\ \ar@{.}[r] & 0\ar@{.}[r] & }},
\end{equation*}
to form the expansion of~$\nu(\gamma)$
in $\DOp\circ\DOp$.

Under the assumption $\tilde{\DOp}(0) = \tilde{\DOp}(1) = 0$,
the associativity relations of~\S\ref{CobarConstruction:TreeComonad}
imply that the coproduct $\nu: \DOp\rightarrow\DOp\circ\DOp$
defined by the construction of this paragraph
is coassociative.
The coproduct $\nu: \DOp\rightarrow\DOp\circ\DOp$
is also counital by construction.
Hence,
we obtain:

\begin{obsv}
The structure of a coalgebra over $\tilde{\FOp}$ includes the structure of a coaugmented cooperad.
\end{obsv}

The equivalence of~\S\ref{TreeOperadStructures}
can be dualized too.\footnote{Note however that the assumption $\tilde{\DOp}(0) = 0$
is also needed to define partial composition coproducts
$\rho_e: \tilde{\DOp}(I\setminus\{e\}\cup J)\rightarrow\tilde{\DOp}(I)\otimes\tilde{\DOp}(J)$
from a cooperad coproduct $\nu: \DOp\rightarrow\DOp\circ\DOp$.}
Therefore,
we have finally:

\begin{prop}
The category of coalgebras over $\tilde{\FOp}$ is equivalent to the category of coaugmented cooperads.\qed
\end{prop}

\subsubsection{The cobar construction of a cooperad}\label{CobarConstruction:Definition}
The cobar construction of a (connected) cooperad $\DOp$
is a quasi-free operad $B^c(\DOp)$
such that
\begin{equation*}
B^c(\DOp) = (\FOp(\Sigma^{-1}\tilde{\DOp}),\partial_{\beta}),
\end{equation*}
where $\Sigma^{-1}\tilde{\DOp}$
is the desuspension of the coaugmentation coideal of~$\DOp$.
The desuspension $\Sigma^{-1} M$ of a $\Sigma_*$-object $M$
is defined by tensor products
\begin{equation*}
(\Sigma^{-1} M)(r) = \kk\sigma\otimes M(r),
\end{equation*}
where $\sigma$ is a homogeneous element of degree $-1$
and $\delta(\sigma) = 0$.
The purpose of this paragraph is to review the definition
of the twisting derivation $\partial_{\beta}$.

Recall that $\Theta_2(I)$ denotes the groupoid of $I$-trees with $2$ vertices.
The dg-module of tree tensors associated to an $I$-tree $\tau\in\Theta_2(I)$
with vertices $V(\tau) = \{u,v\}$
reads $\tau(M) = M(I_u)\otimes M(I_v)$.
By symmetry of tensor products,
this dg-module of tree tensors satisfies the commutation relation
\begin{equation*}
\tau(\Sigma^{-1} M)
= (\Sigma^{-1} M(I_u))\otimes(\Sigma^{-1} M(I_v))
\simeq\Sigma^{-2}(M(I_u)\otimes M(I_v))
= \Sigma^{-2}\tau(M)
\end{equation*}
with respect to desuspensions.
Accordingly,
for the coaugmentation coideal of a connected cooperad $M = \tilde{\DOp}$,
the tree coproducts $\rho_{\tau}: \tilde{\DOp}(I)\rightarrow\tau(\tilde{\DOp})$
associated to trees with two vertices $\tau\in\Theta_2(I)$
give rise to homomorphisms
of degree $-1$:
\begin{equation*}
\Sigma^{-1}\tilde{\DOp}(I)\xrightarrow{\rho_{\tau}}\tau(\Sigma^{-1}\tilde{\DOp}).
\end{equation*}
The homomorphism
$\beta: \Sigma^{-1}\tilde{\DOp}\rightarrow\FOp(\Sigma^{-1}\tilde{\DOp})$
that determines the twisting derivation of the cobar construction
is formed by the sum of these homomorphisms.
Graphically,
the expansion of~$\beta$
reads:
\begin{equation*}
\beta\left\{\vcenter{\xymatrix@H=6pt@W=4pt@M=2pt@R=8pt@C=4pt{ i_1\ar[dr] & \cdots & i_n\ar[dl] \\
& *+<6pt>[F]{\sigma\otimes\gamma}\ar[d] & \\
& 0 & }}\right\}
= \sum_{\substack{\tau\in\Theta_2(I)\\ \rho_{\tau}(\gamma)}}
\pm\left\{\vcenter{\xymatrix@H=6pt@W=4pt@M=2pt@R=8pt@C=4pt{ & i_*\ar[dr] & \cdots & i_*\ar[dl] & \\
i_*\ar[drr] & \cdots & *+<6pt>[F]{\sigma\otimes\gamma_*}\ar[d] & \cdots & i_*\ar[dll] \\
&& *+<6pt>[F]{\sigma\otimes\gamma_*}\ar[d] && \\
&& 0 && }}\right\},
\end{equation*}
for any element $\sigma\otimes\gamma\in\Sigma^{-1}\tilde{\DOp}(I)$.
The sign $\pm$
is produced by the commutation of a suspension factor $\sigma$
with a factor $\gamma_*$:
\begin{equation*}
\sigma\otimes\gamma\mapsto\sigma^2\otimes\gamma_*\otimes\gamma_*\mapsto\sigma\otimes\gamma_*\otimes\sigma\otimes\gamma_*.
\end{equation*}
In order to obtain the right signs in the expansion of the operad derivation $\partial_{\beta}$
determined by the homomorphism $\beta$,
it is crucial to patch the tensors $\sigma\otimes\gamma_*\in\Sigma^{-1}\tilde{\DOp}$.

The commutation of tree coproducts with differentials
implies $\delta(\beta) = 0$.
The associativity properties of two vertex tree coproducts
(dualize the relations of figures~\ref{Fig:TreeLinearAssociativity}-\ref{Fig:TreeRamifiedAssociativity})
imply:

\begin{claim}
The derivation $\partial_{\beta}$ satisfies the identity $\partial_{\beta}\cdot\beta = 0$.\qed
\end{claim}

Note simply that a permutation of homogeneous elements $\sigma$ reverses a sign in the associativity relation
and transforms the associativity identity into a vanishing relation.

By proposition~\ref{QuasiFreeOperads:TwistingDerivationConstruction},
the relations $\delta(\beta) = \partial_{\beta}\cdot\beta = 0$ imply that $\partial_{\beta}$
satisfies the equation of twisting homomorphisms.
Hence,
we have a well-defined twisted dg-operad $B^c(\DOp) = (\FOp(\Sigma^{-1}\tilde{\DOp}),\partial_{\beta})$.

\subsubsection{Operadic twisting cochains}\label{CobarConstruction:TwistingCochains}
According to proposition~\ref{QuasiFreeOperads:MorphismConstruction}
an operad morphism $\phi: B^c(\DOp)\rightarrow\POp$
is equivalent to a homomorphism of degree $0$
\begin{equation*}
\Sigma^{-1}\tilde{\DOp}\xrightarrow{\theta}\POp
\end{equation*}
such that $\delta(\theta) = \phi_{\theta}\cdot\beta$.
Since the cooperad $\tilde{\DOp}$ is connected,
we have automatically $\theta(\Sigma^{-1}\tilde{\DOp})\subset\tilde{\POp}$.
The homomorphism $\phi_{\theta}\cdot\beta$
is defined by the composite
\begin{equation*}
\Sigma^{-1}\tilde{\DOp}(I)\xrightarrow{\beta}\bigoplus_{\tau}\tau(\Sigma^{-1}\tilde{\DOp})
\xrightarrow{\tau(\theta)}\bigoplus_{\tau}\tau(\tilde{\POp})
\xrightarrow{\lambda_*}\tilde{\POp},
\end{equation*}
where~$\lambda_*$ refers to the tree composition products of~$\POp$
and the homomorphism~$\tau(\theta)$
applies~$\theta$ to all factors of the tree tensor product over~$\tau$.

Naturally,
a homomorphism of degree $0$ on a desuspension~$\Sigma^{-1}\tilde{\DOp}$
is equivalent to a homomorphism of degree $-1$
on~$\tilde{\DOp}$.
For simplicity,
we also use the notation $\theta$
to refer to the homomorphism of degree $-1$
equivalent to $\theta: \Sigma^{-1}\tilde{\DOp}\rightarrow\POp$.

By definition of~$\beta$,
this equation $\delta(\theta) = \phi_{\theta}\cdot\beta$
is equivalent to the identity
\begin{equation}
\delta(\theta)\left\{\vcenter{\xymatrix@H=6pt@W=4pt@M=2pt@R=8pt@C=4pt{ i_1\ar[dr] & \cdots & i_n\ar[dl] \\
& *+<6pt>[F]{\gamma}\ar[d] & \\
& 0 & }}\right\}
= \sum_{\substack{\tau\in\Theta_2(I)\\ \rho_{\tau}(\gamma)}}
\pm\lambda_*\left\{\vcenter{\xymatrix@H=6pt@W=4pt@M=2pt@R=8pt@C=4pt{ & i_*\ar[dr] & \cdots & i_*\ar[dl] & \\
i_*\ar[drr] & \cdots & *+<6pt>[F]{\theta(\gamma_*)}\ar[d] & \cdots & i_*\ar[dll] \\
&& *+<6pt>[F]{\theta(\gamma_*)}\ar[d] && \\
&& 0 && }}\right\}
\end{equation}
in $\POp$, for every $\gamma\in\DOp$.
In this expression,
the sign $\pm$
is produced by the commutation of the homomorphism $\theta: \tilde{\DOp}\rightarrow\tilde{\POp}$
of degree $-1$
with a factor $\gamma_*$.
Naturally,
this sign agrees with the sign produced by the differential of~$\gamma$
when we take the homomorphism of degree $0$
equivalent to~$\theta$.

In the literature (see~\cite[\S 2.3]{GetzlerJones}),
the homomorphisms $\theta: \tilde{\DOp}\rightarrow\POp$
of degree $-1$
such that (*) holds
are called (operadic) twisting cochains.
Usually,
a twisting cochain is supposed to be defined on the cooperad $\DOp$ as a whole
with the convention that $\theta(1) = 0$
for the unit element $1\in\DOp(1)$.
The set of twisting cochains, denoted by $\Tw_{\Op}(\DOp,\POp)$
forms a bifunctor of $\DOp$ and $\POp$
and the correspondence between twisting cochains $\theta: \tilde{\DOp}\rightarrow\POp$
and morphisms $\phi_{\theta}: B^c(\DOp)\rightarrow\POp$
gives a natural bijection
\begin{equation*}
\Mor_{\Op}(B^c(\DOp),\POp)\simeq\Tw_{\Op}(\DOp,\POp).
\end{equation*}

The homomorphism of degree $-1$
\begin{equation*}
\tilde{\DOp}\xrightarrow{\iota}\underbrace{(\FOp(\Sigma^{-1}\tilde{\DOp}),\partial_{\beta})}_{B^c(\DOp)}
\end{equation*}
determined to the universal morphism $\eta: \Sigma^{-1}\tilde{\DOp}\rightarrow\FOp(\Sigma^{-1}\tilde{\DOp})$
forms a universal twisting cochain
in the sense that the associated morphism $\phi_{\iota}: B^c(\DOp)\rightarrow B^c(\DOp)$
is the identity of~$B^c(\DOp)$.

\subsubsection{Operad-cooperad twisted complexes}\label{CobarConstruction:OperadCoperadTwistedComplexes}
A twisting homomorphism $\partial_{\theta}: \POp\circ\DOp\rightarrow\POp\circ\DOp$
is naturally associated to any twisting cochain $\theta\in\Tw_{\Op}(\DOp,\POp)$.
In~\S\ref{QuasiFreeReplacements},
we use the twisted composite $\Sigma_*$-object $(\POp\circ\DOp,\partial_{\theta})$
determined by this twisting homomorphism $\partial_{\theta}: \POp\circ\DOp\rightarrow\POp\circ\DOp$
to define natural quasi-free replacements
in the category of $\POp$-algebras, for certain operads $\POp$.
For the moment,
we focus on intrinsic properties of the $\Sigma_*$-object $(\POp\circ\DOp,\partial_{\theta})$
and we review the definition of the twisting homomorphism $\partial_{\theta}$.

The homomorphism $\partial_{\theta}: \POp\circ\DOp\rightarrow\POp\circ\DOp$
splits in two parts $\partial_{\theta} = \partial_{\theta}^0 + \partial_{\theta}^1$.
The image of a two level tree tensor $\varpi$ under $\partial_{\theta}^1: \POp\circ\DOp\rightarrow\POp\circ\DOp$
is determined by the operations represented in figure~\ref{Fig:CompositeTwistingDifferential}.
In~(1), we perform the two-vertex composition coproduct $\lambda_*(\gamma_i)$
of a factor~$\gamma_i$
such that $\gamma_i\in\tilde{\DOp}$.
In~(2), we apply the twisting cochain $\theta$ to the bottom factor $(\gamma_i)_*$ of the coproduct of~$\gamma_i$.
In~(3), we apply the two-vertex composition product of the operad $\POp$
to retrieve an element of~$\POp\circ\DOp$.
Unital vertices
$\xymatrix@H=6pt@W=3pt@M=2pt@!R=1pt@!C=1pt{ \ar[r] & *+<3mm>[o][F]{1}\ar[r] & }$,
labeled by unit elements $1\in\IOp(1)$,
have simply to be added on the edges $e$ going directly from a tree input $s(e) = i$ to the vertex at level $0$
to have a well defined element of~$\POp\circ\DOp$.
Sum up over all factors $\gamma_i$ of the coaugmentation coideal $\tilde{\DOp}$
to obtain the expansion of~$\partial_{\theta}^1(\varpi)$.
\begin{figure}
\begin{align*} 
& \left\{\vcenter{\xymatrix@M=3pt@H=4pt@W=3pt@R=8pt@C=4pt{ \ar@{.}[r] &
i_*\ar[dr]\ar@{.}[r] & \cdots\ar@{.}[r]\ar@{}[d]|{\displaystyle{\cdots}} & i_*\ar[dl]\ar@{.}[r] &
\cdots\ar@{.}[r] & i_*\ar[dr]\ar@{.}[r] & \cdots\ar@{.}[r]\ar@{}[d]|{\displaystyle{\cdots}} & i_*\ar[dl]\ar@{.}[r] & \\
*+<2pt>{1}\ar@{.}[rr] && *+<8pt>[F]{\gamma_1}\ar[drr]\ar@{.}[rr] && \cdots\ar@{.}[rr] && *+<8pt>[F]{\gamma_r}\ar[dll]\ar@{.}[rr] && \\
*+<2pt>{0}\ar@{.}[rrrr] &&&& *+<8pt>[F]{p}\ar[d]\ar@{.}[rrrr] &&&& \\
\ar@{.}[rrrr] &&&& 0\ar@{.}[rrrr] &&&& }}\right\}
\\
\xrightarrow{(1)} \sum_{\substack{\tau\in\Theta_2(I)\\ \rho_{\tau}(\gamma_i)}}\pm &
\left\{\vcenter{\xymatrix@M=3pt@H=4pt@W=3pt@R=8pt@C=4pt{ \ar@{.}[r] &
i_*\ar[dr]\ar@{.}[r] & \cdots\ar@{.}[r] & i_*\ar[dl]\ar@{.}[r] &
i_*\ar[ddrr]\ar@{.}[r] & i_*\ar[dr]\ar@{.}[r] & \cdots\ar@{.}[r] & i_*\ar[dl]\ar@{.}[r] & i_*\ar[ddll]\ar@{.}[r] &
i_*\ar[dr]\ar@{.}[r] & \cdots\ar@{.}[r] & i_*\ar[dl]\ar@{.}[r] & \\
*+<2pt>{1}\ar@{.}[rr] && *+<8pt>[F]{\gamma_1}\ar@/_8pt/[ddrrrr]\ar@{.}[rr]
&& \cdots
&& *+<8pt>[F]{(\gamma_i)_*}\ar[d]\save[]!C.[d]!C *+<8pt>[F-,]\frm{}\restore\ar@{.}[]!R;[rr]\ar@{.}[]!L;[ll] &&
\cdots\ar@{.}[rr] && *+<8pt>[F]{\gamma_r}\ar@/^8pt/[ddllll]\ar@{.}[rr] && \\
&&&&&& *+<8pt>[F]{(\gamma_i)_*}\ar[d] &&&&&& \\
*+<2pt>{0}\ar@{.}[rrrrrr] &&&&&& *+<8pt>[F]{p}\ar[d]\ar@{.}[rrrrrr] &&&&&& \\
\ar@{.}[rrrrrr] &&&&&& 0\ar@{.}[rrrrrr] &&&&&& }}\right\}
\\
\xrightarrow{(2)} \sum_{\substack{\tau\in\Theta_2(I)\\ \rho_{\tau}(\gamma_i)}}\pm &
\left\{\vcenter{\xymatrix@M=3pt@H=4pt@W=3pt@R=8pt@C=4pt{ \ar@{.}[r] &
i_*\ar[dr]\ar@{.}[r] & \cdots\ar@{.}[r] & i_*\ar[dl]\ar@{.}[r] &
i_*\ar[ddrr]\ar@{.}[r] & i_*\ar[dr]\ar@{.}[r] & \cdots\ar@{.}[r] & i_*\ar[dl]\ar@{.}[r] & i_*\ar[ddll]\ar@{.}[r] &
i_*\ar[dr]\ar@{.}[r] & \cdots\ar@{.}[r] & i_*\ar[dl]\ar@{.}[r] & \\
*+<2pt>{1}\ar@{.}[rr] && *+<8pt>[F]{\gamma_1}\ar@/_8pt/[ddrrrr]\ar@{.}[rr]
&& \cdots
&& *+<8pt>[F]{(\gamma_i)_*}\ar[d]\ar@{.}[rr]\ar@{.}[ll] &&
\cdots\ar@{.}[rr] && *+<8pt>[F]{\gamma_r}\ar@/^8pt/[ddllll]\ar@{.}[rr] && \\
&&&&&& *+<8pt>[F]{\theta((\gamma_i)_*)}\ar[d] &&&&&& \\
*+<2pt>{0}\ar@{.}[rrrrrr] &&&&&& *+<8pt>[F]{p}\ar[d]\ar@{.}[rrrrrr] &&&&&& \\
\ar@{.}[rrrrrr] &&&&&& 0\ar@{.}[rrrrrr] &&&&&& }}\right\}
\\
\xrightarrow{(3)} \sum_{\substack{\tau\in\Theta_2(I)\\ \rho_{\tau}(\gamma_i)}}\pm &
\left\{\vcenter{\xymatrix@M=3pt@H=4pt@W=3pt@R=8pt@C=4pt{ \ar@{.}[r] &
i_*\ar[dr]\ar@{.}[r] & \cdots\ar@{.}[r] & i_*\ar[dl]\ar@{.}[r] &
i_*\ar[d]\ar@{.}[r] & i_*\ar[dr]\ar@{.}[r] & \cdots\ar@{.}[r] & i_*\ar[dl]\ar@{.}[r] & i_*\ar[d]\ar@{.}[r] &
i_*\ar[dr]\ar@{.}[r] & \cdots\ar@{.}[r] & i_*\ar[dl]\ar@{.}[r] & \\
*+<2pt>{1}\ar@{.}[rr] && *+<8pt>[F]{\gamma_1}\ar@/_8pt/[ddrrrr]\ar@{.}[rr]
&& *+<8pt>[F]{1}\ar@/_4pt/[drr]\ar@{.}[r] & \cdots
& *+<8pt>[F]{(\gamma_i)_*}\ar[d]\ar@{.}[r]\ar@{.}[l] &
\cdots\ar@{.}[r] &
*+<8pt>[F]{1}\ar@/^4pt/[dll]\ar@{.}[rr] && *+<8pt>[F]{\gamma_r}\ar@/^8pt/[ddllll]\ar@{.}[rr] && \\
&&&&&& *+<8pt>[F]{\theta((\gamma_i)_*)}\ar[d]\save[]!C.[d]!C *+<4pt>[F-,]\frm{}*+<6pt>\frm{\{}*+<0pt>\frm{\}}
\restore\ar@{}[]!L-<8pt,0pt>;[d]!L-<8pt,0pt>_(0.3){\displaystyle{\lambda_*}} &&&&&& \\
*+<2pt>{0}\ar@{.}[rrrrrr] &&&&&& *+<8pt>[F]{p}\ar[d]\ar@{.}[rrrrrr] &&&&&& \\
\ar@{.}[rrrrrr] &&&&&& 0\ar@{.}[rrrrrr] &&&&&& }}\right\}
\end{align*}
\caption{}\label{Fig:CompositeTwistingDifferential}\end{figure}

The homomorphism $\partial_{\theta}^0$ is given by the same process,
but we do not perform the composition coproduct and simply push the factor $\gamma_i\in\tilde{\DOp}$
to the bottom.

By a straightforward inspection, we check that:

\begin{claim}
The homomorphism $\partial_{\theta} = \partial_{\theta}^0 + \partial_{\theta}^1$
satisfies the equation of twisting homomorphisms $\delta(\partial_{\theta}) + \partial_{\theta}^2 = 0$.\qed
\end{claim}

Hence, we have a well-defined twisted complex $(\POp\circ\DOp,\partial_{\theta})$.
Note moreover that the augmentations $\epsilon: \POp\rightarrow\IOp$ and $\epsilon: \DOp\rightarrow\IOp$
induce a natural morphism $\epsilon: (\POp\circ\DOp,\partial_{\theta})\rightarrow\IOp$.

The applications of twisted $\Sigma_*$-objects $(\POp\circ\DOp,\partial_{\theta})$
are derived from the following result:

\begin{thm}\label{CobarConstruction:UniversalAcyclicity}
Let $\DOp$ be a $\C$-cofibrant connected cooperad.
The twisted composite $\Sigma_*$-object
\begin{equation*}
(B^c(\DOp)\circ\DOp,\partial_{\iota})
\end{equation*}
associated to the universal twisting cochain $\iota: \DOp\rightarrow B^c(\DOp)$
is weakly equivalent to the unit $\Sigma_*$-object $\IOp$.
The weak-equivalence is defined by the augmentation
\begin{equation*}
\epsilon: (B^c(\DOp)\circ\DOp,\partial_{\iota})\rightarrow \IOp.
\end{equation*}
\end{thm}

The twisted complex $(B^c(\DOp)\circ\DOp,\partial_{\iota})$
can be identified with the two-sided cobar complex $B^c(I,\DOp,\DOp)$
of~\cite[\S 4]{FresseKoszulDuality}.
The acyclicity of a dual operadic complex~$B(I,\POp,\POp)$
is proved in~\cite[\S 4.7]{FresseKoszulDuality}
and in~\cite[\S 2.3]{GetzlerJones}.
In~\cite[\S 4]{FresseKoszulDuality},
we give arguments to prove that $B^c(I,\DOp,\DOp)$
is acyclic when the cooperad $\DOp$ is non-negatively graded.
In the following proof,
we adapt the argument of~\cite{GetzlerJones}
in order to extend the result to the case of $\ZZ$-graded dg-cooperads.

\begin{proof}
Since we assume $\DOp(0) = 0$ and $\DOp(1) = \kk$,
the augmentation $\epsilon: (B^c(\DOp)\circ\DOp,\partial_{\theta})\rightarrow\IOp$
reduces to an identity isomorphism in arity $1$.
Therefore we are reduced to check that the complex $(B^c(\DOp)\circ\DOp,\partial_{\theta})$
is acyclic in arity $r$, for any $r\geq 2$.

The elements of the composite $\Sigma_*$-object $B^c(\DOp)\circ\DOp$
are represented by tensors arranged on a two level tree
such that the vertex at level $0$
consists itself of a tree tensor representing an element of~$B^c(\DOp) = (F(\Sigma^{-1}\tilde{\DOp}),\partial_{\beta})$.
From this representation,
we see that an element of $B^c(\DOp)\circ\DOp$ is equivalent to a tensor $\varpi\in\tau(\DOp)$
arranged on a big tree with two level of vertices
such that:
\begin{enumerate}\setcounter{enumi}{-1}
\item
the ingoing edges of the tree $e$
target to a vertex at level $1$;
\item
the edges $e$ arising from a vertex at level $1$
target to a vertex at level $0$
or the output of the tree;
\item
the edges $e$ arising from a vertex at level $0$
target to a vertex at level $0$
or the output of the tree;
\item
the vertices $v$ at level $1$ are either labeled by an element $\gamma_v\in\tilde{\DOp}(I_v)$
of the coaugmentation coideal of~$\DOp$
or by a unit element $1\in\IOp(I_v)$;
\item
the vertices $v$ at level $0$ are labeled by elements $\sigma\otimes\gamma_v\in\Sigma^{-1}\tilde{\DOp}(I_v)$
in the desuspension of the coaugmentation coideal of~$\DOp$.
\end{enumerate}
An example is displayed in figure~\ref{Fig:TreeTensorComposite}.
\begin{figure}
\begin{equation*}
\xymatrix@M=4pt@H=4pt@W=6pt@R=8pt@C=8pt{ \ar@{.}[r] & i\ar[d]\ar@{.}[rr] && j\ar[d]\ar@{.}[r] &
k\ar[dr]\ar@{.}[rr] && l\ar[dl]\ar@{.}[r] & \\
*+<2pt>{1}\ar@{.}[r] & *+<8pt>[F]{1}\ar[dr]\ar@{.}[rr] && *+<8pt>[F]{1}\ar[dl]\ar@{.}[rr] &&
*+<8pt>[F]{\gamma_3}\ar[ddl]\ar@{.}[rr] && \\
\ar@{}[d]|{\displaystyle 0} & \save[]!C.[r]!C.[drrrr]!C *+<4pt>\frm{\{}\restore & *+<8pt>[F]{\sigma\otimes\gamma_2}\ar[drr] &&&&& \\
&&&& *+<8pt>[F]{\sigma\otimes\gamma_1}\ar[d] &&& \\
\ar@{.}[rrrr] &&&& 0\ar@{.}[rrr] &&& }
\end{equation*}
\caption{}\label{Fig:TreeTensorComposite}\end{figure}
The tree tensor defined by the subtree of vertices at level $0$
represents the factor $p\in B^c(\DOp)$ of the composite $B^c(\DOp)\circ\DOp$.
If the set of vertices at level $0$ is empty,
then this factor is the unit element $1\in B^c(\DOp)$.

Formally,
we consider the groupoid $\Pi(I)$ of pairs $(\xi,l)$,
where $\xi$ is an $I$-tree and $l$ is a map $l: V(\xi)\rightarrow\{0,1\}$
which assigns a level to the vertices of~$\xi$
so that (1-2) are satisfied.
The isomorphisms of~$\Pi(I)$ are the isomorphisms of $I$-trees
which preserve level structures.
To a level tree $(\xi,l)\in\Pi(I)$,
we associate the tensor product
\begin{equation*}
\xi(\DOp,l) = \Bigl\{\bigotimes_{v\in l^{-1}(0)} \Sigma^{-1}\tilde{\DOp}(I_v)\Bigr\}
\otimes\Bigl\{\bigotimes_{v\in l^{-1}(1)} \DOp(I_v)\Bigr\}.
\end{equation*}
The composite $\Sigma_*$-object $B^c(\DOp)\circ\DOp$
has an expansion of the form
\begin{equation*}
B^c(\DOp)\circ\DOp(I) = \bigoplus_{(\xi,l)\in\Pi(I)} \xi(\DOp,l)/\equiv,
\end{equation*}
where the sum is divided out by the action of isomorphisms.

The correspondence of~\S\ref{TreeOperadStructures:ReducedCompositionStructure}
between trees with two levels and trees of height $2$
can be generalized to the level trees of~$\Pi(I)$.
Let $\Lambda(I)$ be the groupoid of pairs $(\tau,l)$,
where $\tau$ is an $I$-tree and $l$ is a map $l: V(\xi)\rightarrow\{0,1\}$
such that (1-2) are satisfied
but not necessarily (0).
To a level tree of~$\Lambda(I)$,
we associate a module of tree tensors, defined like $\xi(\DOp,l)$,
but where every factor is associated to an element of the coaugmentation coideal of~$\DOp$:
\begin{equation*}
\tau(\tilde{\DOp},l) = \Bigl\{\bigotimes_{v\in l^{-1}(0)} \Sigma^{-1}\tilde{\DOp}(I_v)\Bigr\}
\otimes\Bigl\{\bigotimes_{v\in l^{-1}(1)} \tilde{\DOp}(I_v)\Bigr\}
\end{equation*}

To a reduced level tree $(\tau,l)\in\Lambda(I)$,
we can associate a tree with two levels $(\widehat{\tau},\widehat{l})\in\Pi(I)$
defined by the insertion of unital vertices $\xymatrix@H=6pt@W=3pt@M=2pt@!R=1pt@!C=1pt{ \ar[r] & *+<3mm>[o][F]{1}\ar[r] & }$
on edges $e$ going directly from a tree input $s(e) = i$ to a vertex at level $0$.
The tree tensors of~$\tau(\tilde{\DOp},l)$
are associated to elements of~$\widehat{\tau}(\DOp,\widehat{l})$
with unit elements $1\in\IOp(1)$
on the inserted unital vertices $\xymatrix@H=6pt@W=3pt@M=2pt@!R=1pt@!C=1pt{ \ar[r] & *+<3mm>[o][F]{1}\ar[r] & }$.
As an example,
the tree tensor of figure~\ref{Fig:ReducedTreeTensorComposite}
is associated to the reduced tree tensor of figure~\ref{Fig:TreeTensorComposite}.
\begin{figure}
\begin{equation*}
\xymatrix@M=4pt@H=4pt@W=6pt@R=8pt@C=8pt{ \ar@{.}[r] & i\ar@/_/[ddr]\ar@{.}[rr] && j\ar@/^/[ddl]\ar@{.}[r] &
k\ar[dr]\ar@{.}[rr] && l\ar[dl]\ar@{.}[r] & \\
*+<2pt>{1}\ar@{.}[rrrrr] &&&&&
*+<8pt>[F]{\gamma_3}\ar[ddl]\ar@{.}[rr] && \\
\ar@{}[d]|{\displaystyle 0} & \save[]!C.[r]!C.[drrrr]!C *+<4pt>\frm{\{}\restore & *+<8pt>[F]{\sigma\otimes\gamma_2}\ar[drr] &&&&& \\
&&&& *+<8pt>[F]{\sigma\otimes\gamma_1}\ar[d] &&& \\
\ar@{.}[rrrr] &&&& 0\ar@{.}[rrr] &&& }
\end{equation*}
\caption{}\label{Fig:ReducedTreeTensorComposite}\end{figure}
The expansion of~$B^c(\DOp)\circ\DOp(I)$
has a reduction:
\begin{equation*}
B^c(\DOp)\circ\DOp(I) = \bigoplus_{(\tau,l)\in\Lambda(I)} \tau(\tilde{\DOp},l)/\equiv,
\end{equation*}
where $(\tau,l)$
runs over the groupoid $\Lambda(I)$.

Let $F_s B^c(\DOp)\circ\DOp$
be the submodule of $B^c(\DOp)\circ\DOp$ generated by the summands $\tau(\tilde{\DOp},l)$
such that $\tau$ has at least $-s$ vertices.
The total differential of $B^c(\DOp)\circ\DOp$
satisfies clearly
\begin{equation*}
(\delta+\partial_{\iota})(F_s B^c(\DOp)\circ\DOp)\subset F_s B^c(\DOp)\circ\DOp,
\end{equation*}
for every $s\in\ZZ$.
Hence we have a spectral sequence $E^0\Rightarrow H_*(B^c(\DOp)\circ\DOp)$
such that
\begin{equation*}
E^0_{s *} = \bigoplus_{\substack{(\tau,l)\in\Lambda(I)\\ |V(\tau)| = -s}} \tau(\tilde{\DOp},l)/\equiv,
\end{equation*}
where the sum ranges over trees with $-s$ vertices.
This spectral sequence converges in a strong sense
because, for a fixed arity $r$,
the assumption $\tilde{\DOp}(0) = \tilde{\DOp}(1) = 0$
implies that the filtration is bounded:
\begin{equation*}
0 = F_{- s_r} B^c(\DOp)\circ\DOp(r)\subset\dots\subset F_{0} B^c(\DOp)\circ\DOp(r) = B^c(\DOp)\circ\DOp(r).
\end{equation*}

The differential $d^0: E^0\rightarrow E^0$
reduces to terms $\delta: B^c(\DOp)\circ\DOp\rightarrow B^c(\DOp)\circ\DOp$ induced by the internal differential of~$\DOp$
and to the terms $\partial^0_{\iota}$
of the twisting homomorphisms $\partial_{\iota}$
(which push a vertex at level $1$
to the bottom)
since the other part of~$\partial_{\iota}$
clearly decreases filtrations by $1$
as well as the twisting homomorphism $\partial_{\beta}$
of the cobar construction $B^c(\DOp)$.

The summand
\begin{equation*}
E^0_{\tau} = \Bigl\{\bigoplus_l \tau(\tilde{\DOp},l)/\equiv\Bigr\}\subset E^0
\end{equation*}
associated to a given isomorphism class of~$I$-trees $\tau$
is clearly preserved by this differential $d^0 = \delta+\partial^0_{\iota}$.
We define a contracting chain homotopy $h: E^0_{\tau}\rightarrow E^0_{\tau}$
to prove that each summand $(E^0_{\tau},d^0)$ with $|V(\tau)|>0$
is acyclic.

For this purpose,
we fix a vertex $v_1\in V(\tau)$ such that $I_{v_1}\subset I$, the inputs of~$\tau$ (such a vertex always exist)
\footnote{N.B.: As we assume $\tilde{\DOp}(0) = 0$,
we have necessarily $I_{v_1}\not=\emptyset$ (otherwise the dg-module $\tau(\tilde{\DOp})$ vanishes)
and this implies that $v_1$ is fixed under tree isomorphisms.}.
The chain homotopy $h: E^0_{\tau}\rightarrow E^0_{\tau}$ simply raises the vertex $v_1$ to level $1$
for tree tensors $\xi\in\tau(\tilde{\DOp},l)$
such that $l(v_1) = 0$
and cancels the tree tensors $\xi\in\tau(\tilde{\DOp},l)$
for which $l(v_1) = 1$.
As an example,
if the vertex $v_1$ such that $I_{v_1} = \{i,j\}$
is the distinguished vertex
of the tree underlying the tensor of figure~\ref{Fig:ReducedTreeTensorComposite},
then the image of this tree tensor under the chain homotopy $h$
is the tree tensor represented in figure~\ref{Fig:ReducedTreeTensorHomotopy}.
\begin{figure}
\begin{equation*}
\xymatrix@M=4pt@H=4pt@W=6pt@R=8pt@C=8pt{ \ar@{.}[r] & i\ar[dr]\ar@{.}[rr] && j\ar[dl]\ar@{.}[rr] &&
k\ar[dr]\ar@{.}[rr] && l\ar[dl]\ar@{.}[r] & \\
*+<2pt>{1}\ar@{.}[rr] && *+<8pt>[F]{\gamma_2}\ar[drr]\ar@{.}[rrrr] &&&& *+<8pt>[F]{\gamma_3}\ar[dll]\ar@{.}[rr] && \\
0 & \save[]!C.[r]!C.[rrrrr]!C \Biggl\{\restore &&& *+<8pt>[F]{\sigma\otimes\gamma_1}\ar[d] &&&& \\
\ar@{.}[rrrr] &&&& 0\ar@{.}[rrrr] &&&& }
\end{equation*}
\caption{}\label{Fig:ReducedTreeTensorHomotopy}\end{figure}

By an immediate inspection,
we obtain the identities $\delta h + h\delta = 0$ and $\partial^0_{\iota} h + h\partial^0_{\iota} = 0$
from which the acyclicity of~$E^0_{\tau}$ follows.

Finally,
we obtain that $E^1$ reduces to the summand $\IOp$.
Hence,
we conclude that the spectral sequence degenerates and $H_*(B^c(\DOp)\circ\DOp,\partial_{\iota}) = \IOp$.
\end{proof}

To complete the results of this section,
we prove:

\begin{thm}\label{CobarConstruction:CobarModelAcyclicity}
Let $\DOp$ be a $\C$-cofibrant connected cooperad.
Let $\POp$ be a $\C$-cofibrant operad.
If the morphism $\phi_{\theta}: B^c(\DOp)\rightarrow\POp$
associated to a twisting cochain $\theta: \DOp\rightarrow\POp$
defines a weak-equivalence of operads,
then the complex $(\POp\circ\DOp,\partial_{\theta})$ determined by $\theta$
is weakly equivalent to the unit $\Sigma_*$-object $\IOp$.
The weak-equivalence is defined by the augmentation $\epsilon: (\POp\circ\DOp,\partial_{\theta})\rightarrow \IOp$.
\end{thm}

The construction $(\POp\circ\DOp,\partial_\theta)$
is clearly functorial with respect to the operad morphisms $\phi: \POp\rightarrow\QOp$
commuting with given twisting cochain:
\begin{equation*}
\xymatrix{ & \DOp\ar[dl]_{\theta}\ar[dr]^{\upsilon} & \\
\POp\ar[rr]_{\phi} && \QOp }.
\end{equation*}
In particular,
the natural morphism $\phi_{\theta}: B^c(\DOp)\rightarrow\POp$
associated to a twisting cochain $\theta: \DOp\rightarrow\POp$
induces a morphism of twisted $\Sigma_*$-objects
\begin{equation*}
(B^c(\DOp)\circ\DOp,\partial_{\iota})
\xrightarrow{\phi_\theta{}_*}(\POp\circ\DOp,\partial_{\theta}).
\end{equation*}
The idea is to prove that this morphism
is a weak-equivalence
and to deduce the conclusion of theorem~\ref{CobarConstruction:CobarModelAcyclicity}
from theorem~\ref{CobarConstruction:UniversalAcyclicity}.
For this purpose,
we use a spectral sequence argument.

First,
observe that every twisted object $(\POp\circ\DOp,\partial_\theta)$
is endowed with a natural filtration.
Indeed,
let $F_s\POp\circ\DOp$ be the submodule of $\POp\circ\DOp$
spanned by level tree tensors
whose factor $p\in\POp(r)$
verifies $r\geq s$.
From the definition of the twisting homomorphism $\partial_\theta$,
we see immediately that the assumption $\tilde{\DOp}(0) = \tilde{\DOp}(1) = 0$
implies
\begin{equation*}
\partial_{\theta}(F_s\POp\circ\DOp)\subset F_{s-1}\POp\circ\DOp.
\end{equation*}

Thus, we obtain:

\begin{obsv}\label{CobarConstruction:TwistedComplexSpectralSequence}
For any complex $(\POp\circ\DOp,\partial_{\theta})$,
we have a spectral sequence
\begin{equation*}
E^r(\POp\circ\DOp,\partial_{\theta})\Rightarrow H_*(\POp\circ\DOp,\partial_{\theta}),
\end{equation*}
such that $E^0 = \POp\circ\DOp$
and where $d^0: E^0\rightarrow E^0$
reduces to the natural differential of the composite $\Sigma_*$-object $\POp\circ\DOp$,
the differential $\delta: \POp\circ\DOp\rightarrow\POp\circ\DOp$
induced by the internal differential of~$\POp$ and $\DOp$.
\end{obsv}

This spectral sequence converges in a strong sense
if $\POp$ and $\DOp$ are connected
because, for a fixed arity $r$, we have $F_s\POp\circ\DOp(r) = \POp\circ\DOp(r)$ for $s\geq r$
and $F_s\POp\circ\DOp(r) = 0$ for $s\leq 0$.
Hence,
we obtain that the filtration is bounded.

\medskip
We have clearly:

\begin{obsv}\label{CobarConstruction:TwistedComplexFunctoriality}
Any operad morphism $\phi: \POp\rightarrow\QOp$
that fits a commutative diagram
\begin{equation*}
\xymatrix{ & \DOp\ar[dl]_{\theta}\ar[dr]^{\upsilon} & \\
\POp\ar[rr]_{\phi} && \QOp },
\end{equation*}
where $\theta$ and $\upsilon$ are given twisting cochains,
preserves filtrations and yields a morphism on the spectral sequence
of observation~\ref{CobarConstruction:TwistedComplexSpectralSequence}.
\end{obsv}

The analysis of~\cite[\S\S 1.3.5-1.3.10]{FresseKoszulDuality} (see also~\cite[\S 15.3]{FresseModuleBook})
shows that any weak-equivalence $\phi: M\xrightarrow{\sim} N$
between $\C$-cofibrant $\Sigma_*$-objects
induces a weak-equivalence
\begin{equation*}
\phi\circ P: M\circ P\xrightarrow{\sim} N\circ P,
\end{equation*}
for any $\C$-cofibrant $\Sigma_*$-object $P$
such that $P(1) = 0$.
The application of this result to the spectral sequence
of observation~\ref{CobarConstruction:TwistedComplexSpectralSequence}
gives:

\begin{lemm}\label{CobarConstruction:TwistedComplexEquivalence}
In observation~\ref{CobarConstruction:TwistedComplexFunctoriality},
suppose that the morphism $\phi: \POp\xrightarrow{\sim}\QOp$
is a weak-equivalence between $\C$-cofibrant operads
and the cooperad $\DOp$ is $\C$-cofibrant (and connected).
Then
the morphism of twisted $\Sigma_*$-objects
\begin{equation*}
(\POp\circ\DOp,\partial_{\theta})\xrightarrow{\phi_*}(\QOp\circ\DOp,\partial_{\upsilon})
\end{equation*}
induced by $\phi$
forms a weak-equivalence.
\end{lemm}

\begin{proof}
The morphism induces a weak-equivalence
at the $E^0$-stage
of the spectral sequence:
\begin{equation*}
\underbrace{E^0(\POp\circ\DOp,\partial_{\theta})}_{= \POp\circ\DOp}
\xrightarrow{\phi_*}\underbrace{E^0(\QOp\circ\DOp,\partial_{\upsilon})}_{= \QOp\circ\DOp}.
\end{equation*}
The conclusion follows immediately (recall simply that the spectral sequence is defined by a locally bounded filtration
when the cooperad $\DOp$ is connected).
\end{proof}

The result of theorem~\ref{CobarConstruction:CobarModelAcyclicity}
is an immediate corollary of this lemma and theorem~\ref{CobarConstruction:UniversalAcyclicity}.
Hence we are done with the proof of proposition~\ref{CobarConstruction:CobarModelAcyclicity}.\qed

\subsubsection*{Remark}
In this paper,
we only use the implication of theorem~\ref{CobarConstruction:CobarModelAcyclicity}
but the converse statement is also true in the non-negatively graded setting:
under the assumption of theorem~\ref{CobarConstruction:CobarModelAcyclicity},
the morphism $\phi_\theta: B^c(\DOp)\rightarrow\POp$ associated to a twisting cochain $\theta: \DOp\rightarrow\POp$
forms a weak-equivalence
if and only if the twisted $\Sigma_*$-object $(\POp\circ\DOp,\partial_{\theta})$
is acyclic (apply the spectral sequence argument of~\cite[\S 2]{FresseKoszulDuality}).

\subsubsection{Applications of the bar duality of operads}
Dual to the cobar construction,
we have an operadic bar construction which gives a cooperad $B(\POp)$
such that
\begin{equation*}
\Mor_{\Op}(B^c(\DOp,\POp)\simeq\Tw_{\Op}(\DOp,\POp)\simeq\Hom_{\Op^c}(\DOp,B(\POp))
\end{equation*}
for any operad $\POp$ (see~\cite[\S 2.3]{GetzlerJones}).
If the operad $\POp$ is a $\C$-cofibrant augmented operad such that $\tilde{\POp}(0) = \tilde{\POp}(1) = 0$,
then the adjoint of the identity morphism $\id: B(\POp)\rightarrow B(\POp)$
defines a natural weak-equivalence $\epsilon_{\iota}: B^c(B(\POp))\xrightarrow{\sim}\POp$
(adapt the argument of~\cite[Theorem 3.2.16]{GinzburgKapranov}, see also~\cite[\S 4.8]{FresseKoszulDuality}).

Therefore
any connected $\C$-cofibrant operad $\POp$ has a model of the form $B^c(\DOp)$,
where $\DOp$ is a connected $\C$-cofibrant cooperad.
If $\POp$ is $\Sigma_*$-cofibrant,
then an easy inspection of constructions shows that the cooperad $\DOp = B(\POp)$
is $\Sigma_*$-cofibrant as well.
In general,
we can replace $\POp$ by an equivalent $\Sigma_*$-cofibrant operad $\QOp\xrightarrow{\sim}\POp$
(according to~\cite{BergerFresse}, we can take the internal tensor product of~$\POp$ with the Barratt-Eccles operad for~$\QOp$)
to form a $\Sigma_*$-cofibrant cooperad $\DOp = B(\QOp)$
such that $\phi: B^c(\DOp)\xrightarrow{\sim}\POp$.

The Koszul duality of operads replaces the cooperad $\DOp$
of the bar duality by a smaller model.
For instance:
\begin{itemize}
\item
for the operad of commutative algebras $\COp$,
we have a weak-equivalence $\phi: B^c(\Lambda^{-1}\LOp^{\vee})\xrightarrow{\sim}\COp$,
where $\Lambda^{-1}\LOp^{\vee}$ is an operadic desuspension of the cooperad of Lie coalgebras,
\item
for the operad of Lie algebras $\LOp$,
we have dually $\phi: B^c(\Lambda^{-1}\COp^{\vee})\xrightarrow{\sim}\LOp$,
where $\Lambda^{-1}\COp^{\vee}$ is the desuspension of the cooperad of cocommutative coalgebras,
\item
for the operad of associative algebras $\AOp$,
we have a weak-equivalence $\phi: B^c(\Lambda^{-1}\AOp^{\vee})\xrightarrow{\sim}\AOp$,
where $\Lambda^{-1}\AOp^{\vee}$ is the desuspension of the cooperad of coassociative coalgebras
\end{itemize}
(see~\cite{GinzburgKapranov} and~\cite{FresseKoszulDuality}
for the generalization of the results of~\cite{GinzburgKapranov} in positive characteristic).
The associative cooperad $\DOp = \Lambda^{-1}\AOp^{\vee}$ is $\Sigma_*$-cofibrant,
but we do not know how to produce a $\Sigma_*$-cofibrant cooperad
from $\Lambda^{-1}\COp^{\vee}$ and $\Lambda^{-1}\LOp^{\vee}$
(when the ground ring does not contain $\QQ$).

\section{The cobar construction\\ and\\ cofibrant models of algebras over operads}\label{CoalgebraModels}
\renewcommand{\thesubsubsection}{\thesubsection.\arabic{subsubsection}}

The purpose of this section is to review the applications of the bar duality of operads
to the definition of explicit cofibrant models
in categories of algebras over an operad.

In~\S\ref{CooperadCoalgebras},
we study quasi-cofree coalgebras over a cooperad~$\DOp$.
We check that the structure of a quasi-cofree coalgebra
is equivalent to an algebra over the cobar construction of~$\DOp$.
The construction of explicit cofibrant models
in categories of algebras over an operad
arises from this equivalence
and is addressed in~\S\ref{QuasiFreeReplacements}.

These applications of the bar duality of operads have been introduced in~\cite[\S 2, \S 4]{GetzlerJones}
for operads in non-negatively graded modules over a field of characteristic zero.
Though basic definitions have an obvious generalization in the context of unbounded dg-modules over a ring,
we have to review the theory carefully
in order to check that homotopical applications of the constructions hold properly in our setting.

In this section,
we use that a $\POp$-algebra structure on a dg-module $A$ is determined by a morphism $\phi: \POp\rightarrow\End_A$,
where $\End_A$ is the endomorphism operad of~$A$ (see~\S\ref{OperadAlgebras:EndomorphismOperads}),
and we adopt the convention to represent a $\POp$-algebra by a pair $(A,\phi)$
whenever $A$ does not come with a natural $\POp$-algebra structure.

\subsection{Coalgebras over cooperads}\label{CooperadCoalgebras}
In~\S\ref{OperadAlgebras:OperadMonad},
we recall that an operad $\POp$ determines a monad $S(\POp): \C\rightarrow\C$
on the category of dg-modules.
The category of $\POp$-algebras
is defined as the category of algebras over this monad.
This definition can be dualized for cooperads:
the augmentation $\epsilon: \DOp\rightarrow\IOp$
and the coproduct $\nu: \DOp\rightarrow\DOp\circ\DOp$ of a cooperad $\DOp$
induce an augmentation and a coproduct
\begin{equation*}
S(\DOp)\xrightarrow{S(\epsilon)} S(\IOp)\simeq\Id
\quad\text{and}
\quad S(\DOp)\xrightarrow{S(\nu)} S(\DOp\circ\DOp)\simeq S(\DOp)\circ S(\DOp),
\end{equation*}
which satisfy unit and associativity relations,
so that the functor $S(\DOp): \C\rightarrow\C$ associated to~$\DOp$
inherits a comonad structure.
Define a coalgebra over~$\DOp$
as a coalgebra over this comonad $S(\DOp)$.

Note that this notion of a $\DOp$-coalgebra does not agree with standard definitions for usual cooperads.
Firstly,
as the functor $S(\DOp)$ is given by a direct sum,
we consider $\DOp$-coalgebras which are in some sense connected.
Secondly,
as the functor $S(\DOp)$ is given by coinvariant tensors
and not invariant tensors,
we consider $\DOp$-coalgebras equipped with operations dual to divided powers.
For these matters,
we refer to~\cite{FresseSimplicialHomotopy}.

Recall that $\DOp$ is supposed to verify $\DOp(0) = 0$, $\DOp(1) = \kk$
and admits a splitting $\DOp = \IOp\oplus\tilde{\DOp}$.
In our constructions,
we identify the dg-module $C$ with the summand $C = S(\IOp,C)$
of~$S(\DOp,C)$
and the augmentation $\epsilon: S(\DOp,C)\rightarrow C$
with a projection onto this summand.

The object $\DOp(C) = S(\DOp,C)$ also represents the cofree $\DOp$-coalgebra
associated to a dg-module $C\in\C$ (again, when we say cofree $\DOp$-coalgebra,
we refer to the category of $\DOp$-coalgebras used in the paper and our definition may differ from other usual notions of cofree coalgebras).
The quasi-cofree coalgebras, of which study forms the object of this subsection,
are $\DOp$-coalgebras $\Gamma$
such that $\Gamma = (\DOp(C),\partial)$.

Our first objective is to introduce an appropriate representation
for the structure of a coalgebra.
Then we study the structure of quasi-cofree $\DOp$-coalgebras and the definition of morphisms between quasi-cofree $\DOp$-coalgebras.

\subsubsection{The graphical representation of the structure of coalgebras over cooperads}\label{CooperadCoalgebras:TreeRepresentation}
By definition,
a coalgebra over~$\DOp$ consists of a dg-module $C\in\C$
together with a coproduct $\rho: C\rightarrow S(\DOp,C)$
which satisfies the standard counit and coassociativity relations
with respect to the augmentation and coproduct of the comonad $S(\DOp)$.
In view of the graphical representation of an element of~$S(\DOp,C)$,
the coproduct of an element $c\in C$
has an expansion of the form:
\begin{equation*}
\rho(c) = \sum_{\rho(c)}
\left\{\vcenter{\xymatrix@H=6pt@W=3pt@M=2pt@!R=1pt@!C=1pt{ c_*\ar[dr]\ar@{.}[r] & \cdots\ar@{.}[r]\ar@{}[d]|{\displaystyle{\cdots}} & c_*\ar[dl] \\
& *+<8pt>[F]{\gamma_*}\ar[d] & \\ & 0 & }}\right\}.
\end{equation*}

The coproduct of a cofree coalgebra $\DOp(C) = S(\DOp,C)$
is the natural morphism
\begin{equation*}
S(\DOp,C)\xrightarrow{S(\nu)} S(\DOp\circ\DOp,C)\simeq S(\DOp,S(\DOp,C))
\end{equation*}
induced by the coproduct of~$\DOp$.
In our graphical representation,
the coproduct of~$\DOp(C)$
is given by the picture
\begin{equation*}
\rho\left\{\vcenter{\xymatrix@H=6pt@W=3pt@M=2pt@!R=1pt@!C=1pt{ c_1\ar[dr] & \cdots\ar@{}[d]|{\displaystyle{\cdots}} & c_n\ar[dl] \\
& *+<8pt>[F]{\gamma}\ar[d] & \\ & 0 & }}\right\}
= \sum_{\nu(\gamma)}
\left\{\vcenter{\xymatrix@H=6pt@W=3pt@M=2pt@!R=1pt@!C=1pt{ \ar@{.}[r] &
c_*\ar[dr]\ar@{.}[r] & \cdots\ar@{.}[r]\ar@{}[d]|{\displaystyle{\cdots}} & c_*\ar[dl]\ar@{.}[r] &
\cdots\ar@{.}[r] & c_*\ar[dr]\ar@{.}[r] & \cdots\ar@{.}[r]\ar@{}[d]|{\displaystyle{\cdots}} & c_*\ar[dl]\ar@{.}[r] & \\
*+<2pt>{1}\ar@{.}[rr] && *+<8pt>[F]{\gamma_*}\ar[drr]\ar@{.}[rr] && \cdots\ar@{.}[rr] && *+<8pt>[F]{\gamma_*}\ar[dll]\ar@{.}[rr] && \\
*+<2pt>{0}\ar@{.}[rrrr] &&&& *+<8pt>[F]{\gamma_*}\ar[d]\ar@{.}[rrrr] &&&& \\
\ar@{.}[rrrr] &&&& 0\ar@{.}[rrrr] &&&& }}\right\},
\end{equation*}
where we form the image of~$\gamma\in\DOp$
under the coproduct of~$\DOp$.
The input labels $c_1,\dots,c_n\in C$
are permuted according to the sharing of indices in the coproduct of~$\gamma$.

\subsubsection{Quasi-cofree coalgebras}\label{CooperadCoalgebras:QuasiCofreeDefinition}
For a coalgebra $C$ over a cooperad $\DOp$,
the commutation of the coproduct $\rho: C\rightarrow S(\DOp,C)$
with differentials amounts to an identity
\begin{equation}
\rho(\delta(c)) = \sum_{\rho(c)}
\left\{\vcenter{\xymatrix@H=6pt@W=3pt@M=2pt@!R=1pt@!C=1pt{ c_*\ar[dr]\ar@{.}[r] & \cdots\ar@{.}[r]\ar@{}[d]|{\displaystyle{\cdots}} & c_*\ar[dl] \\
\ar@{.}[r] & *+<8pt>[F]{\delta(\gamma_*)}\ar@{.}[r]\ar[d] & \\ \ar@{.}[r] & 0\ar@{.}[r] & }}\right\}
+  \sum_{\rho(c)}
\pm\left\{\vcenter{\xymatrix@H=6pt@W=3pt@M=2pt@!R=1pt@!C=1pt{ c_*\ar[drr]\ar@{.}[r] & \cdots & \delta(c_*)\ar@{.}[r]\ar@{.}[l]\ar[d] &
\cdots & c_*\ar[dll]\ar@{.}[l] \\
\ar@{.}[rr] && *+<8pt>[F]{\gamma_*}\ar@{.}[rr]\ar[d] && \\
\ar@{.}[rr] && 0\ar@{.}[rr] && }}\right\},
\end{equation}
for every $c\in C$.
In the second sum,
the differential $\delta$ is applied to the input label $c_i$,
where $i = 1,\dots,n$.
The signs $\pm$ are determined by the commutation of~$\delta$
with the factors of the tree tensor product.

Say that a homomorphism $\theta: C\rightarrow C$
is a $\DOp$-coderivation
if we have the identity
\begin{equation}\renewcommand{\theequation}{**}
\rho(\theta(c)) = \sum_{\rho(c)}
\pm\left\{\vcenter{\xymatrix@H=6pt@W=3pt@M=2pt@!R=1pt@!C=1pt{ c_*\ar[drr]\ar@{.}[r] & \cdots & \theta(c_*)\ar@{.}[r]\ar@{.}[l]\ar[d] &
\cdots & c_*\ar[dll]\ar@{.}[l] \\
\ar@{.}[rr] && *+<8pt>[F]{\gamma_*}\ar@{.}[rr]\ar[d] && \\ \ar@{.}[rr] && 0\ar@{.}[rr] && }}\right\},
\end{equation}
for every $c\in C$.
The differential $\delta+\partial: C\rightarrow C$
defined by the addition of a twisting homomorphism $\partial\in\Hom_{\C}(C,C)$
to the internal differential $\delta: C\rightarrow C$ of a $\DOp$-coalgebra $C$
satisfies the coderivation relation (*)
if and only if $\partial$ satisfies the coderivation relation (**).
Thus,
a twisted dg-module $(C,\partial)$ associated to a $\DOp$-coalgebra
inherits a $\DOp$-coalgebra structure
if and only if $\partial$ forms a $\DOp$-coderivation.

A quasi-cofree $\DOp$-coalgebra is a twisted coalgebra $\Gamma = (\DOp(C),\partial)$
formed from a cofree coalgebra $\DOp(C)$.
We use the following propositions to define quasi-cofree $\DOp$-coalgebras:

\begin{prop}[{See~\cite[Proposition 2.14]{GetzlerJones}}]\label{CooperadCoalgebras:CofreeCoderivations}
For a cofree coalgebra $\DOp(C)$,
we have a bijective correspondence between $\DOp$-coderivations $\partial: \DOp(C)\rightarrow\DOp(C)$
and homomorphisms $\alpha: \DOp(C)\rightarrow C$.
The homomorphism $\alpha$ associated to a coderivation $\partial$
is given by the composite of~$\partial: \DOp(C)\rightarrow\DOp(C)$
with the canonical projection $\epsilon: \DOp(C)\rightarrow C$.

Conversely,
the coderivation associated to $\alpha$,
for which we adopt the notation $\partial = \partial_{\alpha}$,
is determined by the formula
\begin{multline*}
\partial_{\alpha}
\left\{\vcenter{\xymatrix@H=6pt@W=3pt@M=2pt@!R=1pt@!C=1pt{ c_1\ar[dr]\ar@{.}[r] & \cdots\ar@{.}[r]\ar@{}[d]|{\displaystyle{\cdots}} & c_n\ar[dl] \\
\ar@{.}[r] & *+<8pt>[F]{\gamma}\ar[d]\ar@{.}[r] & \\
\ar@{.}[r] & 0\ar@{.}[r] & }}\right\}
= \sum_{i}
\pm\left\{\vcenter{\xymatrix@H=6pt@W=3pt@M=2pt@!R=1pt@!C=1pt{ c_1\ar[drr]\ar@{.}[r] & \cdots\ar@{.}[r] &
\alpha(c_i)\ar[d]\ar@{.}[r] & \cdots\ar@{.}[r] & c_n\ar[dll] \\
\ar@{.}[rr] && *+<8pt>[F]{\gamma}\ar@{.}[rr]\ar[d] && \\
\ar@{.}[rr] && 0\ar@{.}[rr] && }}\right\}
\\
+ \sum_{\substack{\tau\in\Theta_2(n)\\ \rho_{\tau}(\gamma)}}
\pm\left\{\vcenter{\xymatrix@H=6pt@W=4pt@M=2pt@R=8pt@C=4pt{ &&
\save [].[rrd]!C *+<4pt>[F-,]\frm{}*+<6pt>\frm{\{}*+<0pt>\frm{\}} \restore\ar@{}[]!L-<8pt,0pt>;[d]!L-<8pt,0pt>_{\displaystyle\alpha}
c_*\ar@{.}[r]\ar[dr] & \cdots\ar@{.}[r] & c_*\ar[dl] && \\
c_*\ar@{.}[r]\ar@/_6pt/[drrr] & \cdots && *+<6pt>[F]{\gamma_*}\ar@{.}[r]\ar@{.}[l]\ar[d] && \cdots\ar@{.}[r] & c_*\ar@/^6pt/[dlll] \\
\ar@{.}[rrr] &&& *+<6pt>[F]{\gamma_*}\ar[d]\ar@{.}[rrr] &&& \\
\ar@{.}[rrr] &&& 0\ar@{.}[rrr] &&& }}\right\},
\end{multline*}
for every element of~$\DOp(C)$,
where we use the convention of~\S\ref{CobarConstruction:TreeComonad}
for the representation of the coproduct of~$\gamma\in\DOp$.
The input labels $c_*\in C$ are permuted according to the sharing of inputs in the coproduct of~$\gamma$.
\qed
\end{prop}

Usually,
we assume that the homomorphism $\alpha: \DOp(C)\rightarrow C$
vanishes on $C\subset\DOp(C)$.
In this case,
we obtain:

\begin{prop}\label{CooperadCoalgebras:TwistingCofreeCoderivations}
Let $\alpha: \DOp(C)\rightarrow C$
be a homomorphism of degree $-1$
such that $\alpha|_C = 0$.

A $\DOp$-coderivation of degree $-1$
\begin{equation*}
\partial_{\alpha}: \DOp(C)\rightarrow\DOp(C)
\end{equation*}
satisfies the equation of twisting homomorphisms $\delta(\partial_{\alpha}) + \partial_{\alpha}^2 = 0$,
so that the pair $(\DOp(C),\partial_{\alpha})$ defines a quasi-cofree coalgebra,
if and only if the homomorphism $\alpha: \DOp(C)\rightarrow C$
satisfies the relation
\begin{equation*}
\delta(\alpha)
\left\{\vcenter{\xymatrix@H=6pt@W=3pt@M=2pt@!R=1pt@!C=1pt{ c_1\ar[dr]\ar@{.}[r] & \cdots\ar@{.}[r]\ar@{}[d]|{\displaystyle{\cdots}} & c_n\ar[dl] \\
\ar@{.}[r] & *+<8pt>[F]{\gamma}\ar[d]\ar@{.}[r] & \\
\ar@{.}[r] & 0\ar@{.}[r] & }}\right\}
+ \sum_{\substack{\tau\in\Theta_2(n)\\ \rho_{\tau}(\gamma)}}
\pm\alpha\left\{\vcenter{\xymatrix@H=6pt@W=4pt@M=2pt@R=8pt@C=4pt{ &&
\save [].[rrd]!C *+<4pt>[F-,]\frm{}*+<6pt>\frm{\{}*+<0pt>\frm{\}} \restore\ar@{}[]!L-<8pt,0pt>;[d]!L-<8pt,0pt>_{\displaystyle\alpha}
c_*\ar[dr]\ar@{.}[r] & \cdots\ar@{.}[r] & c_*\ar[dl] && \\
c_*\ar@{.}[r]\ar@/_6pt/[drrr] & \cdots && *+<6pt>[F]{\gamma_*}\ar@{.}[r]\ar@{.}[l]\ar[d] && \cdots\ar@{.}[r] & c_*\ar@/^6pt/[dlll] \\
\ar@{.}[rrr] &&& *+<6pt>[F]{\gamma_*}\ar[d]\ar@{.}[rrr] &&& \\
\ar@{.}[rrr] &&& 0\ar@{.}[rrr] &&& }}\right\} = 0
\end{equation*}
for every element of~$\DOp(C)$.
\end{prop}

\begin{proof}Exercise.\end{proof}

Then we have the following result:

\begin{prop}[{See~\cite[Proposition 2.15]{GetzlerJones}}]\label{CooperadCoalgebras:CobarOperadAlgebras}
A morphism of dg-operads $\phi: B^c(\DOp)\rightarrow\End_A$
is equivalent
to a map $\alpha: \DOp(A)\rightarrow A$
which satisfies the equation of paragraph~\ref{CooperadCoalgebras:TwistingCofreeCoderivations}
and such that the restriction $\alpha|_A$ vanishes.
\end{prop}

\begin{proof}
In~\S\ref{CobarConstruction:TwistingCochains},
we recall that an operad morphism $\phi_{\theta}: B^c(\DOp)\rightarrow\End_A$
is uniquely determined by a twisting cochain $\theta: \tilde{\DOp}\rightarrow\End_A$.
By adjunction,
the collection of $\Sigma_*$-equivariant homomorphisms of degree $-1$
\begin{equation*}
\tilde{\DOp}(r)\xrightarrow{\theta}\End_A(r) = \Hom_{\C}(A^{\otimes r},A)
\end{equation*}
underlying the twisting cochain
is equivalent to a dg-module homomorphism of degree $-1$
\begin{equation*}
\alpha: \DOp(A)\rightarrow A
\end{equation*}
such that $\alpha|_A = 0$.
Basically,
the homomorphism $\alpha$ is defined by the expression
\begin{equation*}
\alpha\left\{\vcenter{\xymatrix@H=6pt@W=3pt@M=2pt@!R=1pt@!C=1pt{ a_1\ar[dr]\ar@{.}[r] & \cdots\ar@{.}[r]\ar@{}[d]|{\displaystyle{\cdots}} & a_n\ar[dl] \\
\ar@{.}[r] & *+<8pt>[F]{\gamma}\ar@{.}[r]\ar[d] & \\
\ar@{.}[r] & 0\ar@{.}[r] & }}\right\}
= \theta(\gamma)(a_1,\dots,a_n),
\end{equation*}
for every element of~$\DOp(A)$,
where we apply the homomorphism $\theta(\gamma)\in\Hom_{\C}(A^{\otimes n},A)$
associated to $\gamma\in\DOp(n)$
to the input labels $a_1,\dots,a_n\in A$.

We see immediately that the equation of a twisting cochain,
represented in~\S\ref{CobarConstruction:TwistingCochains},
is equivalent to the equation of lemma~\ref{CooperadCoalgebras:TwistingCofreeCoderivations}.
The conclusion follows.
\end{proof}

Thus a quasi-cofree coalgebra $\Gamma_{\POp}(A) = (\DOp(A),\partial_{\alpha})$
is naturally associated to any $B^c(\DOp)$-algebra $A$.
We have further:

\begin{prop}\label{CooperadCoalgebras:CobarOperadAlgebraMorphisms}
Let $(A,\phi)$ and $(B,\psi)$
be $B^c(\DOp)$-algebras
with structure morphisms $\phi: B^c(\DOp)\rightarrow\End_A$
and $\psi: B^c(\DOp)\rightarrow\End_B$.
A morphism of dg-modules $f: A\rightarrow B$
defines a morphism of $B^c(\DOp)$-algebras $f: (A,\phi)\rightarrow(B,\psi)$
if and only if the coalgebra morphism $\DOp(f): \DOp(A)\rightarrow\DOp(B)$
induced by $f$
defines a morphism
\begin{equation*}
(\DOp(A),\partial_{\alpha})\xrightarrow{\DOp(f)}(\DOp(B),\partial_{\beta}),
\end{equation*}
between the quasi-cofree $\DOp$-coalgebras associated to $(A,\phi)$ and $(B,\psi)$.
\end{prop}

\begin{proof}Exercise.\end{proof}

Thus the category of $B^c(\DOp)$-algebras is equivalent to an explicit subcategory
of the category of $\DOp$-coalgebras.
The purpose of the next statement is to determine the set of all morphisms
between the quasi-cofree $\DOp$-coalgebras.
From these statements,
we see that the category of $B^c(\DOp)$-algebras
is not at all equivalent to the full subcategory generated by quasi-cofree $\DOp$-coalgebras,
but we aim to prove in the next subsection
that every morphism of quasi-cofree $\DOp$-coalgebras
defines a morphism in the homotopy category of $B^c(\DOp)$-algebras (whenever this notion makes sense).

First,
we have the following observation which is an extension of the universal
property of cofree coalgebras:

\begin{obsv}\label{CooperadCoalgebras:CofreeMorphisms}
The homomorphisms $\phi: \DOp(A)\rightarrow\DOp(B)$ of degree $0$
and commuting with coalgebra structures
are in bijection with homomorphisms of dg-modules $f: \DOp(A)\rightarrow B$.
The homomorphism $f$ associated to $\phi$
is given by the composite of $\phi: \DOp(A)\rightarrow\DOp(B)$
with the universal morphism $\epsilon: \DOp(B)\rightarrow B$.

Conversely,
the homomorphism associated to $f$,
for which we adopt the notation $\phi = \phi_{f}$,
is determined by the formula
\begin{equation*}\phi_f\left\{\vcenter{\xymatrix@M=3pt@H=4pt@W=6pt@R=8pt@C=6pt{ a_1\ar[dr]\ar@{.}[r] & \cdots\ar@{.}[r] & a_n\ar[dl] \\
\ar@{.}[r] & *+<4pt>[F]{\gamma}\ar@{.}[r]\ar[d] & \\
\ar@{.}[r] & 0\ar@{.}[r] & }}\right\}
= \quad\sum_{\nu(\gamma)}
\left\{\vcenter{\xymatrix@M=3pt@H=4pt@W=6pt@R=8pt@C=6pt{ &
\save [].[rrdd]!C *+<6pt>[F-,]\frm{}*+<6pt>\frm{\{}*+<0pt>\frm{\}} \restore \ar@{}[]!DL-<8pt,0pt>;[dd]!UL-<8pt,0pt>_{\displaystyle f} a_*\ar[dr]\ar@{.}[r] &
\cdots\ar@{.}[r] & a_*\ar[dl] &
&
\save [].[rrdd]!C *+<6pt>[F-,]\frm{}*+<6pt>\frm{\{}*+<0pt>\frm{\}} \restore \ar@{}[]!DL-<8pt,0pt>;[dd]!UL-<8pt,0pt>_{\displaystyle f} a_*\ar[dr]\ar@{.}[r] &
\cdots\ar@{.}[r] & a_*\ar[dl] & \\
&
& *+<4pt>[F]{\gamma_*}\ar@{.}[r]\ar@{.}[l]\ar@/_6pt/[ddrr] &
& &
& *+<4pt>[F]{\gamma_*}\ar@{.}[r]\ar@{.}[l]\ar@/^6pt/[ddll] & & \\
\ar@{.}[r] && & \ar@{.}[ll] & \cdots\ar@{.}[r]\ar@{.}[l] & \ar@{.}[rr] & && \ar@{.}[l] \\
\ar@{.}[rrrr] &&&& *+<4pt>[F]{\gamma_*}\ar@{.}[rrrr]\ar[d] &&&& \\
\ar@{.}[rrrr] &&&& 0\ar@{.}[rrrr] &&&& }}\right\},
\end{equation*}
where we form the coproduct $\nu: \DOp\rightarrow\DOp\circ\DOp$ of the element~$\gamma\in\DOp$.
\end{obsv}

We have then:

\begin{prop}\label{CooperadCoalgebras:QuasiFreeMorphisms}
The homomorphism of cofree coalgebras $\phi_f: \DOp(A)\rightarrow\DOp(B)$
associated to a homomorphism $f: A\rightarrow B$
defines a morphism between quasi-cofree coalgebras
\begin{equation*}
(\DOp(A),\partial_{\alpha})\rightarrow(\DOp(B),\partial_{\beta})
\end{equation*}
if and only if we have the identity
\begin{multline*}
\delta(f)\left\{\vcenter{\xymatrix@M=4pt@H=4pt@W=6pt@R=8pt@C=8pt{ a_1\ar[dr]\ar@{.}[r] & \cdots\ar@{.}[r] & a_n\ar[dl] \\
\ar@{.}[r] & *+<4pt>[F]{\gamma}\ar[d]\ar@{.}[r] & \\
\ar@{.}[r] & 0\ar@{.}[r] & }}\right\}
- \sum_{\substack{\tau\in\Theta_2(n)\\ \rho_{\tau}(\gamma)}}
\pm f\left\{\vcenter{\xymatrix@H=6pt@W=4pt@M=2pt@R=8pt@C=4pt{ &&
\save [].[rrd]!C *+<6pt>[F-,]\frm{}*+<6pt>\frm{\{}*+<0pt>\frm{\}} \restore\ar@{}[]!L-<8pt,0pt>;[d]!L-<8pt,0pt>_{\displaystyle\alpha}
a_*\ar[dr]\ar@{.}[r] & \cdots\ar@{.}[r] & a_*\ar[dl] && \\
a_*\ar@{.}[r]\ar@/_6pt/[drrr] & \cdots\ar@{.}[r] && *+<6pt>[F]{\gamma_*}\ar[d]\ar@{.}[r]\ar@{.}[l] && \cdots\ar@{.}[r]\ar@{.}[l]  & a_*\ar@/^6pt/[dlll] \\
\ar@{.}[rrr] &&& *+<6pt>[F]{\gamma_*}\ar[d]\ar@{.}[rrr] &&& \\
\ar@{.}[rrr] &&& 0\ar@{.}[rrr] &&& }}\right\} \\
+ \sum_{\nu(\gamma)}\beta\left\{\vcenter{\xymatrix@M=4pt@H=4pt@W=6pt@R=8pt@C=8pt{ &
\save [].[rrdd]!C *+<6pt>[F-,]\frm{}*+<6pt>\frm{\{}*+<0pt>\frm{\}}\restore \ar@{}[]!DL-<8pt,0pt>;[dd]!UL-<8pt,0pt>_{\displaystyle f} a_*\ar[dr]\ar@{.}[r] &
\cdots\ar@{.}[r] & a_*\ar[dl] &
&
\save [].[rrdd]!C *+<6pt>[F-,]\frm{}*+<6pt>\frm{\{}*+<0pt>\frm{\}} \restore \ar@{}[]!DL-<8pt,0pt>;[dd]!UL-<8pt,0pt>_{\displaystyle f} a_*\ar[dr]\ar@{.}[r] &
\cdots\ar@{.}[r] & a_*\ar[dl] & \\
&
& *+<4pt>[F]{\gamma_*}\ar@{.}[r]\ar@{.}[l]\ar@/_6pt/[ddrr] &
& &
& *+<4pt>[F]{\gamma_*}\ar@{.}[r]\ar@{.}[l]\ar@/^6pt/[ddll] & & \\
\ar@{.}[r] & \ar@{.}[rr] & & \ar@{.}[r] & \cdots\ar@{.}[r] && & \ar@{.}[ll] & \ar@{.}[l] \\
\ar@{.}[rrrr] &&&& *+<4pt>[F]{\gamma_*}\ar@{.}[rrrr]\ar[d] &&&& \\
\ar@{.}[rrrr] &&&& 0\ar@{.}[rrrr] &&&& }}\right\} = 0
\end{multline*}
for every element of~$\DOp(A)$.
\end{prop}

\begin{proof}Exercise.\end{proof}

\subsection{Bar duality and quasi-free replacements}\label{QuasiFreeReplacements}
Throughout this subsection,
we suppose given an operad $\POp$
together with a weak-equivalence $\phi_{\theta}: B^c(\DOp)\xrightarrow{\sim}\POp$
from the cobar construction of a (connected) cooperad to $\POp$.
Moreover,
we assume that $\POp$ is $\Sigma_*$-cofibrant,
as well as the cooperad $\DOp$.
Accordingly,
the category of $\POp$-algebras inherits a natural semi-model structure
from dg-modules.

Our aim is two-fold.

First we associate a quasi-free $\POp$-algebra $R_{\POp}(\Gamma)$
to any $\DOp$-coalgebra $\Gamma$.
We apply this construction to the quasi-cofree $\DOp$-coalgebras $\Gamma = (\DOp(A),\partial_{\alpha})$
associated to $\POp$-algebras $A$.
We prove that the composite construction $R_A = R_{\POp}(\DOp(A),\partial_{\alpha})$
defines a cofibrant replacement functor on the category of $\C$-cofibrant $\POp$-algebras.

Then we use the construction $R_A = R_{\POp}(\DOp(A),\partial_{\alpha})$
to prove that the morphisms of quasi-cofree $\DOp$-coalgebras
represent morphisms in the homotopy category of~$\POp$-algebras.

\subsubsection{Construction of quasi-free algebras}\label{QuasiFreeReplacements:QuasiFreeAlgebraConstruction}
Let $\Gamma$ be a $\DOp$-coalgebra.
Let $\omega$ be the composite
\begin{equation*}
\Gamma\xrightarrow{\rho} S(\DOp,\Gamma)\xrightarrow{S(\theta,\Gamma)} S(\POp,\Gamma) = \DOp(\Gamma),
\end{equation*}
where
$\rho$ refers to the coproduct of $\Gamma$
and
$\theta: \DOp\rightarrow\POp$
is the twisting cochain equivalent to the morphism $\phi_{\theta}: B^c(\DOp)\rightarrow\POp$.
In view of the representation of~\S\ref{CooperadCoalgebras:TreeRepresentation}
for the coproduct of an element $c\in\Gamma$,
the image of $c$ under $\omega: \Gamma\rightarrow\POp(\Gamma)$
is given by:
\begin{equation*}
\omega(c) = \sum_{\rho(c)}
\left\{\vcenter{\xymatrix@H=6pt@W=3pt@M=2pt@!R=1pt@!C=1pt{ c_*\ar[dr]\ar@{.}[r] & \cdots\ar@{.}[r]\ar@{}[d]|{\displaystyle{\cdots}} & c_*\ar[dl] \\
\ar@{.}[r] & *+<8pt>[F]{\theta(\gamma_*)}\ar@{.}[r]\ar[d] & \\
\ar@{.}[r] & 0\ar@{.}[r] & }}\right\}.
\end{equation*}

The quasi-free $\POp$-algebra associated to $\Gamma$
is defined by the pair $R_{\POp}(\Gamma) = (\POp(\Gamma),\partial_{\gamma})$,
where $\partial_{\gamma}: \POp(\Gamma)\rightarrow\POp(\Gamma)$
is the $\POp$-algebra derivation associated to the homomorphism $\gamma: \Gamma\rightarrow\POp(\Gamma)$.
Note simply that:

\begin{claim}\label{QuasiFreeReplacements:QuasiFreeStructureVerification}
The homomorphism $\omega: \Gamma\rightarrow\POp(\Gamma)$
verifies the equation $\delta(\omega) + \partial_{\omega}\cdot\omega = 0$
of proposition~\ref{OperadAlgebras:TwistingDerivationConstruction}.
\end{claim}

Hence, the quasi-free $\POp$-algebra $R_{\POp}(\Gamma) = (\POp(\Gamma),\partial_{\omega})$
is well defined.

\begin{proof}
On one hand,
the image of an element $c\in\Gamma$ under the composite $\partial_{\omega}\cdot\omega$
has an expansion of the form
\begin{equation}
\partial_{\omega}\cdot\omega(c) = - \sum_{\substack{\rho(c)\\ *,\rho(c_*)}}
\pm\left\{\vcenter{\xymatrix@H=6pt@W=3pt@M=2pt@!R=1pt@!C=1pt{ c_*\ar@/_4pt/[ddrr]\ar@{.}[r]|{\displaystyle{\cdots}} &
c_*\ar[dr]\ar@{.}[r] & \cdots & c_*\ar[dl]\ar@{.}[l] &
c_*\ar@/^4pt/[ddll]\ar@{.}[l]|{\displaystyle{\cdots}} \\
&& *+<3mm>[F]{\theta(\gamma_*)}\ar[d]\save[]!C.[d]!C *+<4pt>[F-,]\frm{}*+<6pt>\frm{\{}*+<0pt>\frm{\}}
\restore\ar@{}[]!L-<6pt,0pt>;[d]!L-<6pt,0pt>_(0.4){\displaystyle{\lambda_*}} && \\
\ar@{.}[rr] && *+<3mm>[F]{\theta(\gamma_*)}\ar[d] && \ar@{.}[ll] \\
\ar@{.}[rr] && 0\ar@{.}[rr] && }}\right\},
\end{equation}
where we perform the coproduct (one by one) of the factors $c_*$
in the expansion of~$\rho(c)$
and $\lambda_*$ refers to an evaluation of the tree composition product of~$\POp$.
The additional minus sign
comes from a permutation of the symbols $\theta$.

This expression can be identified with the image of the iterated coproduct
\begin{equation*}
S(\DOp,\rho)\cdot\rho(c) =
\sum_{\substack{\rho(c)\\ \rho(c_*),\dots,\rho(c_*)}}
\pm\left\{\vcenter{\xymatrix@H=6pt@W=3pt@M=2pt@!R=1pt@!C=1pt{ \ar@{.}[r] &
c_*\ar[dr]\ar@{.}[r] & \cdots\ar@{.}[r]\ar@{}[d]|{\displaystyle{\cdots}} & c_*\ar[dl]\ar@{.}[r] &
\cdots\ar@{.}[r] & c_*\ar[dr]\ar@{.}[r] & \cdots\ar@{.}[r]\ar@{}[d]|{\displaystyle{\cdots}} & c_*\ar[dl]\ar@{.}[r] & \\
*+<2pt>{1}\ar@{.}[rr] && *+<8pt>[F]{\gamma_*}\ar[drr]\ar@{.}[rr] && \cdots\ar@{.}[rr] && *+<8pt>[F]{\gamma_*}\ar[dll]\ar@{.}[rr] && \\
*+<2pt>{0}\ar@{.}[rrrr] &&&& *+<8pt>[F]{\gamma_*}\ar[d]\ar@{.}[rrrr] &&&& \\
\ar@{.}[rrrr] &&&& 0\ar@{.}[rrrr] &&&& }}\right\}
\end{equation*}
under the homomorphism
\begin{equation*}
\sum_{e = 1,\dots,n_*}
\left\{\vcenter{\xymatrix@H=8pt@W=3pt@M=2pt@R=16pt@!C=1pt{ 1\ar@{.}[r] & *+<3mm>[F]{\epsilon}\ar[drr]\ar@{.}[r] & \cdots\ar@{.}[r] &
*+<3mm>[F]{\theta}\ar[d]|(0.4){e}\ar@{.}[r] & \cdots\ar@{.}[r] &  *+<3mm>[F]{\epsilon}\ar[dll]\ar@{.}[r] & \\
0\ar@{.}[rrr] &&& *+<3mm>[F]{\theta}\ar[d]\ar@{.}[rrr] &&& \\
\ar@{.}[rrr] &&& 0\ar@{.}[rrr] &&& }}\right\},
\end{equation*}
where $\epsilon: \DOp\rightarrow\IOp$ represents the augmentation of~$\DOp$.
The sum ranges over the positions of the second level.
The coassociativity relation $S(\DOp,\rho)\cdot\rho = S(\nu,\Gamma)\cdot\rho$
implies that (*) is equivalent to a summation over $\rho(c)$ and $\rho(\gamma_*)$,
where we take the coproduct of the factor~$\gamma_*$
in~$\DOp$.

Then the equation of twisting cochains implies that (*) agrees with:
\begin{equation}\renewcommand{\theequation}{**}
- \sum_{\rho(c)}
\left\{\vcenter{\xymatrix@H=6pt@W=3pt@M=2pt@!R=1pt@!C=1pt{ c_*\ar[dr]\ar@{.}[r] & \cdots\ar@{.}[r]\ar@{}[d]|{\displaystyle{\cdots}} & c_*\ar[dl] \\
\ar@{.}[r] & *+<8pt>[F]{\delta(\theta)(\gamma_*)}\ar[d]\ar@{.}[r] & \\
\ar@{.}[r] & 0\ar@{.}[r] & }}\right\}.
\end{equation}
The coderivation relation (*) of~\S\ref{CooperadCoalgebras:QuasiCofreeDefinition}
implies immediately that (**) represents the expansion of~$- \delta(\omega)(c) = - \delta(\omega(c)) - \omega(\delta(c))$.
Hence we are done.
\end{proof}

\subsubsection{The quasi-free replacement of a $\POp$-algebra}\label{QuasiFreeReplacements:ReplacementConstruction}
Let $A$ be a $\POp$-algebra.
Observe that $A$ forms a $B^c(\DOp)$-algebra
by restriction of structure
and hence has an associated quasi-free $\DOp$-coalgebra $\Gamma_{\POp}(A) = (\DOp(A),\partial_{\alpha})$.
Recall that the homomorphism $\alpha: \DOp(A)\rightarrow A$
which determines the twisting coderivation of~$\Gamma_{\POp}(A)$
satisfies $\alpha|_A = 0$.

We form the quasi-free $\POp$-algebra $R_A = R_{\POp}(\DOp(A),\partial_{\alpha})$
associated to $\Gamma_{\POp}(A)$.
We aim to prove that $R_A$
defines a natural cofibrant replacement of~$A$
when $A$ is $\C$-cofibrant.

By definition,
the underlying free $\POp$-algebra of~$R_A$ can be identified with the composite
\begin{equation*}
\POp(\DOp(A)) = S(\POp\circ\DOp,A)
\end{equation*}
when we forget all twisting homomorphisms.
This free $\POp$-algebra is equipped with a twisting derivation $\partial_{\alpha}$
induced by the twisting coderivation of~$\Gamma_{\POp}(A) = (\DOp(A),\partial_{\alpha})$,
and with the twisting homomorphism $\partial_{\omega}$
determined by the underlying coalgebra structure of~$\Gamma_{\POp}(A)$.
The total differential of~$R_A$ is the sum $\delta+\partial_{\alpha}+\partial_{\omega}$,
where $\delta$ refers to the natural differential of~$\POp(\DOp(A))$.

We have a natural morphism of $\POp$-algebras
\begin{equation*}
\epsilon: \POp(\DOp(A))\rightarrow A
\end{equation*}
induced by the augmentation $\epsilon: \DOp(A)\rightarrow A$
of the cofree $\DOp$-coalgebra $\DOp(A)$
and a natural morphism of dg-modules
\begin{equation*}
\eta: A\rightarrow\POp(\DOp(A))
\end{equation*}
defined by the composite of the coaugmentation $\eta: A\rightarrow\DOp(A)$
with the unit morphism $\eta: \DOp(A)\rightarrow\POp(\DOp(A))$
of the free $\POp$-algebra~$\POp(\DOp(A))$.
We have in fact $\POp\circ\DOp(1) = \IOp(1) = \kk$
and the coaugmentation $\eta: A\rightarrow\POp(\DOp(A))$
can be identified with the morphism
\begin{equation*}
A = S(\IOp,A)\rightarrow S(\POp\circ\DOp,A)
\end{equation*}
yielded by this relation.

Observe first:

\begin{prop}\label{QuasiFreeReplacements:StructureMorphisms}
The morphism $\epsilon: \POp(\DOp(A))\rightarrow A$
defines a morphism of $\POp$-algebras
\begin{equation*}
R_A = R_{\POp}(\DOp(A),\partial_{\alpha})\xrightarrow{\epsilon} A.
\end{equation*}

The morphism $\eta: A\rightarrow\POp(\DOp(A))$
defines a morphism of dg-modules
\begin{equation*}
A\xrightarrow{\eta} R_{\POp}(\DOp(A),\partial_{\alpha}) = R_A
\end{equation*}
such that $\epsilon\eta = \id$.
\end{prop}

\begin{proof}
We have $\delta(\epsilon) = \delta(\eta) = 0$
since $\epsilon$ and $\eta$ are defined by morphisms of dg-modules
when we forget all twisting homomorphisms.
As a consequence,
we only have to check
\begin{equation}
\epsilon\cdot(\partial_{\omega}+\partial_{\alpha}) = 0
\quad\text{and}\quad(\partial_{\omega}+\partial_{\alpha})\cdot\eta = 0
\end{equation}
in order to prove that $\epsilon$ and $\eta$
commute with total differentials.

By proposition~\ref{OperadAlgebras:QuasiFreeMorphismConstruction},
we are reduced to prove the identity $\epsilon\cdot(\partial_{\omega}+\partial_{\alpha})(c) = 0$
for a generating element
\begin{equation*}
c = \left\{\vcenter{\xymatrix@H=6pt@W=3pt@M=2pt@!R=1pt@!C=1pt{ a_1\ar[dr]\ar@{.}[r] & \cdots\ar@{.}[r]\ar@{}[d]|{\displaystyle{\cdots}} & a_n\ar[dl] \\
\ar@{.}[r] & *+<8pt>[F]{\gamma}\ar[d]\ar@{.}[r] & \\
\ar@{.}[r] & 0\ar@{.}[r] & }}\right\}\in\DOp(A).
\end{equation*}
Then the equation amounts to
\begin{equation}\renewcommand{\theequation}{**}
\epsilon(\omega(c))+\alpha(c) = 0
\end{equation}
by construction of~$\partial_{\omega}: \POp(\DOp(A))\rightarrow\POp(\DOp(A))$
and $\partial_{\alpha}: \DOp(A)\rightarrow\DOp(A)$.
But we have:
\begin{equation*}
\epsilon(\omega(c))
= \lambda_*\left\{\vcenter{\xymatrix@H=6pt@W=3pt@M=2pt@!R=1pt@!C=1pt{ a_1\ar[dr]\ar@{.}[r] & \cdots\ar@{.}[r]\ar@{}[d]|{\displaystyle{\cdots}} & a_n\ar[dl] \\
\ar@{.}[r] & *+<8pt>[F]{\theta(\gamma)}\ar[d]\ar@{.}[r] & \\
\ar@{.}[r] & 0\ar@{.}[r] & }}\right\} = \theta(\gamma)(a_1,\dots,a_n)\in A,
\end{equation*}
where we perform the evaluation of~$\theta(\gamma)\in B^c(\DOp)(n)$
on $a_1,\dots,a_n\in A$.
This operation is also the definition of~$\alpha(c)$.
Hence,
equation (**) is satisfied and this achieves the proof that $\epsilon$
defines a morphism of $\POp$-algebras:
\begin{equation*}
R_A = R_{\POp}(\DOp(A),\partial_{\alpha})\xrightarrow{\epsilon} A.
\end{equation*}

Equation (*)
is trivially satisfied for~$\eta$
because $\partial_{\omega}|_{\DOp(A)} = \omega$ and $\partial_{\alpha}$
vanishes on $A\subset\DOp(A)$.
Hence,
the morphism $\epsilon$
defines a morphism of dg-modules
\begin{equation*}
A\xrightarrow{\eta} R_{\POp}(\DOp(A),\partial_{\alpha}) = R_A.
\end{equation*}
The relation $\epsilon\eta = \Id$
is immediate from the definition of~$\epsilon$ and $\eta$.
\end{proof}

Our theorem reads:

\begin{thm}\label{QuasiFreeReplacements:Statement}
Let $\POp$ be any $\Sigma_*$-cofibrant operad.
Let $\DOp$ be any $\Sigma_*$-cofibrant (connected) cooperad together with a twisting cochain $\theta: \DOp\rightarrow\POp$
associated to a weak-equivalence $\phi_{\theta}: B^c(\DOp)\xrightarrow{\sim}\POp$.

If $A$ is a $\C$-cofibrant $\POp$-algebra,
then the augmentation $\epsilon: R_{\POp}(\DOp(A),\partial_{\alpha})\rightarrow A$ defines a weak-equivalence
and $R_A = R_{\POp}(\DOp(A),\partial_{\alpha})$
forms a cofibrant replacement of~$A$ in the category of $\POp$-algebras.
\end{thm}

This theorem gives a generalization of~\cite[Theorem 2.19]{GetzlerJones}.
The result of this reference is established for operads in non-negatively graded dg-modules
over a field of characteristic $0$.

The proof of this theorem is deferred to a series of lemmas.

\begin{lemm}\label{QuasiFreeReplacements:CofibrantStructure}
Under the assumptions of the theorem,
the quasi-cofree $\POp$-algebra $R_A = R_{\POp}(\DOp(A),\partial_{\alpha})$
forms a cofibrant $\POp$-algebra.
\end{lemm}

\begin{proof}
Equip the dg-module $\DOp(A)$
with the filtration
\begin{equation*}
0 = \DOp_{\leq 0}(A)\subset\cdots\subset\DOp_{\leq\lambda}(A)\subset\cdots\subset\colim_{\lambda\in\NN}\DOp_{\leq\lambda}(A) = A
\end{equation*}
such that
\begin{equation*}
\DOp_{\leq\lambda}(A) = \bigoplus_{r\leq\lambda} (\DOp(r)\otimes A^{\otimes r})_{\Sigma_r}.
\end{equation*}
The assumptions of the theorem imply readily
that each summand $(\DOp(r)\otimes A^{\otimes r})_{\Sigma_r}$
forms a cofibrant dg-module.
Hence,
each embedding
\begin{equation*}
\DOp_{\leq\lambda-1}(A)\hookrightarrow\DOp_{\leq\lambda}(A)
\end{equation*}
forms a cofibration of dg-modules.

Recall that $\alpha|_A = 0$ by construction (see proposition~\ref{CooperadCoalgebras:CobarOperadAlgebras}).
The representation of proposition~\ref{CooperadCoalgebras:CofreeCoderivations}
shows that the twisting homomorphism $\partial_{\alpha}$
satisfies
\begin{equation*}
\partial_{\alpha}(\DOp_{\lambda}(A))\subset\DOp_{\lambda-1}(A),
\end{equation*}
because the factors $\gamma_*\in\tilde{\DOp}(I_*)$ of the coproduct of an element $\gamma\in\tilde{\DOp}(r)$
must satisfy $1<|I_*|<r$
when $\tilde{\DOp}(0) = \tilde{\DOp}(1) = 0$.
For the twisting derivation $\partial_{\omega}$,
we have trivially
\begin{equation*}
\partial_{\omega}(\DOp_{\leq 1}(A)) = 0
\quad\text{and}
\quad\partial_{\omega}(\DOp_{\leq\lambda}(A))\subset\POp(A) = \POp(\DOp_{\leq 1}(A))
\quad\text{for $\lambda\geq 2$}.
\end{equation*}

These verifications show that the quasi-free $\POp$-algebra $R_A = R_{\POp}(\DOp(A),\partial_{\alpha})$
fulfils the requirements of proposition~\ref{OperadAlgebras:CofibrantAlgebras}
from which we conclude that $R_A$ forms a cofibrant $\POp$-algebra.
\end{proof}

\begin{lemm}[{compare with~\cite[Theorem 2.19]{GetzlerJones}}]\label{QuasiFreeReplacements:Equivalences}
The morphisms
\begin{equation*}
\xymatrix{ R_A = R_{\POp}(\DOp(A),\partial_{\alpha})\ar@<+2pt>[r]^(0.7){\epsilon}\ar@<-2pt>@{<-}[r]_(0.7){\eta} & A }
\end{equation*}
are weak-equivalences
as long as the operad $\POp$ and the cooperad $\DOp$ are $\Sigma_*$-cofibrant
and the morphism $\phi_{\theta}: B^c(\DOp)\rightarrow\POp$
is a weak-equivalence.
\end{lemm}

\begin{proof}
Recall that the underlying $\POp$-algebra
of~$R_A$ can be identified with the composite object:
\begin{equation*}
\POp(\DOp(A)) = S(\POp\circ\DOp,A)
\end{equation*}
when we forget all twisting derivations.
Consider the nested sequence
of dg-modules such that
\begin{equation*}
F_{s} S(\POp\circ\DOp,A) = \bigoplus_{r\geq s} (\POp\circ\DOp(r)\otimes A^{\otimes r})_{\Sigma_r}\subset S(\POp\circ\DOp,A),
\end{equation*}
for $s\in\NN$.

The objects $F_{s} S(\POp\circ\DOp,A)$ are trivially preserved by internal differentials.
The vanishing of the twisting cochain $\theta$ on~$\DOp(1)$
implies immediately,
from the expression of~\S\ref{CooperadCoalgebras:CofreeCoderivations},
that the twisting coderivation $\partial_{\alpha}: \DOp(A)\rightarrow\DOp(A)$
satisfies
\begin{equation*}
\partial_{\alpha}(\DOp(s)\otimes A^{\otimes s})_{\Sigma_s}\subset\bigoplus_{r\geq s-1}(\DOp(r)\otimes A^{\otimes r})_{\Sigma_r},
\end{equation*}
from which we deduce the relation
\begin{equation*}
\partial_{\alpha}(F_{s} S(\POp\circ\DOp,A))\subset F_{s-1} S(\POp\circ\DOp,A)
\end{equation*}
for the derivation $\partial_{\alpha}: \POp(\DOp(A))\rightarrow\POp(\DOp(A))$
induced by $\partial_{\alpha}: \DOp(A)\rightarrow\DOp(A)$.
For the quasi-cofree coalgebra $\Gamma = (\DOp(A),\partial_{\alpha})$,
the twisting derivation $\partial_{\omega}: \POp(\DOp(A))\rightarrow\POp(\DOp(A))$ of~$R_A = R_{\POp}(\DOp(A),\partial_{\alpha})$
can immediately be identified with the homomorphism $S(\partial_\theta,A)$
induced by the twisting homomorphism $\partial_{\theta}: \POp\circ\DOp\rightarrow\POp\circ\DOp$
of~\S\ref{CobarConstruction:OperadCoperadTwistedComplexes}.
Hence we have $\partial_{\omega}(F_{s} S(\POp\circ\DOp,A))\subset F_{s} S(\POp\circ\DOp,A)$.

Hence,
the filtration
\begin{equation*}
0 = F_{0} S(\POp\circ\DOp,A)\subset\cdots\subset F_{0} S(\POp\circ\DOp,A)\subset\cdots\subset S(\POp\circ\DOp,A),
\end{equation*}
defines a filtration of the twisted dg-module $R_A$
and determines a right-hand half-plane homological spectral sequence $E^r\Rightarrow H_*(R_A)$
such that
\begin{equation*}
(E^0,d^0) = S((\POp\circ\DOp,\partial_{\theta}),A),
\end{equation*}
the dg-module associated to $A$ by the twisted $\Sigma_*$-object $(\POp\circ\DOp,\partial_{\theta})$.

Recall that $\eta: A\rightarrow R_A$ is yielded by the canonical morphism $\eta: \IOp\rightarrow\POp\circ\DOp$
given by the identity $\POp\circ\DOp(1) = \IOp(1) = \kk$.

By theorem~\ref{CobarConstruction:CobarModelAcyclicity},
the canonical morphism $\eta: \IOp\rightarrow\POp\circ\DOp$
defines a weak-equivalence $\eta: \IOp\xrightarrow{\sim}(\POp\circ\DOp,\partial_{\theta})$.
Since $\POp$ and $\DOp$ are $\Sigma_*$-cofibrant,
the composite $\Sigma_*$-object $\POp\circ\DOp$
is $\Sigma_*$-cofibrant too.
Hence,
the morphism $\eta: \IOp\rightarrow\POp\circ\DOp$
induces a weak-equivalence
\begin{equation*}
A = S(\IOp,A)\xrightarrow{S(\eta,A)} S((\POp\circ\DOp,\partial_{\theta}),A)
\end{equation*}
for every $A\in\C$.
Therefore
our spectral sequence satisfies
\begin{equation*}
E^1_{s *} = \begin{cases} A, & \text{if $s = 1$}, \\ 0, & \text{otherwise}, \end{cases}
\end{equation*}
and degenerates at the $E^1$-stage,
from which we also deduce that the spectral sequence converges toward $H_*(R_A)$.

Our arguments imply moreover that $\eta: A\rightarrow R_A$
induces an isomorphism $\eta: H_*(A)\xrightarrow{\simeq} E^1_{1 *} = H_*(R_A)$,
from which we conclude that $\eta: A\rightarrow R_A$
defines a weak-equivalence.
Since $\epsilon\eta = \id$, the morphism $\epsilon: R_A\rightarrow A$
defines a weak-equivalence as well.
Hence we are done.
\end{proof}

This lemma achieves the proof of theorem~\ref{QuasiFreeReplacements:Statement}.\qed

\medskip
By construction,
the quasi-free $\POp$-algebra $R_A = R_{\POp}(\DOp(A),\partial_{\alpha})$
associated to a $\POp$-algebra $A$
is functorial with respect to all morphisms of $\DOp$-coalgebras
\begin{equation*}
(\DOp(A),\partial_{\alpha})\xrightarrow{\phi_f}(\DOp(B),\partial_{\beta}),
\end{equation*}
and not only with respect to morphisms of the category of $\POp$-algebras.
Hence,
theorem~\ref{QuasiFreeReplacements:Statement} gives as a corollary:

\begin{prop}\label{QuasiFreeReplacements:HomotopyMorphisms}
Suppose we have a morphism of quasi-cofree coalgebras
\begin{equation*}
\phi_f: (\DOp(A),\partial_{\alpha})\rightarrow(\DOp(B),\partial_{\beta}),
\end{equation*}
where $(\DOp(A),\partial_{\alpha})$ and $(\DOp(B),\partial_{\beta})$
are quasi-cofree coalgebras associated to $\POp$-algebras $(A,\phi)$ and $(B,\psi)$.

The morphism of $\POp$-algebras induced by $\phi_f$
fits a diagram
\begin{equation*}
\xymatrix{ R_{\POp}(\DOp(A),\partial_{\alpha})\ar[d]^{\sim}\ar[r]^{\phi_f} & R_{\POp}(\DOp(B),\partial_{\beta})\ar[d]^{\sim} \\
(A,\phi) & (B,\psi) },
\end{equation*}
in which the vertical morphisms
are weak-equivalences of $\POp$-algebras.
Accordingly,
this morphism yields a morphism from $(A,\phi)$ to $(B,\psi)$
in the homotopy category of $\POp$-algebras.\qed
\end{prop}

Moreover:

\begin{prop}\label{QuasiFreeReplacements:HomotopyEquivalences}
Let $f: \DOp(A)\rightarrow B$
be the homomorphism which determines the morphism of quasi-cofree coalgebras $\phi_f$
in proposition~\ref{QuasiFreeReplacements:HomotopyMorphisms}.

If the restriction $f|_A$
defines a weak-equivalence of dg-modules $f|_A: A\xrightarrow{\sim} B$,
then $\phi_f$ induces a weak-equivalence
of $\POp$-algebras
\begin{equation*}
\phi_f: R_{\POp}(\DOp(A),\partial_{\alpha})\xrightarrow{\sim} R_{\POp}(\DOp(B),\partial_{\beta}).
\end{equation*}
\end{prop}

\begin{proof}
The diagram in the category dg-modules
\begin{equation*}
\xymatrix{ R_{\POp}(\DOp(A),\partial_{\alpha})\ar[r]^{\phi_f} & R_{\POp}(\DOp(B),\partial_{\beta}) \\
(A,\phi)\ar[u]^{\sim}_{\eta}\ar[r]^{f|_A} & (B,\psi)\ar[u]^{\sim}_{\eta} }
\end{equation*}
commutes.
By lemma~\ref{QuasiFreeReplacements:Equivalences},
the vertical morphisms of this diagram
are weak equivalences.
Hence,
if $f|_A$ is a weak-equivalence,
then so is $\phi_f$
by the two-out-of-three axiom of model categories.
\end{proof}

\subsubsection{Examples}
Recall that the associative operad $\AOp$
is endowed with a weak-equivalence $\phi_{\kappa}: B^c(\Lambda^{-1}\AOp^{\vee})\rightarrow\AOp$,
where $\DOp = \Lambda^{-1}\AOp^{\vee}$
is a desuspension of the dual cooperad of~$\AOp$.
In this example,
the homomorphisms $f: \Lambda^{-1}\AOp^{\vee}(A)\rightarrow B$
associated to morphisms of quasi-free coalgebras
$\phi_f: (\Lambda^{-1}\AOp^{\vee}(A),\partial_{\alpha})\rightarrow(\Lambda^{-1}\AOp^{\vee}(B),\partial_{\beta})$
are identified with the usual $A_\infty$-morphisms (see~\cite{Kadeishvili}).
The results of this subsection apply to this example since the associative operad $\AOp$
and the cooperad $\DOp = \Lambda^{-1}\AOp^{\vee}$
are obviously $\Sigma_*$-cofibrant.

The Lie operad $\LOp$
is endowed with a weak-equivalence $\phi_{\kappa}: B^c(\Lambda^{-1}\COp^{\vee})\rightarrow\LOp$,
where $\DOp = \Lambda^{-1}\COp^{\vee}$
is a desuspension of the dual cooperad of the commutative operad~$\COp$.
In this example,
the homomorphisms $f: \Lambda^{-1}\COp^{\vee}(A)\rightarrow B$
associated to morphisms of quasi-free coalgebras
$\phi_f: (\Lambda^{-1}\COp^{\vee}(A),\partial_{\alpha})\rightarrow(\Lambda^{-1}\COp^{\vee}(B),\partial_{\beta})$
are identified with the usual $L_\infty$-morphisms.
But the commutative operad $\COp$
and the Lie operads $\LOp$
are not $\Sigma_*$-cofibrant in positive characteristic.
Therefore we have to replace $\COp$ and $\LOp$
by $\Sigma_*$-cofibrant operads
in order to apply the results of this subsection.

\section{Cylinder objects and homotopy morphisms}\label{CylinderHomotopyMorphisms}
\renewcommand{\thesubsubsection}{\thesubsection.\arabic{subsubsection}}

The goal of this section is to define a correspondence between left homotopies in the category of operads
and certain equivalences in the homotopy category of algebras over an operad.
Throughout the section,
we consider a quasi-free operad such that $\QOp = B^c(\DOp)$.
To obtain our result,
we introduce an explicit cylinder object in the category of operads $\Cyl\QOp$
such that a left homotopy $\tilde{\psi}: \Cyl\QOp\rightarrow\End_A$
toward an endomorphism operad $\End_A$
is equivalent to a morphism of quasi-cofree $\DOp$-coalgebras
\begin{equation*}
(\DOp(A),\partial_{\alpha^0})\xrightarrow{\phi}(\DOp(A),\partial_{\alpha^1})
\end{equation*}
which reduces to the identity on~$A$.
Then we apply proposition~\ref{QuasiFreeReplacements:HomotopyEquivalences}
to produce a chain of weak-equivalences in the category of $\POp$-algebras
\begin{equation*}
\xymatrix{ R_{\POp}(\DOp(A),\partial_{\alpha^0})\ar[d]^{\sim}\ar[r]^{\sim} & R_{\POp}(\DOp(A),\partial_{\alpha^1})\ar[d]^{\sim} \\
(A,\phi^0) & (B,\phi^1) }
\end{equation*}
from that coalgebra morphism.

The cylinder object $\Cyl\QOp$ is defined in~\S\ref{CylinderOperads}.
The correspondence between left homotopies $\tilde{\psi}: \Cyl\QOp\rightarrow\End_A$
and homotopy morphisms
is addressed in~\S\ref{HomotopyMorphisms}.

\subsection{Bar duality for operads and cylinder objects}\label{CylinderOperads}
By definition,
a cylinder object associated to $\QOp$ is an operad $\Cyl\QOp$
together with morphisms
\begin{equation*}
\xymatrix{ \QOp\ar@<+2pt>[r]^(0.3){d^0}\ar@<-2pt>[r]_(0.3){d^1} & \Cyl\QOp\ar[r]^(0.7){s^0} & \QOp }
\end{equation*}
such that the morphism $(d^0,d^1): \QOp\vee\QOp\rightarrow\Cyl\QOp$ is a cofibration,
the morphism $s^0: \Cyl\QOp\rightarrow\QOp$
is a weak-equivalence
and $s^0 d^0 = s^0 d^1 = \id$.

Recall that the operadic cobar construction $\QOp = B^c(\DOp)$
is a quasi-free operad of the form
\begin{equation*}
\QOp = (\FOp(\kk\sigma\otimes\tilde{\DOp}),\partial_{\beta}),
\end{equation*}
where $\sigma$ is a homogeneous element of degree $-1$.
The idea is to use the standard cylinder object of~$\kk\sigma$
in the category of dg-modules
(the definition of this cylinder object $\Cyl(\kk\sigma)$ is reviewed in~\S\ref{CylinderOperads:Construction})
and to form an extension of the cobar construction
\begin{equation*}
\Cyl\QOp = (\FOp(\Cyl(\kk\sigma)\otimes\tilde{\DOp}),\partial_{\beta}).
\end{equation*}
Our first task is to define an appropriate twisting derivation
\begin{equation*}
\partial_{\beta}: \FOp(\Cyl(\kk\sigma)\otimes\tilde{\DOp})\rightarrow\FOp(\Cyl(\kk\sigma)\otimes\tilde{\DOp})
\end{equation*}
in order to define this quasi-free operad $\Cyl\QOp$.
Then we apply results of~\S\ref{QuasiFreeOperads}
to check the properties of cylinder objects.

\subsubsection{The construction of the cylinder object}\label{CylinderOperads:Construction}
We have
\begin{equation*}
\Cyl(\kk\sigma) = (\kk\sigma^0\oplus\kk\sigma^1\oplus\kk\sigma^{0 1},\partial),
\end{equation*}
where $\sigma^0,\sigma^1$ are homogeneous elements of degree $-1$,
the element $\sigma^{01}$ has degree $0$
and $\partial$ is the differential such that $\partial(\sigma^{01}) = \sigma^1-\sigma^0$.
For our use (except in the proof of lemma~\ref{CylinderOperads:CylinderRequirements}),
it will be more natural to put the differential
of $\Cyl(\kk\sigma)$
in the twisting derivation of~$\Cyl\QOp$.
Therefore
we consider the graded $\kk$-module $K = \kk\sigma^0\oplus\kk\sigma^1\oplus\kk\sigma^{01}$
underlying~$\Cyl(\kk\sigma)$
and we form the free operad $\FOp(K\otimes\tilde{D})$.

The operad derivation $\partial_{\beta}: \FOp(K\otimes\tilde{D})\rightarrow\FOp(K\otimes\tilde{D})$
is associated to the homomorphism $\beta: K\otimes\tilde{D}\rightarrow\FOp(K\otimes\tilde{D})$
defined by the formulas of figure~\ref{Fig:CylinderDifferential}.
In these formulas,
the signs (like the sign of the cobar construction)
are produced by the commutation of homogeneous elements $\sigma^{\epsilon}$
with factors $\gamma_*$
when we patch the tensor products $\sigma^{\epsilon}\otimes\gamma_*$.
\begin{figure}
\begin{align*}
& \begin{aligned}
\beta\left\{\vcenter{\xymatrix@H=6pt@W=4pt@M=2pt@R=8pt@C=4pt{ i_1\ar[dr] & \cdots & i_n\ar[dl] \\
& *+<6pt>[F]{\sigma^{\epsilon}\otimes\gamma}\ar[d] & \\
& 0 & }}\right\}
& = \sum_{\substack{\tau\in\Theta_2(I)\\ \rho_{\tau}(\gamma)}}
\pm\left\{\vcenter{\xymatrix@H=6pt@W=4pt@M=2pt@R=8pt@C=4pt{ & i_*\ar[dr] & \cdots & i_*\ar[dl] & \\
i_*\ar[drr] & \cdots & *+<6pt>[F]{\sigma^{\epsilon}\otimes\gamma_*}\ar[d] & \cdots & i_*\ar[dll] \\
&& *+<6pt>[F]{\sigma^{\epsilon}\otimes\gamma_*}\ar[d] && \\
&& 0 && }}\right\},
\quad\text{for $\epsilon = 0,1$}, \\
\\
\beta\left\{\vcenter{\xymatrix@H=6pt@W=4pt@M=2pt@R=8pt@C=4pt{ i_1\ar[dr] & \cdots & i_n\ar[dl] \\
& *+<6pt>[F]{\sigma^{0 1}\otimes\gamma}\ar[d] & \\
& 0 & }}\right\}
& =
\left\{\vcenter{\xymatrix@H=6pt@W=4pt@M=2pt@R=8pt@C=4pt{ i_1\ar[dr] & \cdots & i_n\ar[dl] \\
& *+<6pt>[F]{\sigma^{1}\otimes\gamma}\ar[d] & \\
& 0 & }}\right\}
-
\left\{\vcenter{\xymatrix@H=6pt@W=4pt@M=2pt@R=8pt@C=4pt{ i_1\ar[dr] & \cdots & i_n\ar[dl] \\
& *+<6pt>[F]{\sigma^{0}\otimes\gamma}\ar[d] & \\
& 0 & }}\right\}
\end{aligned}
\\
\\
& \qquad\qquad\qquad + \sum_{\substack{\tau\in\Theta_2(I)\\ \rho_{\tau}(\gamma)}}
\pm\left\{\vcenter{\xymatrix@H=6pt@W=4pt@M=2pt@R=8pt@C=4pt{ & i_*\ar[dr] & \cdots & i_*\ar[dl] & \\
i_*\ar[drr] & \cdots & *+<6pt>[F]{\sigma^{1}\otimes\gamma_*}\ar[d] & \cdots & i_*\ar[dll] \\
&& *+<6pt>[F]{\sigma^{0 1}\otimes\gamma_*}\ar[d] && \\
&& 0 && }}\right\},
\\
\\
& \qquad\qquad\qquad  - \sum_{\substack{\tau\in\Psi_2(I)\\ \rho_{\tau}(\gamma)}}
\left\{\vcenter{\xymatrix@H=6pt@W=3pt@M=2pt@!R=1pt@!C=1pt{ &
i_*\ar[dr] & \cdots\ar@{}[d]|{\displaystyle{\cdots}} & i_*\ar[dl] &
& i_*\ar[dr] & \cdots\ar@{}[d]|{\displaystyle{\cdots}} & i_*\ar[dl] & \\
i_*\ar@/_8pt/[drrrr] &
& *+<3mm>[F]{\sigma^{0 1}\otimes\gamma_*}\ar[drr] &
& \cdots &
& *+<3mm>[F]{\sigma^{0 1}\otimes\gamma_*}\ar[dll] &
& i_*\ar@/^8pt/[dllll] \\
&&&& *+<3mm>[F]{\sigma^0\otimes\gamma_*}\ar[d] &&&& \\ &&&& 0 &&&& }}\right\}.
\end{align*}
\caption{}\label{Fig:CylinderDifferential}
\end{figure}

Note that:

\begin{claim}\label{CylinderOperads:ConstructionVerification}
The homomorphism $\beta: K\otimes\tilde{\DOp}\rightarrow\FOp(K\otimes\tilde{\DOp})$
defined by the formula of figure~\ref{Fig:CylinderDifferential}
satisfies the relation $\delta(\beta) = 0$ with respect to the internal differential of~$\tilde{\DOp}$
and $\partial_{\beta}\cdot\beta = 0$.
\end{claim}

Hence,
by proposition~\ref{QuasiFreeOperads:TwistingDerivationConstruction},
we have a well-defined quasi-free operad such that $\Cyl\QOp = (\FOp(K\otimes\tilde{\DOp}),\partial_{\beta})$.

\begin{proof}
The assertion $\delta(\beta) = 0$ is immediate since $\beta$ is essentially a combination of coproducts and identity
morphisms in $\tilde{\DOp}$.

For generating elements $c = \sigma^{\epsilon}\otimes\gamma$, $\epsilon = 0,1$,
the homomorphism $\beta$
is identified with the homomorphism which determines
the derivation of the cobar construction.
Hence we already know that $\partial_{\beta}\cdot\beta(\sigma^{\epsilon}\otimes\gamma)$
vanishes for $\epsilon = 0,1$.
Recall briefly that this assertion is a consequence of the coassociativity of the cooperad coproduct,
expressed by dual versions of the commutative diagrams~(\ref{Fig:TreeLinearAssociativity}-\ref{Fig:TreeRamifiedAssociativity})
of~\S\ref{TreeOperadStructures:TreeAssociativity}.

The identity $\partial_{\beta}\cdot\beta(\sigma^{01}\otimes\gamma) = 0$ follows from a straightforward generalization of this verification:
check that each tree of the expansion of~$\partial_{\beta}\cdot\beta(\sigma^{01}\otimes\gamma)$
occurs twice and use the coassociativity of the cooperad structure
to conclude that all terms vanish.
\end{proof}

We check now:

\begin{lemm}\label{CylinderOperads:StructureMorphisms}
The morphisms of graded $\kk$-modules
\begin{equation*}
\xymatrix{ \kk\sigma\ar@<2pt>[r]^(0.3){d^0}\ar@<-2pt>[r]_(0.3){d^1} &
\kk\sigma^0\oplus\kk\sigma^1\oplus\kk\sigma^{0 1}\ar[r]^(0.7){s^0} &
\kk\sigma }
\end{equation*}
such that
$d^0(\sigma) = \sigma^0$, $d^1(\sigma) = \sigma^1$,
$s^0(\sigma^0) = s^0(\sigma^1) = \sigma$, and $s^0(\sigma^{01}) = 0$,
induce operad morphisms
\begin{equation*}
\xymatrix{ \QOp\ar@<2pt>[r]^(0.45){d^0}\ar@<-2pt>[r]_(0.45){d^1} & \Cyl\QOp\ar[r]^(0.55){s^0} & \QOp }
\end{equation*}
such that $s^0 d^0 = s^0 d^1 = \id$.
\end{lemm}

\begin{proof}
By proposition~\ref{QuasiFreeOperads:InducedMorphisms}
we are reduced to check identities
\begin{equation*}
\beta\cdot\phi(c) = \phi\cdot\beta(c)
\end{equation*}
on generating elements
in order to prove the commutation of $\phi = d^0,d^1,s^0$
with differentials.
This verification is immediate from the definition
of the homomorphisms $\beta: \kk\sigma\otimes\tilde{\DOp}\rightarrow\FOp(\kk\sigma\otimes\tilde{\DOp})$
and $\beta: K\otimes\tilde{\DOp}\rightarrow\FOp(K\otimes\tilde{\DOp})$.
\end{proof}

\begin{lemm}\label{CylinderOperads:CylinderRequirements}
If the cooperad $\DOp$ is $\Sigma_*$-cofibrant,
then the morphism $(d^0,d^1): \QOp\vee\QOp\rightarrow\Cyl\QOp$ is a cofibration
and $s^0: \Cyl\QOp\rightarrow\QOp$
is an acyclic fibration.
\end{lemm}

\begin{proof}
In this proof,
we use that $d^0,d^1,s^0$ are identified with morphisms of quasi-free operads
induced by morphisms of $\Sigma_*$-objects:
\begin{equation*}
\xymatrix{ \kk\sigma\otimes\tilde{\DOp}\ar@<2pt>[r]^{d^0\otimes\id}\ar@<-2pt>[r]_{d^1\otimes\id} &
\Cyl(\kk\sigma)\otimes\tilde{\DOp}\ar[r]^{s^0\otimes\id} &
\kk\sigma\otimes\tilde{\DOp} }.
\end{equation*}
Since $d^0,d^1: \kk\sigma\rightarrow\Cyl(\kk\sigma)$
are acyclic cofibrations and $\tilde{\DOp}$
is $\Sigma_*$-cofibrant,
we obtain that
\begin{equation*}
d^0,d^1: \kk\sigma\otimes\tilde{\DOp}\rightarrow\Cyl(\kk\sigma)\otimes\tilde{\DOp}
\end{equation*}
are acyclic cofibrations of~$\Sigma_*$-objects
and proposition~\ref{QuasiFreeOperads:Cofibrations}
implies that the induced morphisms $d^0,d^1: \QOp\rightarrow\Cyl\QOp$ are acyclic cofibrations of operads.
Accordingly,
we conclude from the identities $s^0 d^0 = s^0 d^1 = \id$
that $s^0$ is a weak-equivalence.
Note that $s^0$
is obviously a fibration.

We also have
\begin{equation*}
\QOp\vee\QOp = (\FOp((\kk\sigma^0\oplus\kk\sigma^1)\otimes\tilde{\DOp}),\partial_{\beta})
\end{equation*}
and $(d^0,d^1): \QOp\vee\QOp\rightarrow\Cyl\QOp$
can be identified with the morphism of quasi-free operads
induced by the morphism of $\Sigma_*$-objects:
\begin{equation*}
(\kk\sigma^0\oplus\kk\sigma^1)\otimes\tilde{\DOp}\xrightarrow{(d^0,d^1)\otimes\tilde{\DOp}}\Cyl(\kk\sigma)\otimes\tilde{\DOp},
\end{equation*}
which forms a cofibration.
Accordingly,
proposition~\ref{QuasiFreeOperads:Cofibrations}
implies that the morphism $(d^0,d^1): \QOp\vee\QOp\rightarrow\Cyl\QOp$
forms a cofibration of operads.
\end{proof}

From these lemmas, we conclude:

\begin{thm}\label{CylinderOperads:Statement}
For any $\Sigma_*$-cofibrant cooperad~$\DOp$,
the quasi-free operad of~\S\ref{CylinderOperads:Construction}
\begin{equation*}
\Cyl\QOp = (\FOp(\Cyl(\kk\sigma)\otimes\tilde{\DOp}),\partial_{\beta})
\end{equation*}
defines a cylinder object associated to the cobar construction $\QOp = B^c(\DOp)$.
\qed
\end{thm}

\subsection{Cylinder objects and algebra equivalences}\label{HomotopyMorphisms}
Let $\psi^0,\psi^1: \QOp\rightarrow\End_A$
be a pair of morphisms which provide a dg-module $A$
with two $\QOp$-algebra structures.
Recall that a left homotopy between $(\psi^0,\psi^1)$
is a morphism $\tilde{\psi}: \Cyl\QOp\rightarrow\End_A$
on a cylinder object associated to~$\QOp$
such that the diagram
\begin{equation*}
\xymatrix{ \QOp\vee\QOp\ar[r]^{(\psi^0,\psi^1)}\ar[d]_{(d^0,d^1)} & \End_A \\
\Cyl\QOp\ar@{.>}[ur]_{\tilde{\psi}} & }
\end{equation*}
commutes.
In a model category $\A$
the existence of a left homotopy between morphisms $(\psi^0,\psi^1)$
does not depend on the choice of a cofibrant object.
The same assertion holds in semi-model categories
provided that $(\psi^0,\psi^1)$
have a cofibrant domain.

The first purpose of this subsection is to prove:

\begin{thm}\label{HomotopyMorphisms:OperadicHomotopyCorrespondence}
Let $\QOp = B^c(\DOp)$
be the cobar construction of a $\Sigma_*$-cofibrant cooperad $\DOp$.
Let $\POp$ be any $\Sigma_*$-cofibrant operad
together with a weak-equivalence $\phi_{\theta}: \QOp\xrightarrow{\sim}\POp$.

Let $\phi^0,\phi^1: \POp\rightarrow\End_A$
be a pair of morphisms which provide a dg-module $A$
with two $\POp$-algebra structures.
Let $\Gamma_{\POp}(A,\phi^{\epsilon}) = (\DOp(A),\partial_{\alpha^{\epsilon}})$
be the quasi-cofree $\DOp$-coalgebras
associated to the $\POp$-algebras $(A,\phi^{\epsilon})$, $\epsilon = 0,1$.

We have a bijective correspondence between left homotopies
\begin{equation*}
\xymatrix{ \QOp\vee\QOp\ar[r]\ar[d]_{(d^0,d^1)} & \POp\vee\POp\ar[r]^{(\phi^0,\phi^1)} & \End_A \\
\Cyl\QOp\ar@{.>}[urr]_{\tilde{\psi}} & },
\end{equation*}
and the morphisms of quasi-cofree coalgebras
\begin{equation*}
(\DOp(A),\partial_{\alpha^1})\xrightarrow{\phi_f}(\DOp(A),\partial_{\alpha^0})
\end{equation*}
which reduce to the identity on $A$.
\end{thm}

Recall that $\phi_f$ is determined by a homomorphism of dg-modules $f: \DOp(A)\rightarrow A$
of degree $0$.
The assumption on $\phi_f$
amounts to the identity $f|_A = \id_A$.

\begin{proof}
Let $(\psi^0,\psi^1)$ be the composites $\psi^{\epsilon} = \phi^{\epsilon}\cdot\phi_{\theta}$,
where $\epsilon = 0,1$.
We have by definition $\alpha^{\epsilon}(a) = 0$ for $a\in A$.
The relation $\tilde{\psi} d^{\epsilon} = \psi^{\epsilon}$
gives
\begin{equation*}
\alpha^{\epsilon}(c) = \psi^{\epsilon}(\gamma)(a_1,\dots,a_n) \\
= f(d^{\epsilon}(\gamma))(a_1,\dots,a_n) = \tilde{\psi}(\sigma^{\epsilon}\otimes\gamma)(a_1,\dots,a_n),
\end{equation*}
for any element $c = \gamma(a_1,\dots,a_n)\in\DOp(A)$
such that $\gamma\in\tilde{\DOp}(n)$.
The homomorphism $f: \DOp(A)\rightarrow A$
associated to $\tilde{\psi}$
is determined by $f(a) = a$ for $a\in A$
and (with the same conventions)
\begin{equation*}
f(c) = \tilde{\psi}(\sigma^{01}\otimes\gamma)(a_1,\dots,a_n),
\end{equation*}
when we assume $c = \gamma(a_1,\dots,a_n)\in\DOp(A)$, $\gamma\in\tilde{\DOp}(n)$.

Recall that $\tilde{\psi}$
is uniquely determined by the homomorphism of $\Sigma_*$-object $h = \tilde{\psi}|_{K\otimes\tilde{\DOp}}$.
As a consequence,
our relations give a bijective correspondence
between the homomorphisms of degree $0$
\begin{equation*}
\DOp(A)\xrightarrow{f} A
\end{equation*}
such that $f|_A = \id_A$
and the homomorphisms of degree $0$
\begin{equation*}
\FOp(K\otimes\tilde{\DOp})\xrightarrow{\tilde{\psi}}\End_A
\end{equation*}
commuting with composition structures
and such that $\tilde{\psi} d^{\epsilon} = \psi^{\epsilon}$
for $\epsilon = 0,1$.

Thus we are reduced to prove that $f$ satisfies the equation of proposition~\ref{CooperadCoalgebras:QuasiFreeMorphisms}
if and only if $h: K\otimes\tilde{\DOp}\rightarrow\End_A$
satisfies the equation of proposition~\ref{QuasiFreeOperads:MorphismConstruction}.
The verification of this assertion reduces to an immediate inspection
of formulas.
\end{proof}

This theorem gives as a corollary:

\begin{thm}\label{HomotopyMorphisms:OperadicHomotopyEquivalences}
Let $\QOp = B^c(\DOp)$
be the cobar construction of a $\Sigma_*$-cofibrant cooperad $\DOp$.
Let $\POp$ be any $\Sigma_*$-cofibrant operad
together with a weak-equivalence $\phi_{\theta}: \QOp\xrightarrow{\sim}\POp$.

Let $\phi^0,\phi^1: \POp\rightarrow\End_A$
be morphisms which provide a dg-module $A$
with $\POp$-algebra structures.
The existence of a left homotopy
\begin{equation*}
\xymatrix{ \QOp\vee\QOp\ar[r]\ar[d]_{(d^0,d^1)} & \POp\vee\POp\ar[r]^{(\phi^0,\phi^1)} & \End_A \\
\Cyl\QOp\ar@{.>}[urr]_{\tilde{\psi}} & },
\end{equation*}
implies the existence of a fill-in weak-equivalence
in the diagram:
\begin{equation*}
\xymatrix{ R_{\POp}(\DOp(A),\partial_{\alpha^1})\ar[d]^{\sim}\ar@{.>}[r]^{\sim}_{\phi_f} & R_{\POp}(\DOp(A),\partial_{\alpha^0})\ar[d]^{\sim} \\
(A,\phi^1) & (A,\phi^0) }.
\end{equation*}
\end{thm}

\begin{proof}
Apply proposition~\ref{QuasiFreeReplacements:HomotopyEquivalences}
to the morphism $\phi_f$ yielded by theorem~\ref{HomotopyMorphisms:OperadicHomotopyCorrespondence}.
\end{proof}

\subsubsection{Concluding remark: generalizations to properads}
Recall that a prop is a structure which collects operations $p: A^{\otimes m}\rightarrow A^{\otimes n}$
with a variable number of inputs $m$ and outputs $n$.
The free prop is defined by tensors arranged on graphs
and the structure of a prop $\POp$
amounts to composition products $\lambda_{\tau}: \tau(\tilde{\POp})\rightarrow\tilde{\POp}(I)$,
where $\tau$ ranges over arbitrary $I$-graphs.
The notion of a properad introduced in~\cite{ValletteThesis,Vallette}
amounts to a prop whose structure is determined by composition products
over connected graphs.

The bar duality of operads is extended to properads
in~\cite{ValletteThesis,Vallette},
but we have no duality construction on the category of algebras associated to a properad
such that the analogue of theorem~\ref{QuasiFreeReplacements:Statement}
holds.

However,
the definition of the cylinder object $\Cyl B^c(\DOp)$
can be extended to properads
and we still have an equivalence between homotopies in the category of properads
and certain morphisms of $\DOp$-costructures (which are no more coalgebras over $\DOp$).
Moreover,
according to~\cite{FressePropHomotopy},
a left homotopy between morphisms $\tilde{\phi}: \Cyl B^c(\DOp)\rightarrow\End_A$
is still associated to a morphism
\begin{equation*}
(A,\phi^0)\xrightarrow{\sim}\cdot\xleftarrow{\sim}\cdot\xrightarrow{\sim}(A,\phi^1)
\end{equation*}
in the homotopy category of $B^c(\DOp)$-algebras.

Therefore,
the construction of this subsection could be generalized to give a new model
of equivalences in the homotopy category of algebras
over a properad.

\end{document}